\documentclass[reqno]{amsart}
\usepackage{graphicx} 
\usepackage{subfiles}
\usepackage{subcaption}
\usepackage[shortlabels]{enumitem}
\usepackage[T1]{fontenc}
\usepackage{tikz}
\usetikzlibrary{arrows.meta,calc,decorations.pathreplacing,calligraphy}
\usepackage{dsfont}
\usepackage{xfrac}
\usepackage{microtype}
\newtheorem{theorem}{Theorem}[section]
\newtheorem{lemma}{Lemma}[section]
\newtheorem{proposition}{Proposition}[section]
\newtheorem{corollary}{Corollary}[section]

\theoremstyle{definition}
\newtheorem{definition}{Definition}[section]
\theoremstyle{remark}
\newtheorem{remark}{Remark}[section]
\usepackage{physics}
\usepackage{mathtools}
\usepackage{amsmath, amssymb}
\calclayout
\def\undertilde#1{\mathord{\vtop{\ialign{##\crcr
$\hfil\displaystyle{#1}\hfil$\crcr\noalign{\kern1.5pt\nointerlineskip}
$\hfil\widetilde{}\hfil$\crcr\noalign{\kern1.5pt}}}}}
\setlist[enumerate,2]{label=\arabic*)} \newcommand*\diff{\mathop{}\!\mathrm{d}}
\newcommand{\R}{\mathbb{R}}
\newcommand{\eps}{\epsilon}

\newcommand{\T}{\mathbb{T}}
\newcommand{\loc}{\text{\textup{loc}}}

\newcommand{\ds}{\displaystyle}

\newcommand{\Sweak}{\mathcal{S}_{\text{weak}}^{d,\eps}}
\newcommand{\aplim}{\textnormal{ap }\lim}
\usepackage{bbm}
\renewcommand{\norm}[1]{\left\lVert#1\right\rVert}
\renewcommand{\grad}{\nabla}
\allowdisplaybreaks
\usepackage[margin=1in]{geometry}
\usepackage[colorlinks,linkcolor = blue,citecolor = blue]{hyperref}
\numberwithin{equation}{section}
\usepackage{cleveref}

\newcommand{\beq}{\begin{equation}}
\newcommand{\eeq}{\end{equation}}
\usepackage{comment}
\crefformat{footnote}{#2\footnotemark[#1]#3}

\title[The Unique limit of the Glimm--Lax construction]{The Unique Limit of the Glimm--Lax Construction for Sobolev data and obstructions to 1-d convex integration}

\author[Jeffrey Cheng]{Jeffrey Cheng}
\address[Jeffrey Cheng]{\newline Department of Mathematics \newline The University of Texas at Austin, \newline 2515
Speedway, Austin, TX 78712 \newline USA}
\email{jeffrey.cheng@utexas.edu}

\author[Cooper Faile]{Cooper Faile}
\address[Cooper Faile]{\newline Department of Mathematics \newline The University of Texas at Austin, \newline 2515
Speedway, Austin, TX 78712 \newline USA}
\email{jcfaile@utexas.edu}

\author[Sam G. Krupa]{Sam G.  Krupa}
\address[Sam G. Krupa]{\newline Département de mathématiques et applications \newline École normale supérieure, Université PSL, CNRS\newline 45 rue d'Ulm - F 75230 PARIS cedex 05 \newline France}
\email{sam.krupa@ens.fr}

\date{\today}

\thanks{\textbf{Acknowledgment.} 
The authors thank Alexis Vasseur and Matthew Novack for the fruitful discussions during this work. The authors utilized Google Gemini as well as refine.ink for editing. The first and second authors are partially supported by National Science Foundation grants DMS-1840314 \& DMS-2306852 and a joint NSF-EPSRC grant DMS-EPSRC-2219434. The work of the third author is funded by the European Union through the project ``Quantitative Stability and Regularity of Large Data for Conservation Laws.'' Views and opinions expressed are however those of the author(s) only and do not necessarily reflect those of the European Union or European Research Executive Agency (REA). Neither the European Union nor the granting authority can be held responsible for them. }

\begin{document}
\emergencystretch 3em

\begin{abstract}
We consider a genuinely nonlinear $1$-d system of hyperbolic conservation laws with two unknowns. A famous construction of Glimm \& Lax shows that global-in-time ``Glimm--Lax'' weak entropy solutions exist in this setting for any initial data with small $L^\infty$ norm [Mem. Amer. Math. Soc. (1970), no. 101].
Recent work in the $L^1$-stability theory by Bressan, Marconi \& Vaidya has given the first partial uniqueness and stability results for these solutions [Arch. Ration. Mech. Anal. (2025), vol. 249].
In this paper, we build on these results by combining them with recent advances in the $L^2$-theory.
We show that solutions with initial data in the Sobolev space $H^s$ for $s>0$ are unique in the full class of Glimm--Lax solutions that decay in total variation at a rate of $1/t$.
As a secondary result, our techniques are also used to show the recent non-uniqueness result of Chen, Vasseur \& Yu for continuous solutions [preprint (2024)] cannot extend to $C^\alpha$ solutions for $\alpha > 1/2$, alongside some appropriate fractional Sobolev spaces $W^{s,p}$.
An auxiliary result of independent interest is the development of a weighted relative entropy contraction for perturbations of rarefaction waves.
\end{abstract}

\maketitle 

\tableofcontents

\section{Introduction}
In this paper we consider $1$-d hyperbolic $2 \times 2$ systems of conservation laws
\begin{equation} \label{cl}
     \left\{\begin{aligned}
        u_t+[f(u)]_x&=0 & (t,x) &\in \R^+ \times \R, \\
        u(0,x) &=u^0(x) & x &\in \R, 
    \end{aligned}\right.
\end{equation}
where $u=(u_1, u_2) \in \mathcal{G}_0 \subset \R^2$. The state space $\mathcal{G}_0$ is assumed to be convex and allowed to be unbounded, and we denote by $\mathcal{G}$ its interior. For regularity, we take $f=(f_1, f_2) \in C(\mathcal{G}_0) \cap C^4(\mathcal{G})$. Fix a state $d \in \mathcal{G}$. In this context, it is classically known that if both characteristic fields are either linearly degenerate or genuinely nonlinear, there exists a Lipschitz continuous semigroup $\mathcal{S}$, called the ``Standard Riemann Semigroup'', of small-BV solutions around $d$ obtained as the limit of front-tracking approximations, the Glimm scheme, or vanishing viscosity \cite{MR1367356}, \cite{MR1337114}, \cite{MR2150387}, \cite{MR3443431}  (see also \cite{MR1816648}, \cite{MR194770} for existence theory for \eqref{cl} with small-BV initial data). 
When there exists a strictly convex function $\eta$ and a function $q$ verifying 
\begin{equation}\label{eq:entropy-def}
\nabla q=\grad \eta f',
\end{equation}
then we say the system~\eqref{cl} is endowed with an entropy/entropy flux pair $(\eta,q)$. 
If a solution $u$ to the system~\eqref{cl} also satisfies the additional differential inequality in the sense of distributions,
\begin{align}\label{entropic}
[\eta(u)]_t+[q(u)]_x \leq 0,
\end{align}
for every entropy-entropy flux pair $(\eta, q)$, we say $u$ is a \emph{weak entropy solution} to the system~\eqref{cl}. More specifically, such a $u$ verifies
\begin{align*}
\int_0^\infty \int_{-\infty}^\infty [\phi_t(t,x)\eta(u(t,x))+\phi_x(t,x)q(u(t,x))] \diff x \diff t+\int_{-\infty}^\infty \phi(0,x)\eta(u^0(x)) \diff x \geq 0,
\end{align*}
for every $\phi \in C_0^\infty([0,\infty) \times \R)$ with $\phi \geq 0$ and every entropy-entropy flux pair $(\eta,q)$.
A well-developed $L^1$-stability theory for BV solutions shows that if a weak entropy solution exists with finite (but possibly very large) total variation, then it is unique and depends continuously on the initial data (see e.g. \cite{MR4690554}, \cite{MR4661213}, \cite{MR1375345}, \cite{MR1828320}, \cite{MR2091511}, \cite{MR1883740}).

\subsection{The Glimm--Lax class}
However, there is a class of $2 \times 2$ conservation laws for which  a theory of small $L^\infty$, non-BV solutions has developed. We introduce the following condition on the flux $f$:
\vspace{10pt}
\begin{quotation}
    \textbf{(GL)} There exists $r > 0$ such that $f\colon B_r(d) \to \R^2$ is $C^4$ with $f'(d)$ strictly hyperbolic and both characteristic fields genuinely nonlinear.
\end{quotation}
\vspace{10pt}
\par Then, a famous result of Glimm and Lax shows that there exists $\epsilon_1 > 0$ such that if $||u^0-d||_{L^\infty(\R)} \leq \eps_1$, then \eqref{cl} admits a global-in-time weak entropy solution \cite{MR265767} (cf. \cite{MR2737438}). These solutions immediately enter $\textrm{BV}_{\loc}$ and remain small in $L^\infty$. Indeed, let us define
\begin{align*}
TV(u;L):=\sup_{a \in \R}||u||_{BV([a,a+L])}.
\end{align*}
Then, the Glimm--Lax solutions satisfy
\begin{align}
 ||u(t,\cdot) - d||_{L^\infty(\R)} &\leq C_0\sqrt{||u^0-d||_{L^\infty(\R)}}, \label{smalllinfty} \\
TV(u;L) &\leq C_0 \left(\sqrt{||u^0-d||_{L^\infty(\R)}}+\frac{L}{t}\right), \label{tv1/t}
\end{align}
for a constant $C_0 > 0$ depending on $f$ and $d$. Note that there is no restriction on the total variation of $u^0$, which may be infinite. 
We say that these solutions belong to the \emph{Glimm--Lax class}:
\begin{align*}
GL(\eps_1):=\left\{\vphantom{\sum}u\,\middle|\,\text{$u$ is a weak entropy solution to \eqref{cl} verifying \eqref{smalllinfty}, \eqref{tv1/t}, $||u^0-d||_{L^\infty(\R)} \leq \eps_1$}\right\}.
\end{align*}
A major question is then the following:
\newline 
\newline 
\underline{\textbf{Question 1:}} For $2 \times 2$ systems of conservation laws verifying \textbf{(GL)}, is it possible to have a well-posedness theory that includes not just BV solutions, but also solutions $u \in GL(\eps_1)$?
\newline 

\par In particular, establishing the uniqueness and continuous dependence of solutions in $GL(\eps_1)$ has remained an open problem since the original memoir of Glimm and Lax \cite{MR4855159}. In recent work, Bressan, Marconi \& Vaidya showed the uniqueness of solutions to \eqref{cl} for initial data in a class $\mathcal{D}_\alpha$ (see \Cref{eq:D-alpha}) among solutions additionally verifying an algebraic $L^1$ convergence estimate (see \Cref{prop:bmv-main}). They also show unconditional well-posedness for a subclass $(\tilde{\bold{P}}_\alpha) \subset \mathcal{D}_\alpha$ \cite[Theorem 1.2]{2025arXiv250500420B}. Functions in $(\tilde{\bold{P}}_\alpha)$ may have infinite total variation, but it must be concentrated on finitely many open intervals with small total length. In this paper, we aim to resolve {Question 1} for a large class of initial data. In particular, we will be able to show well-posedness for initial data in any fractional Sobolev space, which includes functions that may have infinite total variation on any open interval.

\subsection{Main results}
Our main result is the following:
\begin{theorem}\label{thm:main-gl}
Assume that the system \eqref{cl} verifies \textbf{(GL)}. Fix $p \in [1,\infty]$ and $s > 0$. Then, there exists $\eps_1 > 0$ such that the following holds. Let $u^0 \in W^{s,p}_{\loc}(\R)$ with $\norm{u^0-d}_{L^\infty(\R)} \leq \eps_1$. Then, the solution to \eqref{cl} in $GL(\eps_1)$ with initial data $u^0$ is unique. More specifically, fix $C, R, T >0$. Then, the family of solutions in $GL(\eps_1)$ with initial data in $ W^{s,p}_{\loc}(\R)$ yields a H\"older continuous semigroup, i.e. there exists $K>0$ such that the following estimate holds:
\begin{align*}
||u_1(t)-u_2(t)||_{L^1((-R,R))} \leq K ||u_1^0-u_2^0||_{L^1((-R -ct, R + ct))}^\theta, \indent \forall \  t \in [0,T],
\end{align*}
where $u_1^0,u_2^0$ are any two initial data satisfying $\norm{u^0_i}_{W^{s,p}{((-R-cT, R+cT))}} + \norm{u^0_i}_{L^\infty((-R-cT, R+cT))} \leq C$ for $i=1,2$ and the constant $c > 0$ is the speed of information (see \Cref{rem:soi}). 
Additionally, $\theta < 1$ may be taken arbitrarily close to $1$ by choosing $\eps_1 > 0$ sufficiently small.
\end{theorem}
In particular, \Cref{thm:main-gl} says that the Glimm--Lax construction \cite{MR265767} is well-defined and has a unique limit within $GL(\epsilon_1)$ as long as the initial data is in $H^s(\R)$ for any $s > 0$ and $\epsilon_1$ is sufficiently small. 
\par \Cref{thm:main-gl} is best contextualized in relation to recent works in the $L^1$-stability theory \cite{2025arXiv250500420B}, \cite{MR4841631}, which define classes similar to the following. For any $0 < \alpha < 1$, define 
\begin{align}
\mathcal{D}_\alpha:=\Big\{u^0 \in L^\infty(\R) \, \Big| \, &\text{the Cauchy problem \eqref{cl} has a weak entropy solution} \label{eq:D-alpha} \notag \\
&\text{$u \in GL(\eps_1)$ such that $\sup_{t \in (0,T)}t^{-\alpha}\norm{u(t)-u^0}_{L^1((a,b))} < \infty $ for all $(a,b) \subset \R$}\Big\}, \\
\tilde{\mathcal{D}}_\alpha:=\Big\{u^0 \in L^\infty(\R)\,\Big|\, &\text{the Cauchy problem \eqref{cl} has a weak entropy solution} \notag \\
&\text{$u\in GL(\eps_1)$ such that $\sup_{t \in (0,T)}t^{1-\alpha}\norm{u(t)}_{\textrm{BV}((a,b))} <\infty$ for all $(a,b) \subset \R$}\Big\}. \label{eq:tilde-D-alpha}
\end{align}
The uniqueness of solutions to \eqref{cl} with initial data in $ \mathcal{D}_\alpha$ among solutions additionally verifying the decay rate
\begin{equation} \label{eq:Da-decay}
\sup_{t \in (0,T)} t^{-\alpha} \norm{u(t,\cdot) - u^0}_{L^1((-R,R))} < \infty, \quad \text{for any $R > 0$,}
\end{equation}
was shown in recent work of Bressan, Marconi \& Vaidya \cite[Theorem 1.1]{2025arXiv250500420B}. A corollary of our techniques from the proof of \Cref{thm:main-gl} is that the space $\mathcal{D}_\alpha$ actually contains the Sobolev spaces of fractional order. An analogous statement was proven in the scalar setting in \cite[Proposition 4.2]{MR4841631}, but such an inclusion remained an open problem for solutions to systems satisfying \textbf{(GL)}.

\begin{corollary}\label{thm:main-inc}
For any $\alpha, s \in (0,1)$, there exists $\epsilon_1 > 0$ such that the following inclusions hold
\begin{alignat}{4}
    \tilde{\mathcal{D}}_\alpha \subset&\,\mathcal{D}_\alpha,  &&\label{item:1}\\
     \mathcal{D}_{\alpha}\subset&\, H^{r}_{\loc}(\R; B_{\eps_1}(d)),\quad &&\forall \ r < \frac{\alpha}{2(1+\alpha)}, \label{item:2} \\
    H^{s}_{\loc}(\R; B_{\eps_1}(d)) \subset&\, \mathcal{D}_\omega, &&\forall \ 0 <\omega \leq \frac{s}{2}. \label{item:3}
\end{alignat}
\end{corollary}
Some comments are in order:
\begin{itemize}
    \item The inclusion \eqref{item:1} is known already and was proven in \cite[Equation (1.11)]{2025arXiv250500420B}.
    \item \cite[Theorem 1.2]{2025arXiv250500420B} shows that there exists a subclass $(\tilde{\bold{P}}_\alpha) \subset \mathcal{D}_\alpha$. \Cref{item:3} in \Cref{thm:main-inc} shows that the uniqueness domains $D_\alpha$ are actually extremely large classes that contain fractional Sobolev spaces of all orders. In particular, they contain initial data of infinite total variation that is highly oscillatory everywhere. An example of such data is a path of a fractional Brownian motion, which lies in $W^{s,\infty}_{\loc}$ almost surely for $s$ dependent on the Hurst index.
    \item By combining ~\eqref{item:3} and \cite[Theorem 1.1]{2025arXiv250500420B}, one can show that \eqref{cl} is well-posed in the class of solutions 
    \begin{align*}
    \Big\{u \in GL(\eps_1) \, \Big|\, u \text{ verifies \eqref{eq:Da-decay}}\Big\},
    \end{align*}
    for initial data in $W^{s,p}_{\loc}$. 
    However, this can be upgraded by using \Cref{thm:main-gl} to show well-posedness in the full Glimm--Lax class $GL(\eps_1)$ for initial data in $W^{s,p}_{\loc}$. The key step is to show that \emph{every function} $u \in GL(\eps_1)$ with initial data $u^0 \in W^{s,p}_{\loc}$ verifies the algebraic $L^1$ convergence rate~\eqref{eq:Da-decay}. In fact, \Cref{thm:main-gl}, \eqref{item:2} and \eqref{item:3} show that solutions in $GL(\eps_1)$ which verify \eqref{eq:Da-decay} are \emph{completely characterized} by having Sobolev initial data.
\end{itemize}

\subsection{On non-uniqueness for continuous solutions}
As another application of our methods, we also provide a theorem complementing recent work on non-uniqueness for continuous solutions to $2 \times 2$ systems. For a class of  systems verifying \textbf{(GL)} and additionally verifying a so-called ``twisted condition,'' non-unique continuous periodic solutions to \eqref{cl} were constructed by Chen, Vasseur \& Yu \cite{cvy}. This construction gives solutions with arbitrarily small $L^\infty$ norm, low regularity, and it is unknown if they satisfy the entropy inequality~\eqref{entropic}. It is of interest to know whether this same phenomenon can hold for solutions satisfying an entropy inequality, or possessing a higher degree of regularity. Our next theorem gives an obstruction to extending these results to have more regularity.
\begin{theorem}\label{thm:main-obs}
Let $\alpha > \frac{1}{2}$. Then, there exists $\eps_1 > 0$ such that the following is true. For any $T>0$, suppose $u \in L^\infty(0,T;C^\alpha(\R;B_{\eps_1}(0)))$ is a periodic weak solution to \eqref{cl} with initial data $u^0 \in C^\alpha(\R;B_{\eps_1}(0))$. Then, $u$ coincides with the unique Glimm--Lax solution in $GL(\eps_1)$ with initial data $u^0$.
\end{theorem} 
\begin{remark}
The same techniques used in the proof of \Cref{thm:main-obs} can be used to prove a similar theorem about non-periodic solutions (see \Cref{prop:nonperiodic}). We chose to write \Cref{thm:main-obs} with periodic functions to more clearly compare the result to \cite[Theorem 1.1]{cvy}.
\end{remark}
\begin{remark}
We remark here that solutions with this level of regularity automatically obey the entropy inequality \eqref{entropic} for any (not necessarily convex) entropy $\eta$ (cf. \Cref{prop:entropy-conservation}).
\end{remark}
\Cref{thm:main-obs} shows that, as soon as one has a periodic solution in $C^\alpha_\loc (\R;B_{\eps_1}(0))$ for any $\alpha > \frac{1}{2}$, it must coincide with the unique Glimm--Lax solution. Recall that the Glimm--Lax solution is unique by \Cref{thm:main-gl}. Thus, the non-uniqueness result of \cite{cvy} cannot be strengthened to display non-uniqueness for this degree of regularity. Using the same techniques, we can also prove an analogous theorem for periodic solutions $u \in L^\infty(0,T;W^{s,p}_{\loc}(\R;B_{\eps_1}(0)))$. We are able to consider pairs $s,p$ such that the spaces $W^{s,p}$ do not embed into $C^\alpha$ for $\alpha > \frac{1}{2}$.
\begin{theorem} \label{thm:sob-ext-obs}
    Fix $p \in (1,\infty)$ and let 
    \beq \label{eq:sob-s-con} s > \begin{cases}\frac{p}{2p-1} & p \geq 2, \\ \frac{2}{p+1} & p < 2.\end{cases}\eeq Then, there exists $\eps_1 > 0$ such that the following is true. Suppose $u \in L^\infty(0,T;W^{s,p}_{\loc}(\R;B_{\eps_1}(0)))$ is a weak solution to \eqref{cl} with initial data $u^0 \in W^{s,p}_\loc(\R;B_{\eps_1}(0))$. Then, $u$ is the unique Glimm--Lax solution in $GL(\eps_1)$ with initial data $u^0$.
\end{theorem}
\par In the multi-d setting there are are several results concerning the uniqueness and non-uniqueness of solutions to isentropic Euler; we mention two here. It is known that solutions to the multi-d isentropic Euler system lying in $B_p^{\alpha, \infty}$ for $\alpha > \frac{1}{2}$ and suitable $p$ whose symmetric velocity gradient further verifies a one-sided Lipschitz bound enjoy a weak-strong uniqueness property \cite{MR4000838}. In regards to non-uniqueness, Giri \& Kwon have constructed infinitely many weak entropy solutions to the 3-dimensional isentropic Euler system in $C^\alpha$ for $0 < \alpha < \frac{1}{7}$.

\subsection{Stability with respect to non-BV perturbations}
There is a very elegant classical $L^1$ stability theory for small-BV solutions that holds in far more generality than for the class $\textbf{(GL)}$ \cite{MR1686652}, \cite{MR1723032}. In these results, the perturbations of the small-BV solution must still be small-BV. The $L^2$-stability program regarding stability with respect to non-BV perturbations was introduced by Vasseur \cite{MR2807139} as an extension of the weak-strong stability principle of Dafermos and DiPerna \cite{MR546634}, \cite{MR523630}. The study of non-BV perturbations is significant due to the lack of BV bounds when studying asymptotic limits to conservation laws. For example, the stability of small-BV solutions in the inviscid limit of artificial viscosities was shown in a paper of Bianchini \& Bressan \cite{MR2150387}. However, stability in the limit from Navier--Stokes to Euler remained open precisely because BV bounds do not appear to be available uniformly with respect to the viscosity parameter for physical viscosities. In recent work, Chen, Kang \& Vasseur showed that small-BV solutions are stable in the physical inviscid limit from Navier--Stokes to Euler \cite{nsvs}. Their result is achieved by circumventing the need for uniform BV bounds with respect to the viscosity parameter. We are able to obtain \Cref{thm:main-gl} for precisely the same reason. Solutions in $GL(\eps_1)$ do not have uniform BV bounds in the limit as time goes to zero. In this way, there is a striking analogy between the vanishing viscosity limit from Navier--Stokes to Euler and the flow backwards in time along a Glimm--Lax solution.
\par A similar result to \Cref{thm:main-inc} in the scalar setting \cite[Proposition 4.2]{MR4841631} hinges crucially on the Kru\v{z}kov theory of $L^1$ contraction which applies to solutions without uniform BV estimates in time. In the proof of \Cref{thm:main-inc}, we are able to replace the use of the $L^1$-contraction with an $L^2$ estimate for non-BV solutions. This showcases that the $L^2$ theory is a Kru\v{z}kov-like theory for systems which is strong enough to give stability estimates involving non-BV solutions (see Section \ref{relent} for further discussion of this analogy). 

\subsection{Advancements in the $L^2$-stability theory}
In order to obtain \Cref{thm:main-gl} and \Cref{thm:main-obs}, we will use a quantitative ``weak--BV'' stability result recently developed in \cite{cfk} that represents the current state-of-the-art in the $L^2$-stability program. We use the $L^2$-stability result \`{a} la \cite{cfk} to obtain short time estimates, then use the $L^1$-stability theory developed in \cite{2025arXiv250500420B} later in time. We refer to \Cref{sec:mainproof} for the main idea of the argument. The major obstruction is that Theorem 1 in \cite{cfk} applies only to a special class of $2 \times 2$ conservation laws. In particular, Assumptions 1 (d) in \cite{cfk} is the critical assumption that is not verified for the full class of fluxes verifying \textbf{(GL)}, even locally. (Note that this is the same as {Hypotheses 1} \ref{hyp-d} in \Cref{sec:localstructure}.) For an example of a flux satisfying \textbf{(GL)} and not satisfying this assumption see \Cref{appendix-hypd}. One of the major contributions of the present work is to extend Theorem 1 in \cite{cfk} to apply to fluxes merely verifying the condition \textbf{(GL)}. Our main theorem in this regard, which is crucially used in the proof of \Cref{thm:main-gl}, is the following:
\begin{theorem}\label{thm:cor}
Assume that the system \eqref{cl} verifies \textbf{(GL)}. There exists $\eps_1,\delta_0 > 0$ such that the following is true. Fix $R,T > 0$.  Consider initial data $v^0$ such that $\norm{v^0}_{BV(\R)}+\norm{v^0-d}_{L^\infty(\R)} \leq \delta_0$ and let $\mathcal{S}_tv^0$ be the Standard Riemann Semigroup solution. Then there exists a constant $K = K(d, f)> 0$ such that any solution $u \in GL(\eps_1)$ with initial data $u^0$ satisfies the stability estimate,
\begin{align*}
||u(t, \cdot)-\mathcal{S}_tv^0||_{L^1((-R,R))} \leq K\sqrt{R + T}\norm{u^0-v^0}_{L^2((-R-ct, R+ct))},
\end{align*}
where the constant $c > 0$ is the speed of information (see \Cref{rem:soi}). Further, by interpolation we obtain
\begin{align*}
\norm{u(t, \cdot)-\mathcal{S}_tv^0}_{L^2((-R,R))} \leq K\sqrt[4]{R + T}\sqrt{\norm{u^0-v^0}_{L^2((-R-ct, R+ct))}}.
\end{align*}
\end{theorem}
\begin{remark}
Our main proposition---\Cref{prop:main}---will be slightly more general than \Cref{thm:cor}. It is comparable to Theorem 1 in \cite{cfk}. The major difference is that \Cref{prop:main} applies to a wider class of systems than Theorem 1 in \cite{cfk}. The main advancement which lets us widen the class of systems we may handle is the development of a relative entropy contraction for perturbations of rarefaction waves. 
\end{remark}

\subsection{Open problems}\label{openprobs}
Our work raises some important questions for future study. We list a few of them here.
\newline
\newline 
\textbf{Open Problem \#1:} Is the solution $u$ in $GL(\eps_1)$ with initial data $u^0$ unique if the initial data $u^0$ is not in any fractional Sobolev space?
\newline
\newline
\textbf{Open Problem \#2:} Theorem 1.1 in \cite{cvy} shows non-uniqueness for continuous solutions to \eqref{cl} with no derivatives and no entropy condition. On the other hand, \Cref{thm:main-obs} shows uniqueness for continuous solutions to \eqref{cl} with at least half a derivative. It is open to obtain either uniqueness or non-uniqueness for continuous solutions with intermediate levels of regularity and entropy.
\newline 

\subsection{Outline of the paper}
The paper is structured as follows. Firstly, in \Cref{sec:mainproof}, we assume \Cref{thm:cor} and give the proofs of \Cref{thm:main-gl}, \Cref{thm:main-inc}, \Cref{thm:main-obs}, and \Cref{thm:sob-ext-obs}. In \Cref{sec:localstructure}, we exposit the terminology and structure needed to use the $a$-contraction method to prove \Cref{thm:cor}. We state \Cref{prop:main}, of which \Cref{thm:cor} is a simple corollary. The remainder of the paper will be devoted to proving \Cref{prop:main}. In Section \ref{relent}, we exposit the $a$-contraction method and obtain the desired $L^2$-estimates for shock and rarefaction waves. The relative entropy estimates for shocks will be quoted from \cite{cfk} without proof, while the relative entropy estimates for rarefactions are novel to this work. In \Cref{sec:ft}, we exposit the front tracking method we use (from \cite{MR1367356}, \cite{cfk}), and some key estimates therein. In Section \ref{a}, we construct the weight used to measure the $L^2$-stability. Finally, in Section \ref{thmpf}, we prove \Cref{prop:main}.

\subsection{Notation}
Throughout the paper, we will abide by the following conventions.
\begin{itemize}
    \item We write $g'$ as the derivative of a map $g$ which is either vector valued, matrix valued, or has scalar domain, and $\nabla g$ when $g$ is scalar valued. For higher-order derivatives, we write repeated primes (e.g. $g''$) or write  $\nabla^k g$ (e.g. successive derivatives of a scalar valued function $g$ are written $\nabla^k g$ rather than $(\nabla g)'''$).
    \item The constant $K$ denotes a universal constant which may vary from line to line. It is allowed to depend on the parameters $R$ and $T$ which determine the space-time cone of information, the flux $f$, and the point $d$ we are localizing around.
\end{itemize}

\section{The main ideas and proofs of \Cref{thm:main-gl}, \Cref{thm:main-inc}, \Cref{thm:main-obs}, and \Cref{thm:sob-ext-obs}}\label{sec:mainproof}

\subsection{The semigroup generated by Glimm-Lax solutions at positive times}
For $\tau > 0$, consider the positively invariant domains
$$ \mathcal{D}^{(\tau)} = \left\{ u(s,\cdot)\,\middle|\,s \geq \tau,\  u \in GL(\epsilon_1)\right \}.$$
Recall the existence of at least one solution $u \in GL(\eps_1)$ has been proven for any initial data $u^0$ verifying $||u^0-d||_{L^\infty(\R)} \leq \eps_1$ \cite{MR2737438}, \cite{MR265767}. In \cite{2025arXiv250500420B}, the following Standard Riemann Semigroup on $\mathcal{D}^{(\tau)}$ is shown to be well-defined and have the following properties.
\begin{proposition}[{\cite[contained in the proof of Theorem 1.1]{2025arXiv250500420B}}]\label{prop:L1-bressan}
    For any $\beta > 0$, there exists $\epsilon_1 > 0$ such that the following holds: \newline There exists a function $\mathcal{S}:\R^+ \times \mathcal{D}^{(\tau)} \to \mathcal{D}^{(\tau)}$ verifying the following properties.
    \begin{enumerate}
        \item Given $u(s,\cdot) \in \mathcal{D}^{(\tau)}, \mathcal{S}_tu(s,\cdot)$ is obtained as the unique limit of front tracking approximations at time $t$ with initial data $u(s,\cdot)$.
        \item For any $u(s,\cdot) \in \mathcal{D}^{(\tau)}$ and any $t \geq 0$, $\mathcal{S}_t(u(s,\cdot))=u(s+t,\cdot)$. \label{enum2}
        \item For any
    $\bar u, \bar v \in \mathcal{D}^{(\tau)}$ we have the Lipschitz estimate
    \beq
        \norm{\mathcal{S}_t \bar u - \mathcal{S}_t \bar v}_{L^1((-R,R))} \leq L(\tau)\norm{\bar u - \bar v}_{L^1((-R- ct ,R + c t ))}, \quad \text{for} \quad t \geq 0,
    \eeq
    where $L(\tau) \leq C \tau^{-\beta}$ for some $C > 0$ and the constant $c > 0$ is the speed of information (see \Cref{rem:finitepropspeed}). 
    \end{enumerate}
\end{proposition}
In particular, \eqref{enum2} shows that for any $\tau > 0$ and $u \in GL(\eps_1)$, there is a unique continuation of $u(\tau, \cdot)$ in $D^{(\tau)}$ given as the limit of front-tracking approximations. Thus, if non-uniqueness is to occur, it must happen instantaneously at $t=0$. 
\par Equipped with \Cref{prop:L1-bressan} one can then deduce uniqueness of Glimm--Lax solutions after supposing an additional a priori decay estimate.
For instance, if $u,v \in GL(\eps_1)$ have initial data $u^0$ one can estimate
$$ \norm{u(t) - v(t)}_{L^1} \leq L(\tau) \norm{u(\tau) - v(\tau)}_{L^1} \leq L(\tau) \left(\norm{u(\tau) - u^0}_{L^1} + \norm{v(\tau) - u^0}_{L^1}\right), \indent \text{ for $t \geq \tau$}.$$
From this, if $u(\tau) \to u^0$ and $v(\tau) \to v^0$ converge in $L^1$ sufficiently quickly to overwhelm the blow up of $L(\tau)$ then one can establish a uniqueness result. 
Following this line of reasoning, Bressan, Marconi \& Vaidya proved the following conditional uniqueness and stability result.

\begin{proposition}[{\cite[Theorem 1.1]{2025arXiv250500420B}}] \label{prop:bmv-main}
    Assume that the system \eqref{cl} verifies \textbf{(GL)}. Fix $0 < \alpha < 1$ and $M,R,T>0$. Let $c>0$ be the speed of information (see \Cref{rem:soi}). Then, there exists $\eps_1 > 0$ such that the following holds. Let $u \in GL(\epsilon_1)$ have initial data $u^0$, with the additional fast convergence rate
    \beq \label{eq:bressan-rate} \sup_{t \in (0,T)} t^{-\alpha} \norm{u(t,\cdot) - u^0}_{L^1((-R-c(T-t),R+c(T-t)))} \leq M. \eeq 
    Then, $u$ is the unique solution in $GL(\eps_1)$ with initial data $u^0$ satisfying~\eqref{eq:bressan-rate}. Furthermore, the family of solutions in $GL(\eps_1)$ satisfying~\eqref{eq:bressan-rate} also form a H\"older continuous semigroup, i.e. there exists $K > 0$ such that the following estimate holds:
    \begin{align*}
        \norm{u_1(t)-u_2(t)}_{L^1((-R,R))} \leq K\norm{u_1^0-u_2^0}_{L^1((-R-ct,R+ct))}^\theta \indent t \in [0,T],
    \end{align*}
    where $u_1, u_2 \in GL(\eps_1)$ satisfy \eqref{eq:bressan-rate}. Here, $\theta < 1$ may be taken arbitrarily close to $1$ by choosing $\eps_1 > 0$ sufficiently small.
\end{proposition}
We remark that the above condition~\eqref{eq:bressan-rate} is actually weaker than the assumption made in \cite[Theorem 1.1]{2025arXiv250500420B}, but their proof still holds.

\subsection{Short time weak-BV stability for approximations of Sobolev data}
In the proofs of \Cref{thm:main-gl} and \Cref{thm:main-obs}, it will be necessary to show a short time stability estimate for solutions with BV initial data that is a mollification of a function in $H^s_{\loc}(\R; B_{\eps_1}(d))$. In this section, we give an exposition of this procedure. 
\par To begin, we fix a mollifier $\gamma \in C^\infty(\R)$ satisfying 
\begin{equation} \label{eq:mol-assumptions}
    \gamma \geq 0,\quad \text{supp}(\gamma) \subset[-1,1],\quad \int_\R \gamma(x)\,\diff x = 1, \quad\text{and}\quad \gamma(x) = \gamma(-x)\quad \forall \ x\in \R.
\end{equation} Fix a $t_0 \geq 0$. Let $v \in H^s_{\loc}(\R; B_{\eps_1}(d))$ for $s > 0$. We define the approximants $\{v_\delta\}_{\delta > 0}$ by 
\begin{equation}\label{eq:mol-approx}v_\delta:= v \ast \gamma_\delta,\end{equation} where $\gamma_\delta(x) := \frac{1}{\delta}\gamma(\frac{x}{\delta})$.
We note that for any interval $(-R,R)$ there exists a $K > 0$ such that these approximants satisfy the following bounds,
\begin{align}
    TV(v_\delta; L) &\leq K \norm{v}_{L^\infty(\R)} \frac{L}{\delta}, \quad \forall \ L > 0, \label{eq:Hs-mol-TV} \\
    \norm{v - v_\delta}_{L^2((-R,R))} &\leq K \norm{v}_{H^s( (-R-\delta,R+\delta) )} \delta^{s}. \label{eq:Hs-L2-conv}
\end{align}
We provide a proof in \Cref{lem:sob-for-thm-obs}.

Let $\delta,R >0$ be fixed. Recall that there exists a constant $\delta_0 > 0$ such that if $\norm{u^0}_{\text{BV}} \leq \delta_0$ and $\norm{u^0-d}_{L^\infty} \leq \delta_0$, then a sequence of front-tracking approximations emanating from $u^0$ is well-defined and has a unique limit \cite{MR1367356} (see \Cref{sec:ft} for a detailed overview of front-tracking). 
Define
\begin{align}\label{traplength}
L:=\frac{\delta\delta_0}{2K \norm{v}_{L^\infty(\R)}}.
\end{align}
Then, $\eqref{eq:Hs-mol-TV}$ implies that $TV(v_\delta;L) \leq \delta_0$. 
We define a family of trapezoids $\Delta_n$ for $n \in \mathbb{Z}$, where:
\begin{align*}
\Delta_n:=\left\{(\tau,x)\, \Big|\, \tau \in \left [t_0, t_0+\frac{L}{4c}\right],\, x \in \left[\frac{L}{2}n+c(\tau-t_0), \frac{L}{2}n+L-c(\tau-t_0)\right]\right\},
\end{align*}
where the constant $c > 0$ is the speed of information from \Cref{thm:cor} (see \Cref{fig:trapezoid-sum}). 
\begin{figure}
\begin{center}
\begin{tikzpicture}[scale=1.5]

\draw[->, thick] (0,1.3) -- (0,3.5) node[above] {$t$};

\draw[dotted] (0,3) -- (8,3);
\node[left] at (-0.2,3) {$t_0 + \frac{L}{4 c}$};
\node[right] at (8+0.2,3) {$\hphantom{t_0 + \frac{L}{4 c}}$};

\draw[dotted] (0,2) -- (8,2);
\node[left] at (-0.2,2) {$t_0$};

\draw[->, thick] (0,1.3) -- (8,1.3) node[right] {$x$};

\foreach \x in {1,2,4,5}
    \draw[dashed,color={rgb:red,191;green,87;blue,0}] (\x,2) -- (\x + 2,2) -- (\x + 1.5,3) -- (\x + .5,3) -- cycle; 
\node at (.5,2.5) {$\cdots$};
\node at (7.5,2.5) {$\cdots$};
\draw[thick] (3,2) -- (3 + 2,2) -- (3 + 1.5,3) -- (3 + .5,3) -- (3,2); 
\draw [very thick, decorate, decoration={calligraphic brace,amplitude=5pt,mirror}] (3,2-.1) -- (5,2-.1) node[midway,yshift=-1.2em]{$L$};
\draw [very thick, decorate, decoration={calligraphic brace,amplitude=5pt}] (3.5,3+.1) -- (4.5,3+.1) node[midway,yshift=1.2em]{$L/2$};

\end{tikzpicture}
\end{center}
\caption{A weak entropy solution is defined for short time by taking a limit of small-BV front tracking approximations on each trapezoid.}
\label{fig:trapezoid-sum}
\end{figure}
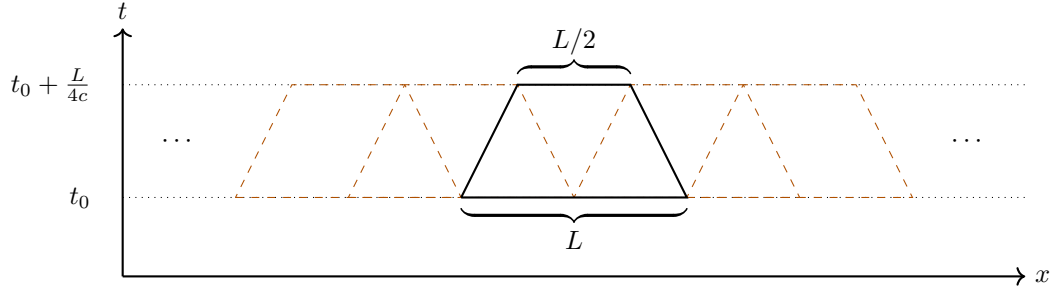
By construction, $v_\delta$ has total variation less than $\delta_0$ at the base of each trapezoid, so a limit of front-tracking approximations $\mathcal{S}_t v_\delta$ may be defined within the trapezoid $\Delta_n$ for each $n$. Further, by \Cref{thm:cor}, we have an estimate
\begin{multline}\label{eq:trapest}
\norm{u(t, \cdot)-\mathcal{S}_t v_\delta}_{L^1((\frac{L}{2}n+c(t-t_0), \frac{L}{2}n+L-c(t-t_0)))} \leq K\sqrt{L} \norm{u(t_0,\cdot)-v_\delta}_{L^2((\frac{L}{2}n, \frac{L}{2}n+L))}\\\text{ for $t \in \left[t_0,t_0+\frac{L}{4c}\right]$},
\end{multline}
for any $u \in GL(\eps_1)$. Note that we may actually define $\mathcal{S}_tv_\delta$ on the entire time strip $\big\{(t,x)\big| t \in [t_0, t_0+\frac{L}{4c}], x \in \R\big\}$ as each front-tracking approximation is well-defined on the overlap of trapezoids $\Delta_n$ and $\Delta_{n+1}$ (cf. \Cref{rem:finitepropspeed}). In other words, $\mathcal{S}_tv_\delta$ is a weak entropy solution to the initial value problem
\begin{align}\label{trapivp}
\left\{\begin{aligned}
        v_t+[f(v)]_x&=0 & (t,x) &\in \left[t_0, t_0+\frac{L}{4c}\right] \times \R, \\
        v(t_0,x) &= v_\delta(x) & x &\in \R. 
    \end{aligned}\right.
\end{align}
    Finally, we may sum the estimate \eqref{eq:trapest} in $n$ and apply Cauchy-Schwarz to obtain
\begin{align*}\label{eq:stripest}
\norm{u(t, \cdot)-\mathcal{S}_t v_\delta}_{L^1((-R+c(t-t_0), R-c(t-t_0)))} &\leq K\sqrt{L} \sum_{n = - 2R/L }^{2R/L} \norm{u(t_0,\cdot)-v_\delta}_{L^2((\frac{L}{2}n, \frac{L}{2}n+L))} \\
&\leq 4K\sqrt{L} \sqrt{\frac{R}{L}} \norm{u(t_0) - v_\delta}_{L^2((-R,R))} \\
&= 4K\sqrt{R} \norm{u(t_0) - v_\delta}_{L^2((-R,R))}, \indent \text{for $t \in \left[t_0,t_0+\frac{L}{4c}\right]$},
\end{align*}
for any $R > 0$. We remark that we also decrease $L$ slightly if needed such that $2R/L$ is an integer.
We summarize this section in the following lemma.
\begin{lemma}\label{lem:trapezoidconstruction}
Fix $R>0$, $s > 0$, and $t_0 \geq 0$.  Then, there exists a constant $ \eps_1 > 0$ such that the following holds. For any $v\in H^s_{\loc}(\R; B_{\eps_1}(d))$, let $\{v_\delta\}_{\delta > 0}$ be a sequence of approximants of $v$ given by mollification~\eqref{eq:mol-approx}. Then, there exists a weak entropy solution $\mathcal{S}_tv_\delta$ to the initial value problem \eqref{trapivp}, where $L$ is defined in \eqref{traplength}, which is obtained as the limit of front-tracking approximations. Further, for any $u \in GL(\eps_1)$, we have an estimate
\begin{align*}
||u(t, \cdot)-\mathcal{S}_t v_\delta||_{L^1((-R+c(t-t_0), R-c(t-t_0)))} \leq K\sqrt{R} \norm{u(t_0, \cdot)-v_\delta}_{L^2((-R,R))}, \quad \textrm{\emph{for} } t \in \left[t_0,t_0+\frac{L}{4c}\right].
\end{align*}
\end{lemma}
We can now begin the proof of \Cref{thm:main-gl}.

\subsection{Proof of \Cref{thm:main-gl}}
Without loss of generality, we treat only the case $u^0 \in H^s_{\loc}(\R)$ as $W^{s,p}_{\loc}(\R)$ embeds into $H^{s^*}_{\loc}(\R)$ for $s^*$ smaller if needed (see \cite[Corollary 2.30]{MR4567945}). Fix $R,T > 0$. Denote the interval $I := [-R -cT, R+cT]$ as the base of the information cone, where the constant $c > 0$ is the speed of information (see \Cref{rem:soi}). Further, let $I_\delta := [-R - cT - \delta, R+cT + \delta]$ be a slightly larger set. This is introduced to accommodate the mollification $u_\delta$ depending on values of $u$ outside $I$. Let $u^0 \in H^s_\loc(\R)$ and let $u_\delta = u^0 \ast \gamma_\delta$. Without loss of generality, we assume $s < 1$. Let $u \in GL(\epsilon_1)$ be any solution in the Glimm--Lax class having initial data $u^0$. Our goal is to show that $u$ verifies \eqref{eq:bressan-rate} and apply \Cref{prop:bmv-main}.
Fix $0 < t < T$ and let $\delta >0$ be a constant to be determined. By the triangle inequality we have 
\beq \label{eq:fast-cv-triangle}\begin{aligned} 
        \norm{u(t,\cdot) - u^0(\cdot)}_{L^1((-R-c(T-t),R+c(T-t)))} \leq& \norm{u(t,\cdot) - \mathcal{S}_tu_\delta }_{L^1((-R-c(T-t),R+c(T-t)))}\\ &\ +\norm{\mathcal{S}_tu_\delta - u_\delta}_{L^1((-R-c(T-t),R+c(T-t)))}\\ &\ +\norm{u_\delta - u^0}_{L^1((-R-c(T-t),R+c(T-t)))}.   
    \end{aligned} \eeq
\par By \Cref{lem:trapezoidconstruction}, we have
\begin{align*}
\norm{u(t,\cdot) - \mathcal{S}_tu_\delta}_{L^1((-R-c(T-t),R+c(T-t)))} \leq K||u^0-u_\delta||_{L^2(I)} \leq K \norm{u^0}_{H^s(I_\delta)} \delta^{s}, \text{ for $t \in \left[0, \frac{L}{4c}\right]$}.
\end{align*}
Choosing $\delta=t^{1-\frac{s}{2}}$, it follows that $t < \frac{L}{4c}$ when it satisfies $t \leq \left(\frac{\delta_0}{8 c \norm{u^0}_{L^\infty}(\R)}\right)^{\frac{2}{s}}$ by the definition of $L$~\eqref{traplength}. 
That is, for $t$ sufficiently small we have
\begin{multline*}
\norm{u(t,\cdot) - \mathcal{S}_tu_\delta}_{L^1(-R-c(T-t),R+c(T-t)))} \leq K\norm{u^0-u_\delta}_{L^2(I)} \leq K\norm{u^0}_{H^s(I_\delta)} t^{\frac{s}{2}}, \, \\\text{for $t \in \left[0,\left(\frac{\delta_0}{8c\norm{u^0}_{L^\infty(\R)} }\right)^{\frac{2}{s}}\right]$},
\end{multline*}
where $\delta^s=t^{s(1-\frac{s}{2})} < t^{\frac{s}{2}}$ for small $t$ follows from $s<1$. Next, we bound the second and third term on the right-hand side of \eqref{eq:fast-cv-triangle}. Recall the following estimates for any BV function $v$ and any interval $(a,b) \subset \R$:
    \begin{align} 
        \limsup_{\epsilon \to 0+} \frac{\norm{\mathcal{S}_{t+\epsilon} v - \mathcal{S}_t v}_{L^1((a,b))} }{\epsilon} &\leq K\ ||\mathcal{S}_tv||_{BV((a,b))}, \label{eq:diff-L1-TV}\\
        ||\mathcal{S}_tv||_{BV((a,b))} &\leq K\ ||v||_{BV((a-ct,a+ct))}, \label{eq:short-time-BV-bound}
    \end{align}
    where~\eqref{eq:short-time-BV-bound} holds for all $t < \frac{L}{4c}$.
    Equation~\eqref{eq:diff-L1-TV} can be found in \cite[Theorem 4.3.1]{MR3468916} while~\eqref{eq:short-time-BV-bound} holds on each $\Delta_n$ by standard front tracking considerations.
    From the equations~\eqref{eq:diff-L1-TV} and~\eqref{eq:short-time-BV-bound} and the total variation of $u_\delta$~\eqref{eq:Hs-mol-TV} we have the estimate
    \begin{align*}
        \norm{\mathcal{S}_tu_\delta - u_\delta}_{L^1((-R-c(T-t),R+c(T-t)))} &\leq K\int_0^t \norm{u_{\delta}}_{BV((-R-cT,R+cT))}\,\diff s \\
        &\leq K\norm{u^0}_{L^\infty(\R)} \frac{2(R+cT)}{\delta}t \leq K \norm{u^0}_{L^\infty(\R)} t^{\frac{s}{2}},
    \end{align*}
    where the last inequality follows from the choice $\delta = t^{1-\frac{s}{2}}$. 
    This choice and~\eqref{eq:Hs-L2-conv} also immediately bounds the final term
    \beq \label{eq:proof-alg-cv} \norm{u_\delta - u^0}_{L^1((-R-c(T-t),R+c(T-t)))} \leq K\norm{u^0}_{H^s(I_\delta)}\delta^{s} \leq K\norm{u^0}_{H^s(I_\delta)} t^{\frac{s}{2}}. \eeq
    This suffices to show 
    \begin{multline} \label{eq:alg-conv}  \norm{u(t,\cdot) - u^0(\cdot)}_{L^1((-R-c(T-t),R+c(T-t)))} \leq K \left(\norm{u^0}_{H^s(I_\delta)} + \norm{u^0}_{L^\infty(\R)}\right) t^{\frac{s}{2}},  \\ \text{for } t \in \left[0,\left(\frac{\delta_0}{8c\norm{u^0}_{L^\infty(\R)} }\right)^{\frac{2}{s}}\right]. \end{multline}
    Furthermore, for $t > \left(\frac{\delta_0}{8c\norm{u^0}_{L^\infty(\R)} }\right)^{\frac{2}{s}}$ we have 
    \beq \label{eq:large-time} \norm{u(t,\cdot) - u^0(\cdot)}_{L^1((-R-c(T-t),R+c(T-t)))} \leq C_0 \sqrt{\eps_1} (2R + 2cT) \leq K \norm{u^0}_{L^\infty(\R)} t^{\frac{s}{2}}, \eeq
    which follows from~\eqref{smalllinfty}.
    This establishes 
    $ u^0 \in {\mathcal{D}}^\alpha$ for $\alpha=\frac{s}{2}$ and that every solution $u \in GL(\eps_1)$ with initial data $u^0$ has the algebraic $L^1$ convergence backwards in time~\eqref{eq:bressan-rate}. 
    In particular, we are able to choose the associated constant $M$ in \Cref{prop:bmv-main} to be larger than the constants appearing in the bounds~\eqref{eq:proof-alg-cv} and~\eqref{eq:large-time}. 
    Taking $\epsilon_1$ sufficiently small to apply \Cref{prop:bmv-main}, \Cref{thm:main-gl} follows. 
\begin{remark}
We remark that the preceding proof is comparable to one in the scalar case showing that solutions with initial data in $H^s$ have the convergence rate~\eqref{eq:Da-decay} \cite[Proposition 4.2]{MR4841631}. 
Their computation uses the Kru\v{z}kov theory, which gives an $L^1$ semigroup on solutions up to $t=0$, to control the first term on the right of the inequality~\eqref{eq:fast-cv-triangle}. 
This is not available at the level of systems, and it is unknown how to bound this term using the $L^1$ theory, even with the refined estimate \Cref{prop:L1-bressan}. 
\Cref{thm:cor} provides an intermediate estimate which suffices due to being a weak-BV type estimate. 
\end{remark}
\begin{remark}
We also remark that one can obtain \Cref{thm:main-gl} by classical methods. Indeed, the term $||u(t,\cdot)-\mathcal{S}_tu_\delta||_{L^1}$ may be estimated by replacing the use of \Cref{lem:trapezoidconstruction} in the proof of \Cref{thm:main-gl} by the classical weak-strong $L^2$-stability estimate (see \cite[Theorem 5.2.1]{MR3468916}). This works because an $L^2$-type estimate is required on a time scale of length $\sim \delta$ and the solution $\mathcal{S}_tu_\delta$ will be Lipschitz on a time scale of length $\sim \delta$.
This, however, requires careful tracking of the blow-up time of the solution, while our more robust estimate from \Cref{lem:trapezoidconstruction} holds after blow-up. 
This can be seen in the proof of \Cref{thm:sob-ext-obs}, where we use \Cref{lem:trapezoidconstruction} for time scale greater than $\sim \delta$. Thus, we use this $L^2$-type estimate on a time scale which is longer than the time scale on which the approximate solutions remain Lipschitz. This shows the versatility of using a weak-BV stability theory rather than the classical weak-strong theory.
\end{remark}
\subsection{Proof of \Cref{thm:main-inc}}
    To begin, we show the inclusion $\tilde {\mathcal{D}}_{\alpha} \subset {\mathcal{D}}_\alpha$ following \cite[p. 3]{2025arXiv250500420B}.
    Fix any $u^0 \in \tilde D_{\alpha}$.
    By definition, there exists a weak solution $u$ with initial data $u^0$ satisfying 
    $$  \sup_{ t \in (0,T)} t^{1-\alpha}\norm{u(t,\cdot )}_{\text{BV}( (a,b))} < \infty,$$
    for any $(a,b) \subset \R$.
    To show the $L^1$ convergence rate as $t \downarrow 0$ as in the definition~\eqref{eq:D-alpha} we recall~\eqref{eq:diff-L1-TV}.
    Integrating~\eqref{eq:diff-L1-TV} in time using the total variation decay of $\tilde D_{\alpha}$ we find
    $$ \norm{u(t,\cdot) - u^0}_{L^1((a,b))} \leq t^{\alpha} \sup_{ t \in (0,T)} \left( t^{1-\alpha}\norm{u(t)}_{\text{BV}((a,b))}\right),$$
    establishing $u^0 \in {\mathcal{D}}_\alpha.$ 

    Next, we show $ {\mathcal{D}}_\alpha \subset  H^r_\loc(\R)$ for $r < \frac{\alpha}{2(1+\alpha)}$.
    Let $u^0 \in \mathcal{D}_\alpha$ and let $u$ be a weak entropy solution with initial data $u^0$ verifying the condition
    \beq \label{eq:Dalpha-inc-cond} \sup_{t \in (0,T)} t^{-\alpha } \norm{u(t,\cdot) - u^0}_{L^1( (a,b))} < \infty, \eeq
    for all $(a,b)\subset \R$.
    For any $h \in \R$ we have 
    \beq \label{eq:HsDa-tri}\begin{aligned}
        \norm{u^0(\cdot) - u^0(\cdot - h)}_{L^2(I_h)} &\leq \norm{u^0(\cdot) - u(t,\cdot)}_{L^2(I_h)}  + \norm{u(t,\cdot) - u(t,\cdot - h)}_{L^2(I_h)}  \\
        &+\norm{u(t,\cdot - h) - u^0(\cdot - h)}_{L^2(I_h)} =: J_1 + J_2 +J_3,
    \end{aligned}\eeq
    for any $t > 0$, where $I_h := (a,b)\cap(a+h,b+h)$.
    We first note that
    \beq \label{eq:HsDa-I1I3} J_1 + J_3 \leq K t^{\alpha/2} \sup_{t \in (0,T)} \left[t^{-\alpha} \norm{u^0(\cdot) - u(t,\cdot)}_{L^1((a,b))} \right]^{1/2}, \eeq
    where~\eqref{eq:Dalpha-inc-cond} implies that the above supremum in~\eqref{eq:HsDa-I1I3} is finite. 
    Next, by Fubini's theorem we have
    \beq \label{eq:HsDa-I2}
    \begin{aligned}
        J_2 &\leq K\norm{u}_{L^\infty([0,t]\times\R)}^{1/2} \norm{u(t,\cdot) -u(t,\cdot - h)}_{L^1(I_h)}^{1/2} \\
        &\leq K\norm{u}_{L^\infty([0,t]\times\R)}^{1/2} \left[|h| \norm{u(t,\cdot)}_{\text{BV}((a,b))} \right]^{1/2} \\
        &\leq K\norm{u}_{L^\infty([0,t]\times\R)}^{1/2} \left[ \frac{|h|(b-a)}{t} + \sqrt{\eps_1}|h|  \right]^{1/2}.
    \end{aligned}\eeq
    Taking $t = |h|^{1-\frac{\alpha}{1+\alpha}}$ in~\eqref{eq:HsDa-I1I3} and~\eqref{eq:HsDa-I2} and substituting these bounds into~\eqref{eq:HsDa-tri}, we finally have 
    $$ \norm{u^0(\cdot) - u^0(\cdot - h)}_{L^2(I_h)} \leq K |h|^{\frac{\alpha}{2(1+\alpha)}}, $$
    where $K$ depends on the supremum~\eqref{eq:Dalpha-inc-cond} and the size of our interval $(a,b)$.
    To complete the proof, we use the above bound to show the $H^r$ seminorm of $u^0$ is finite,
    \begin{align*}
        |u^0|_{H^r((a,b))}^2 &= \int_a^b \int_{x-b}^{x-a} \frac{|u^0(x) - u^0(x-h)|^2}{|h|^{1+2r}}\,\diff h\,\diff x \\
        &= \int_{-(b-a)}^{b-a} \frac{\norm{u^0(\cdot) - u^0(\cdot-h)}_{L^2((a,b) \cap (a+h,b+h))}^2}{|h|^{1+2r}}\,\diff h \\
        &\leq K \int_0^{b-a} \frac{|h|^{\frac{\alpha}{1+\alpha}}}{|h|^{1+2r}}\,\diff h < \infty,
    \end{align*}
    provided $r < \frac{\alpha}{2(1+\alpha)}$.

    To see $H^s_\loc(\R; B_{\eps_1}(d)) \subset {\mathcal{D}}_\omega$ for $\omega \leq \frac{s}{2}$ we note if $\epsilon_1$ is sufficiently small then any data $u^0 \in H^s_\loc(\R; B_{\eps_1}(d))$ has a solution $u \in GL(\epsilon_1)$ by the classical result of Glimm--Lax. 
    The fact that this solution $u$ satisfies 
    $$ \sup_{t \in (0,T)} t^{-s/2 } \norm{u(t,\cdot) - u^0}_{L^1((a,b))} < \infty, $$
    for all $(a,b) \subset \R$
    follows from the proof of \Cref{thm:main-gl} when $\eps_1$ is sufficiently small.
    Additionally, since the solution $u$ is given by the Glimm--Lax construction it satisfies~\eqref{tv1/t}, establishing $u^0 \in {\mathcal{D}}_\omega$.

\subsection{Proof of \Cref{thm:main-obs}} 
Now, we prove \Cref{thm:main-obs}. \Cref{thm:main-obs} will follow directly from \Cref{prop:entropy-conservation} and \Cref{prop:nonperiodic}. To begin, we give a well known criterion to determine whether a periodic solution to \eqref{cl} additionally verifies an entropy inequality \eqref{entropic}. In the forthcoming theorem, we denote the flat, 1-dimensional torus as $\T^1$.
\begin{proposition} \label{prop:entropy-conservation}
    There exists an $\eps_1 > 0$ such that the following is true. 
    Fix $T > 0$ and any convex entropy $\eta\in C^2(B_{\eps_1}(0))$. If $u \in L^3(0,T; B^\alpha_{3,\infty}(\T^1;B_{\eps_1}(0)))$ is a solution to the equation
    \begin{equation} \label{cltorus}
     \begin{aligned}
        u_t+[f(u)]_x&=0 & (t,x) &\in \R^+ \times \T^1, \\
    \end{aligned}
\end{equation}
    then the solution locally conserves entropy, 
    \begin{equation} \label{eq:ent-eq}
        \partial_t \eta(u) + \partial_x q(u) = 0 \quad \text{in}\quad \mathcal{D}'( [0,T] \times \T^1).
    \end{equation}
\end{proposition}
The above proposition is well known, and essentially follows from the considerations in Constantin, E, Titi \cite{CET-Onsager}. We note that entropy conservation for very general systems of conservation laws (even in multi-d) was treated in \cite{MR4038422}, \cite{gwiazda2018note}.
For the reader's convenience, we provide a proof of this proposition for the $1$-d system \eqref{cltorus} in \Cref{appendix-entropy}.

The next proposition is a more general statement than \Cref{thm:main-obs} that shows uniqueness for non-periodic solutions.
\begin{proposition}\label{prop:nonperiodic}
Let $\alpha > \frac{1}{2}$. Then, there exists $\eps_1 > 0$ such that the following is true. For any $T>0$, suppose $u \in L^\infty(0,T;C^\alpha_\loc(\R;B_{\eps_1}(0)))$ is a weak solution to \eqref{cl} verifying \eqref{entropic} for at least one strictly convex entropy $\eta \in C^3(B_{\eps_1(0)};\R)$ with initial data $u^0 \in C^\alpha_\loc(\R;B_{\eps_1}(0))$. Then, $u$ coincides with the unique solution in $GL(\eps_1)$ with initial data $u^0$.
\end{proposition}
\begin{proof}[Proof of \Cref{prop:nonperiodic}]
Let $u \in L^\infty(0,T;C^\alpha(\R;B_{\eps_1}(0)))$ verify \eqref{entropic} with respect to a strictly convex entropy $\eta$. The primary idea is to sample $u$ at a sequence of times uniformly separated by a time step $\tau > 0$. 
\par At each time step, we construct a new approximate solution $v_i$ with initial data approximating $u$ at the time of sampling. 
We then use \Cref{prop:main} to prove a small modification of \Cref{lem:trapezoidconstruction}\footnote{\label{footnote:weak-bv-for-Ca}We remark that \Cref{prop:main} is simply a more general version of \Cref{thm:cor}. \Cref{prop:main} only requires solutions to be regular enough to establish a notion of limit across Lipschitz paths $h(t)$ (cf. \Cref{def:strongtrace}). 
This degree of regularity is guaranteed in this case, due to all solutions considered being continuous. As a result, \Cref{lem:trapezoidconstruction} holds when the weak solution is $C^\alpha$ in space by applying \Cref{prop:main} whenever \Cref{thm:cor} is used.} which lets us deduce that, for a short time after the approximation, $v_i$ and $u$ are close. 
After this short time, we stop comparing $v_i$ to $u$ using \Cref{prop:main} and instead opt to compare the samples $v_i$ to each other using the very fine estimate of \Cref{prop:L1-bressan}. 
In doing this, we are able deduce that any solution with a certain degree of regularity must coincide with the Glimm--Lax solution emanating from the same initial data, which is unique due to \Cref{thm:main-gl}. 

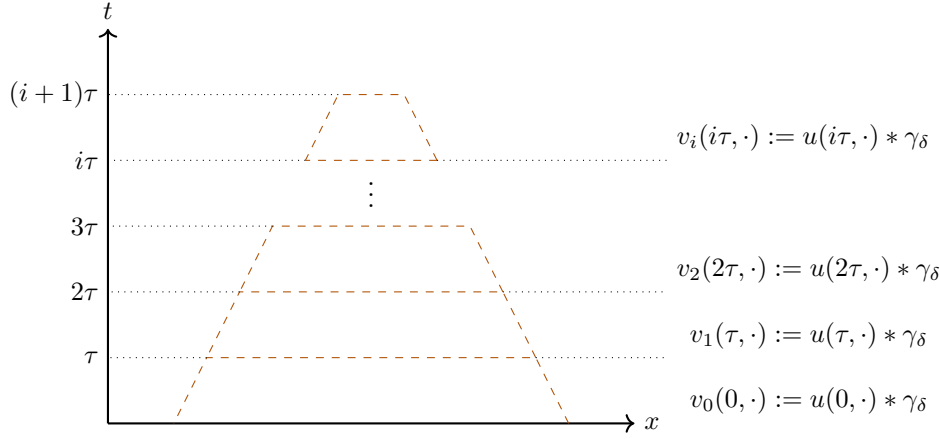
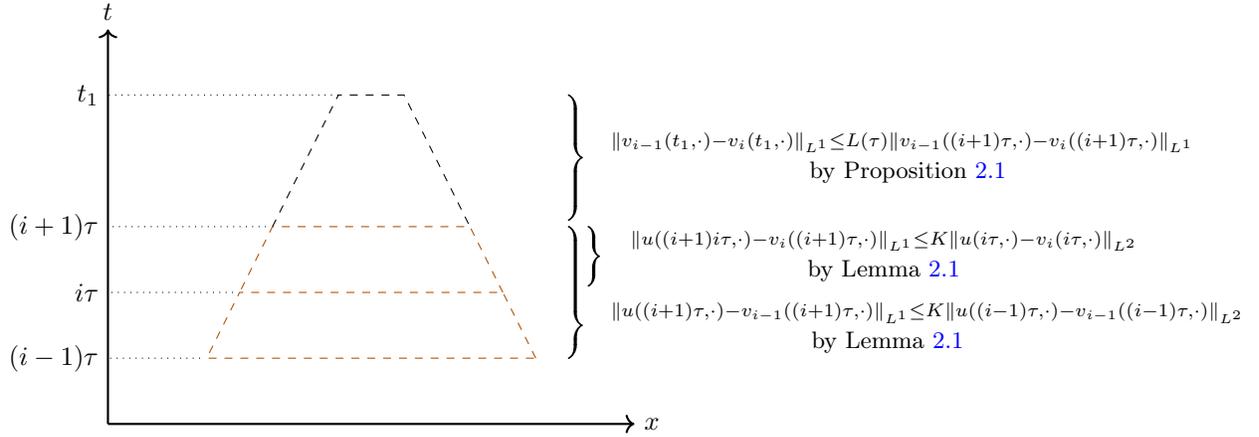
\begin{figure}
\begin{subfigure}{\textwidth}
  \centering
  \begin{tikzpicture}[scale=1.75]

\draw[->, thick] (0,0) -- (0,3) node[above] {$t$};


\draw[->, thick] (0,0) -- (4,0) node[right] {$x$};

\foreach \y in {0,0.5,1,2}
    \draw[dashed,color={rgb:red,191;green,87;blue,0}] (3.5-\y/2,\y) -- (3.5-\y/2-.25,\y+.5) -- (.75+\y/2,\y+.5) -- (.5+\y/2,\y); 
\draw[dotted] (0+.75,0+.5) -- (0,0 + .5) node[anchor=east] {$\tau$};
\foreach \y in {2,3}
    \draw[dotted] (\y/4+.5,\y/2) -- (0,\y/2 ) node[anchor=east] {$\y \tau$};
\draw[dotted] (1.5,2) -- (0,2) node[anchor=east] {$i \tau$};
\draw[dotted] (1.75,2+.5) -- (0,2+.5 ) node[anchor=east] {$(i+1) \tau$};
\draw[dashed,color={rgb:red,191;green,87;blue,0}] (5/2,2) -- (3/2,2); 
\node at (2,1.8) {$\vdots$};
\node[anchor=west] at (3.7,0.5) {$\hphantom{\scriptstyle \norm{u((i+1)\tau,\cdot) - v_{i-1}((i+1)\tau,\cdot)}_{L^1} \leq K\norm{u((i-1)\tau,\cdot) - v_{i-1}((i-1)\tau,\cdot)}_{L^2} }$};

\foreach \y in {0}
    \draw[draw] (4.25,\y/2) node[anchor=south west] {$\hphantom{2}v_{\y}(0,\cdot) := u(0,\cdot)\ast \gamma_\delta$};
\foreach \y in {1}
    \draw[dotted] (3.5-\y/4, \y/2) -- (4.25,\y/2) node[anchor=south west] {$\hphantom{2}v_{\y}( \tau,\cdot) := u(\tau,\cdot)\ast \gamma_\delta$};
\foreach \y in {2}
    \draw[dotted] (3.5-\y/4, \y/2) -- (4.25,\y/2) node[anchor=south west] {$v_{\y}(\y \tau,\cdot) := u(\y \tau,\cdot)\ast \gamma_\delta$};
\foreach \y in {4}
    \draw[dotted] (3.5-\y/4, \y/2) -- (4.25,\y/2) node[anchor=south west] {$v_{i}(i \tau,\cdot) := u(i \tau,\cdot)\ast \gamma_\delta$};

\end{tikzpicture}
  \caption{The definition of the approximants $v_i$.}
  \label{fig:v-def-obs}
\end{subfigure} \\
\begin{subfigure}{\textwidth}
  \centering
  \begin{tikzpicture}[scale=1.75]

\draw[->, thick] (0,0) -- (0,3) node[above] {$t$};
\draw[->, thick] (0,0) -- (4,0) node[right] {$x$};

\foreach \y in {.5}
    \draw[dashed,color={rgb:red,191;green,87;blue,0}] (3.5-\y/2,\y) -- (3.5-\y/2-.25,\y+.5) -- (.75+\y/2,\y+.5) -- (.5+\y/2,\y) -- cycle; 
\foreach \y in {1}
    \draw[dashed,color={rgb:red,191;green,87;blue,0}] (3.5-\y/2,\y) -- (3.5-\y/2-.25,\y+.5) -- (.75+\y/2,\y+.5) -- (.5+\y/2,\y); 
\draw[dotted] (.75-.05,.5) -- (0,.5) node[anchor=east] {$(i-1) \tau$};
\draw[dotted] (1-.05,1) -- (0,1 ) node[anchor=east] {$i \tau$};
\draw[dotted] (1.25-.05,1.5) -- (0,1.5 ) node[anchor=east] {$(i+1) \tau$};
\draw[dotted] (1.75,2.5) -- (0,2.5 ) node[anchor=east] {$t_1$};
\draw[dashed] (1.25,1.5) -- (1.75,2.5) -- (2.25,2.5) -- (2.75,1.5);
\draw [very thick,decorate,decoration={calligraphic brace,amplitude=4pt,mirror}] (3.65,1 + .05) -- (3.65,1.5) node[align=center, midway,anchor=west, xshift=1.2em] {$\scriptstyle \norm{u((i+1)i\tau,\cdot) - v_i((i+1)\tau,\cdot)}_{L^1} \leq K\norm{u(i\tau,\cdot) - v_i(i\tau,\cdot)}_{L^2} $ \\  \small by \Cref{lem:trapezoidconstruction}};
\draw [very thick,decorate,decoration={calligraphic brace,amplitude=5pt,mirror,aspect=0.25}] (3.5,0.5) -- (3.5,1.5) node[align=center, midway,anchor=north west, xshift=1.2em] {$\scriptstyle \norm{u((i+1)\tau,\cdot) - v_{i-1}((i+1)\tau,\cdot)}_{L^1} \leq K\norm{u((i-1)\tau,\cdot) - v_{i-1}((i-1)\tau,\cdot)}_{L^2} $\\  \hspace{-3em}\small by \Cref{lem:trapezoidconstruction}};

\draw [very thick,decorate,decoration={calligraphic brace,amplitude=5pt,mirror}] (3.5,1.5+.05) -- (3.5,2.5) node[align=center, midway,anchor=west, xshift=1.2em] { $\scriptstyle \norm{v_{i-1}(t_1,\cdot) - v_i(t_1,\cdot)}_{L^1} \leq L(\tau)\norm{v_{i-1}( (i+1)\tau,\cdot) - v_i((i+1)\tau,\cdot)}_{L^1} $ \\  \hspace{.45em}\small by \Cref{prop:L1-bressan}};

\end{tikzpicture}
  \caption{Estimating $\norm{v_{i-1} - v_i}_{L^1}$ at the terminal time $t_1$.}
  \label{fig:single-obs}
\end{subfigure}
\caption{In \Cref{fig:v-def-obs} we cover the information cone in trapezoids. At the base of each trapezoid we define the approximant $v_i$, which has initial data approximating $u(i\tau,\cdot)$ and is also a solution to the system~\eqref{cl}. \Cref{fig:single-obs} shows our strategy for estimating the norm $\norm{v_{i-1}(t_1,\cdot)-v_i(t_1,\cdot)}_{L^1}$. For short time (of size $\tau$ or $2\tau$) we use \Cref{lem:trapezoidconstruction} to estimate $\norm{u - v_i}_{L^1}$ and $\norm{u-v_{i-1}}_{L^1}$. From time $(i+1)\tau$ to $t_1$, we apply the $L^1$-Lipschitz bound from \Cref{prop:L1-bressan}. This results in a much better estimate than simply iterating \Cref{lem:trapezoidconstruction}.}
\label{fig:fig}
\end{figure}

Fix an interval $(-R,R) \subset \R$. 
To sample $u$, we first fix a suitable mollifier $\gamma \in C^\infty_c(\R)$ satisfying~\eqref{eq:mol-assumptions}.
For $\delta > 0$ we then denote $\gamma_\delta(x) := \frac{1}{\delta}\gamma(\frac{x}{\delta})$. 
We note that if $v \in C^\alpha_\loc(\R;B_{\eps_1}(d))$ then its mollifications satisfy the estimates
\begin{align}
    \norm{v \ast \gamma_\delta}_{\text{BV}((-a,b))} &\leq \norm{v}_{C^\alpha((-R-\delta,R+\delta))} \frac{b-a}{\delta^{1-\alpha}}, \quad \forall \ (a,b) \subset (-R,R), \label{eq:Ca-TV}\\
    \norm{v - v \ast \gamma_\delta}_{L^2(( -R, R))} &\leq  2R \norm{v}_{C^\alpha( (-R-\delta,R+\delta) )}\delta^\alpha,  \label{eq:Ca-L2}
\end{align}
(cf. the $p = \infty$ case of \Cref{lem:sob-for-thm-obs}).
Hence by the construction of Glimm and Lax \cite{MR265767} and \Cref{thm:main-gl} there exists a unique Glimm--Lax solution $u_{GL} \in GL(\epsilon_1)$ with initial data $u(0,\cdot)$, provided $\epsilon_1$ is sufficiently small. 
To show that $u$ coincides with $u_{GL}$, we consider the unique solution $v_0$ to~\eqref{cl} with initial data $u(0,\cdot) \ast \gamma_\delta$ for a $\delta  > 0$. 
Fix any $t_1 \in (0,T]$. The triangle inequality then gives
    \beq \label{eq:triangle-sob-gl} \begin{aligned} \norm{u(t_1,\cdot) - u_{GL}(t_1,\cdot)}_{L^1((-R + ct_1,R - ct_1))} \leq&\, \norm{u(t_1,\cdot) -v_0(t_1,\cdot)}_{L^1((-R + ct_1,R - ct_1))} \\&+ \norm{u_{GL}(t_1,\cdot) - v_0(t_1,\cdot)}_{L^1((-R + ct_1,R - ct_1))}. \end{aligned} \eeq
   We will proceed by bounding both terms on the right of~\eqref{eq:triangle-sob-gl} by quantities which converge to zero as $\delta \downarrow 0$, provided $\alpha > 1/2$ and $\epsilon_1$ is sufficiently small.

We begin with the term containing $u$. First, we fix a $\tau > 0$ to be determined later.
For each $i = 0,\dots, n$ where $n$ is the largest integer satisfying $n\tau < t_1$ we define $v_i$ as the unique Glimm--Lax solution to the initial value problem (see \Cref{fig:v-def-obs})
$$ \left\{\begin{aligned}
        (v_i)_t+[f(v_i)]_x&=0 & (t,x) &\in [i\tau,\infty) \times \R, \\
        v_i(i\tau,x) &= [u(i\tau,\cdot)\ast\gamma_\delta](x) & x &\in \R. 
    \end{aligned}\right. $$
We note that by \Cref{lem:trapezoidconstruction}, $v_i(t,\cdot)$ coincides with $ \mathcal{S}_{t-i\tau} \left[ u(i\tau,\cdot)\ast \gamma_\delta \right]$. 
This allows us to apply \Cref{lem:trapezoidconstruction} with two key modifications. Firstly, we replace $u \in GL(\eps_1)$ in \Cref{lem:trapezoidconstruction} with $u \in L^\infty(0,T;C^\alpha(\R;B_{\eps_1}(0)))$ (see  Footnote \ref{footnote:weak-bv-for-Ca} on p. \pageref{footnote:weak-bv-for-Ca}). Secondly, using the bound \eqref{eq:Ca-TV} in the proof of \Cref{lem:trapezoidconstruction} instead of \eqref{eq:Hs-mol-TV}, the length of the base of the trapezoids can instead be set as 
$$L = \frac{\delta^{1-\alpha}\delta_0}{\norm{u}_{L^\infty(0,T; C^\alpha( (-R-\delta,R+\delta)))}}.$$
This allows for the estimate \eqref{eq:trapest} to hold on a time interval longer than $\sim \delta$, which will be vital in upcoming estimates. This gives us the bound
\begin{multline} \label{eq:weak-strong-obs-whole}
    \norm{u(t,\cdot) - v_i(t,\cdot)}_{L^1( (-R + ct , R - ct) )} \leq K \norm{u(i\tau,\cdot) - v_i(i\tau,\cdot)}_{L^2((-R+ci\tau,R-ci\tau))}  \\\text{for } t \in \left[i\tau,i\tau + \frac{\delta^{1-\alpha} \delta_0}{8 c\norm{u}_{L^\infty( 0,T; C^\alpha( (-R-\delta,R+\delta)))}}\right].
\end{multline}
We wish to be able to compare the wild solution $u$ to at least one approximate $v_i$ at every time $r \in [0,t_1]$. 
This can be achieved by selecting the size of our time steps to be 
\beq  \tau = \frac{\delta^{1-\alpha} \delta_0}{16\Gamma c} = K \frac{\delta^{1-\alpha}}{\norm{u}_{L^\infty( 0,T; C^\alpha( (-R-\delta,R+\delta) ))}}. \label{eq:def-tau}\eeq
We note that this additionally gives us a bound on the number of samples taken,  
\beq n \leq \frac{t_1}{\tau} \leq K\norm{u}_{L^\infty( 0,T; C^\alpha( (-R-\delta,R+\delta)))} \delta^{\alpha - 1}, \label{eq:n-estimate} \eeq
for some $K > 0$. 

Next, we note the first term in~\eqref{eq:triangle-sob-gl} can be expanded by triangle inequality to give
\beq \label{eq:u-v0-triangle} \begin{aligned} 
    \norm{ u(t_1,\cdot) - v_0(t_1,\cdot) }_{L^1(( -R+ct_1, R-ct_1))} \leq&\, \norm{u(t_1,\cdot) - v_n(t_1,\cdot)}_{L^1(( -R+ct_1, R-ct_1))} \\
    &+ \sum_{i=1}^n \norm{v_i(t_1,\cdot) - v_{i-1}(t_1,\cdot)}_{L^1(( -R+ct_1, R-ct_1))}.
\end{aligned} \eeq
The first term of~\eqref{eq:u-v0-triangle} can be controlled by~\eqref{eq:weak-strong-obs-whole}, giving 
\begin{align*} \norm{u(t_1,\cdot) - v_n(t_1,\cdot)}_{L^1(( -R+ct_1, R-ct_1))} &\leq K \norm{u(n\tau,\cdot) - v_n(n\tau,\cdot)}_{L^2((-R+cn\tau,R-cn\tau))} \\&\leq K \norm{u(n\tau,\cdot) }_{C^\alpha((-R-\delta,R+\delta))} \delta^\alpha, \end{align*}
where the last inequality follows from definition of $v_i$ and~\eqref{eq:Ca-L2}.
For the sum we observe that at time $(i+1)\tau$ we have $v_i((i+1)\tau,\cdot),v_{i-1}((i+1)\tau,\cdot) \in D^{(\tau)}$. 
This allows us to apply \Cref{prop:L1-bressan}, yielding 
$$\norm{v_i(t_1,\cdot) - v_{i-1}(t_1,\cdot)}_{L^1(( -R+ct_1, R-ct_1))} \leq L(\tau)\norm{v_i( (i+1)\tau ,\cdot) - v_{i-1}((i+1)\tau ,\cdot)}_{L^1(( -R+c(i+1)\tau , R-c(i+1)\tau ))}.$$
To return to their time of definition, we simply use the triangle inequality, the bound~\eqref{eq:weak-strong-obs-whole}, and~\eqref{eq:Ca-L2} to find (see  \Cref{fig:single-obs})
\begin{align}
    \norm{v_i( (i+1)\tau ,\cdot) - v_{i-1}((i+1)\tau ,\cdot)}&_{L^1(( -R+c(i+1)\tau , R-c(i+1)\tau ))} \nonumber\\
    \leq&\, \hphantom{+} \norm{v_i( (i+1)\tau ,\cdot) - u((i+1)\tau ,\cdot)}_{L^1(( -R+c(i+1)\tau , R-c(i+1)\tau ))} \nonumber \\
    &+ \norm{v_{i-1}( (i+1)\tau ,\cdot) -u((i+1)\tau ,\cdot)}_{L^1(( -R+c(i+1)\tau , R-c(i+1)\tau ))} \nonumber \\
    \leq&\,\hphantom{+} K\norm{v_i( i\tau ,\cdot) - u(i\tau ,\cdot)}_{L^2(( -R+ci\tau , R-ci\tau ))} \nonumber \\
    &+ K\norm{v_{i-1}( (i-1)\tau ,\cdot) -u((i-1)\tau ,\cdot)}_{L^2(( -R+c(i-1)\tau , R-c(i-1)\tau ))} \nonumber \\
    \leq&\, K \left( \norm{u(i\tau, \cdot)}_{C^\alpha((-R-\delta,R+\delta))} + \norm{u((i-1)\tau, \cdot)}_{C^\alpha((-R-\delta,R+\delta))}\right) \delta^\alpha.\label{eq:obs-where-L2-rate-used}
\end{align} 
Substituting these considerations into~\eqref{eq:u-v0-triangle} and recalling the bounds~\eqref{eq:Ca-L2},~\eqref{eq:n-estimate}, we find 
\begin{align} 
    \norm{ u(t_1,\cdot) - v_0(t_1,\cdot) }&_{L^1(( -R+ct_1, R-ct_1))} \nonumber\\
    &\leq CK\norm{u(n\tau, \cdot)}_{C^\alpha((-R-\delta,R+\delta))} \delta^\alpha + \sum_{i=1}^n 2K\norm{u(i\tau, \cdot)}_{C^\alpha((-R-\delta,R+\delta))}L(\tau) \delta^\alpha \nonumber \\
    &\leq K\delta^\alpha \norm{u}_{L^\infty(0,T;C^\alpha((-R-\delta,R+\delta)))}\left( 1 + L(\tau)n \right).\label{eq:u-v0-triangle2} 
\end{align}
Finally, we recall that by \Cref{prop:L1-bressan} for a $\beta>0$ to be determined, there exists $\eps_1 > 0$ such that the $L^1$ Lipschitz constant satisfies $L(\tau) \leq C\tau^{-\beta}$. 
Substituting this, along with~\eqref{eq:def-tau} and~\eqref{eq:n-estimate}, into~\eqref{eq:u-v0-triangle2} we find 
\begin{multline} \label{eq:u-Ca-term-done} \norm{ u(t_1,\cdot) - v_0(t_1,\cdot) }_{L^1(( -R+ct_1, R-ct_1))} \\\leq K\left( \norm{u}_{L^\infty(0,T;C^\alpha((-R-\delta,R+\delta)))} \delta^\alpha + \norm{u}_{L^\infty(0,T;C^\alpha((-R-\delta,R+\delta)))}^{2 + \beta } \delta^{\beta(\alpha - 1)}\delta^{2\alpha - 1} \right).\end{multline}

For the second term on the right in~\eqref{eq:triangle-sob-gl} we note that $v_0(\tau,\cdot),\ u_{GL}(\tau,\cdot) \in \mathcal{D}^{(\tau)}$, giving us
    $$ \norm{u_{GL}(t_1,\cdot) - v_0(t_1,\cdot)}_{L^1((-R+ct_1, R-ct_1))} \leq L(\tau)\norm{u_{GL}(\tau,\cdot) - v_0(\tau,\cdot)}_{L^1((-R+c\tau, R-c\tau))}.$$
    Then using~\eqref{eq:weak-strong-obs-whole} and~\eqref{eq:Ca-L2} we find
    \beq \label{eq:GL-term-done} \norm{u_{GL}(t_1,\cdot) - v_0(t_1,\cdot)}_{L^1((-R,R))} \leq L(\tau)  K \norm{u}_{L^\infty(0,T;C^\alpha((-R-\delta,R+\delta)))} \delta^\alpha. \eeq

    Substituting~\eqref{eq:u-Ca-term-done} and~\eqref{eq:GL-term-done} into~\eqref{eq:triangle-sob-gl}, we have deduced 
    \beq \label{eq:tri-GL-done} \begin{aligned}
        \norm{u(t_1,\cdot) - u_{GL}(t_1,\cdot)}_{L^1((-R + ct_1, R - ct_1))} \leq K \Big(&\norm{u}_{L^\infty(0,T;C^\alpha((-R-\delta,R+\delta)))} \delta^\alpha \\&+ \norm{u}_{L^\infty(0,T;C^\alpha((-R-\delta,R+\delta)))}^{1+\beta}\delta^{\alpha-\beta(1-\alpha)} \\& +  \norm{u}_{L^\infty(0,T;C^\alpha((-R-\delta,R+\delta)))}^{2+\beta}\delta^{2\alpha - 1 -\beta(1 - \alpha)}\Big),
    \end{aligned}\eeq
    for any $\delta > 0$. 
    We now select the constant $\beta > 0$ to satisfy
    $$  \beta(1-\alpha) < 2\alpha - 1,$$
    which exists due to $1 > \alpha > 1/2$.
    Now, taking the limit as $\delta \downarrow 0$ we establish the right side of~\eqref{eq:tri-GL-done} converges to zero, establishing that the solutions are equal.
\end{proof}
Finally, \Cref{thm:main-obs} follows easily from \Cref{prop:entropy-conservation} and \Cref{prop:nonperiodic}. We give the proof here for convenience.
\begin{proof}[Proof of \Cref{thm:main-obs}]
Let $u \in L^\infty(0,T;C^\alpha_{\loc}(\T^1;B_{\eps_1}(0)))$ be a periodic weak solution to \eqref{cl}. By \Cref{prop:entropy-conservation}, $u$ verifies \eqref{entropic} for any $C^2$ entropy verifying \eqref{eq:entropy-def}. We note that by taking $\eps_1$ smaller if needed, there exists a $C^3$ strictly convex entropy $\eta_0 \in C^3(B_{\eps_1}(0))$ (see \Cref{lem:localstructurelemma}). So by \Cref{prop:nonperiodic}, $u$ coincides with the unique Glimm-Lax solution with initial data $u^0 \in C^\alpha_{\loc}(\R)$.
\end{proof}
\subsection{Proof of \Cref{thm:sob-ext-obs}}
Now, we prove \Cref{thm:sob-ext-obs}. We note that the result of \Cref{thm:sob-ext-obs} includes spaces $W^{s,p}$ which do not embed into $C^\alpha$ for any $\alpha > 1/2$. The proof follows almost identically to the above proof, so we omit it. The only difference is we use \Cref{cor:sob-con} instead of \Cref{prop:entropy-conservation} and \Cref{prop:sob-obs} instead of \Cref{prop:nonperiodic}. In the following we will only focus on the modifications required to prove \Cref{cor:sob-con} and \Cref{prop:sob-obs}.
\par To begin, we note 
\begin{corollary} \label{cor:sob-con}
    There exists an $\eps_1 > 0$ such that the following is true. 
    Fix $p > 1$ and $s$ satisfying~\eqref{eq:sob-s-con} and a convex entropy $\eta \in C^2(B_{\eps_1}(0))$. 
    If $u \in L^\infty(0,T;W^{s,p}(\T^1;B_{\eps_1}(0)))$ is a solution to the system~\eqref{cltorus} then the solution locally conserves entropy, i.e. ~\eqref{eq:ent-eq} holds.
\end{corollary}
\begin{proof}
    By \Cref{prop:entropy-conservation} it suffices to show that $W^{s,p}(\T^1;B_{\eps_1}(d))$ embeds into $B^\alpha_{3,\infty}(\T^1; B_{\eps_1}(d))$ for some $\alpha > 1/3$. 
    To do this, first note the Sobolev embedding
    \beq \label{eq:sob-ent-con-embedding} W^{s,p}(\T^1; B_{\eps_1}(d)) \hookrightarrow \begin{cases}
        W^{s+\frac{1}{p}- \frac{1}{3},3}(\T^1; B_{\eps_1}(d)) & p \geq 3 \\
        W^{\frac{sp}{3},3}(\T^1; B_{\eps_1}(d)) & p < 3.
    \end{cases} \eeq
    Indeed, when $p \geq 3$ we simply note that $\T^1\times \T^1$ has finite measure.
    This gives the $L^p(\T^1\times \T^1) \subset L^3(\T^1\times \T^1)$ inclusion,
    \begin{align*} 
        |u|_{W^{s+\frac{1}{p} - \frac{1}{3},3}(\T^1)} &= \left(\iint_{\T^1\times\T^1} \frac{|u(x) - u(y)|^3}{|x-y|^{1+3s+ \frac{3}{p} - 1}} \,\diff y\,\diff x\right)^{1/3} \\
        &\leq K \left(\iint_{\T^1\times\T^1} \frac{|u(x) - u(y)|^p}{|x-y|^{1+sp}} \,\diff y\,\diff x\right)^{1/p} = K |u|_{W^{s,p}(\T^1)}.
    \end{align*} 
    When $p < 3$ we simply take advantage of the fact that $\norm{u}_{L^\infty(\T^1)}$ is bounded, 
    \begin{align*}
        |u|_{W^{\frac{ps}{3} ,3}(\T^1)} &= \left(\iint_{\T^1\times\T^1} \frac{|u(x) - u(y)|^3}{|x-y|^{1+ps}} \,\diff y\,\diff x\right)^{1/3} \\
        &\leq K \norm{u}_{L^\infty(\T^1)}^{1 - \frac{p}{3}} \left(\iint_{\T^1\times\T^1} \frac{|u(x) - u(y)|^p}{|x-y|^{1+ps}} \,\diff y\,\diff x\right)^{1/3} \\
        &= K \norm{u}_{L^\infty(\T^1)}^{1 - \frac{p}{3}}|u|_{W^{s ,p}(\T^1)}^{\frac{p}{3}}.
    \end{align*}
    From this embedding we break into several cases:
    \begin{description}
        \item[\underline{Case 1, $p \in (1,2)$}] In this case $s > \frac{2}{p+1}$, so by~\eqref{eq:sob-ent-con-embedding} we have $W^{s,p}(\T^1) \hookrightarrow W^{\alpha,3}(\T^1)$ where $\alpha = \frac{2p}{3(p+1)}. $ One can then verify that $\alpha > 1/3$ for all $p \in (1,2)$. 
        
        \item[\underline{Case 2, $p\in [2,3)$}] In this range, $s > \frac{p}{2p-1}$, so by~\eqref{eq:sob-ent-con-embedding} we have $W^{s,p}(\T^1) \hookrightarrow W^{\alpha,3}(\T^1)$ where $\alpha = \frac{p^2}{3(2p-1)}.$ It is again simple to verify that $\alpha > 1/3$ for all $p \in [2,3)$. 

        \item[\underline{Case 3, $p\in [3,6)$}] Here, we note $s > \frac{p}{2p-1}$, so by~\eqref{eq:sob-ent-con-embedding} we have $W^{s,p}(\T^1) \hookrightarrow W^{\alpha,3}(\T^1)$ where $\alpha = \frac{p}{2p-1} + \frac{1}{p} - \frac{1}{3}$. One can verify that $\alpha > \frac{1}{3}$ in this case exactly when $p \in (4 - \sqrt{13} , 4+\sqrt{13})$, which contains this case.
    \end{description}
    \begin{description}
        \item[\underline{Case 4, $p\in [6,\infty)$}] In this case, we note $s > \frac{p}{2p-1} > \frac{1}{2}$ hence $sp > \frac{5}{2}$, allowing us to apply Morrey's embedding,
    $$ W^{s,p}(\T^1) \hookrightarrow C^{s - \frac{1}{p}} (\T^1). $$
        This gives $W^{s,p}(\T^1) \hookrightarrow C^\alpha(\T^1)$ where $\alpha = s-\frac{1}{p} > \frac{1}{2} - \frac{1}{6} = \frac{1}{3}. $
    \end{description}
    Hence, in all cases $W^{s,p}(\T^1) \hookrightarrow W^{\alpha,3}(\T^1)$ for some $\alpha > \frac{1}{3}$.
    We conclude recalling the well known embedding $W^{\alpha,3}(\T^1) = B^\alpha_{3,3}(\T^1) \hookrightarrow B^\alpha_{3,\infty}(\T^1)$, and the result follows from \Cref{prop:entropy-conservation}. 
\end{proof}

Also note from the proof of \Cref{cor:sob-con} that the $s,p$ mentioned in \Cref{thm:sob-ext-obs} always provide an embedding $W^{s,p} \hookrightarrow C^\alpha$ for a suitable $\alpha > 0$. 
In particular, the $u$ under consideration have traces and can be used as the weak solution in \Cref{prop:main} (and thus in \Cref{lem:trapezoidconstruction}). 
\par The next proposition is the analog of \Cref{prop:nonperiodic} in the Sobolev framework. To obtain this result, one can use the estimates from \Cref{lem:sob-for-thm-obs} for $1 < p < \infty$ and make a suitable modification to \Cref{lem:trapezoidconstruction} accounting for the total variation bound for mollified data $u_\delta$ no longer being linear in the interval length parameter $L$. 
\begin{proposition}\label{prop:sob-obs}
Let $s,p$ be as in \Cref{thm:sob-ext-obs}. Then, there exists $\eps_1 > 0$ such that the following is true. For any $T>0$, suppose $u \in L^\infty(0,T;W^{s,p}_\loc(\R;B_{\eps_1}(0)))$ is a weak solution to \eqref{cl} verifying \eqref{entropic} for at least one strictly convex entropy $\eta \in C^3(B_{\eps_1}(0))$ with initial data $u^0 \in W^{s,p}_\loc(\R;B_{\eps_1}(0))$. Then, $u$ coincides with the unique solution in $GL(\eps_1)$ with initial data $u^0$.
\end{proposition}
\begin{proof}
    Fix a $p \in (1,\infty)$ and, without loss of generality, take $s < 1$. 
    We begin by fixing our approximants as mollifications $u_\delta = u\ast \gamma_\delta$, where $\gamma_\delta$ verifies \eqref{eq:mol-assumptions}.
    In this situation, a modification of \Cref{lem:trapezoidconstruction} still holds, but for intervals $(a,b)\subset(-R,R)$ with a larger base length $L = b-a$ due to the total variation bound in \Cref{lem:sob-for-thm-obs} depending on $p$,
    $$ \norm{u_\delta}_{\text{BV}((a,b))} \leq K \norm{u}_{W^{s,p}((-R-\delta,R+\delta))}\frac{|b-a|^{1/q}}{\delta^{1-s}}, \quad \forall \ (a,b) \subset (-R,R).$$
    Specifically, recall in the proof of \Cref{lem:trapezoidconstruction} we selected $L$ such that it satisfied
    $$ \norm{u_\delta}_{\text{BV}((a,b))} < \delta_0 \quad \forall (a,b)\subset (-R,R)\text{ with }b-a < L. $$
    By the total variation estimate~\eqref{eq:TV-frac-L} of \Cref{lem:sob-for-thm-obs} we can ensure this for any $$L < K \norm{u}_{L^\infty(0,T;W^{s,p}((-R-\delta,R+\delta)))}^{-q}\delta_0^q \delta^{q(1-s)},$$
    where $q$ is the H\"older conjugate to $p$. Note that the main difference is that the total variation estimate is no longer linear in $L$.
    By this same consideration, we then modify~\eqref{eq:def-tau} and~\eqref{eq:n-estimate} to 
    \beq \label{eq:new-tau-n} \tau = \frac{L}{16c} = K \frac{\delta_0^q \delta^{q(1-s)}}{16c\norm{u}_{L^\infty(0,T;W^{s,p}((-R-\delta,R+\delta)))}^{q}}\quad \text{and} \quad n \leq \frac{t_1}{\tau} \leq K \norm{u}_{L^\infty(0,T; W^{s,p}((-R-\delta,R+\delta))) }^q \delta^{q(s-1)}.\eeq
    Now, identically as in the proof of \Cref{thm:main-obs}, we can arrive at the bound
    \beq  \label{eq:new-tri} \begin{aligned}
        \norm{u(t_1) - v_0(t_1)}_{L^1} &\leq \norm{u(t_1) - v_n(t_1)}_{L^1}+\sum_{i=1}^n \norm{v_{i-1}(t_1) - v_i(t_1)}_{L^1} \\
        &\leq K\sup_{0 \leq i \leq n} \norm{u(\tau i) - u(\tau i)\ast \gamma_\delta}_{L^2}(1 + L(\tau)n). \end{aligned}\eeq
    Now, recalling~\eqref{eq:L2-cv-frac-L},~\eqref{eq:new-tau-n}, \Cref{prop:L1-bressan}, and \Cref{lem:trapezoidconstruction} we find for any $\beta > 0$ there exists $\epsilon_1 >0$ giving the bound
    $$\sup_{0 \leq i \leq n} \norm{u(\tau i) - u(\tau i)\ast \gamma_\delta}_{L^2}(1 + L(\tau)n) \leq K \norm{u}_{L^\infty(0,T;W^{s,p}( (-R-\delta,R+\delta) ))}^{\min(1,p/2)+q+\beta q} \delta^{\min(s,sp/2)+q(s-1) + \beta q(s-1)}. $$
    
    Firstly, we consider the case $p \geq 2$. Then, when $s$ satisfies $s > \frac{q}{q+1}$, there exists a $\beta > 0$ such that $$ \beta q(1-s)< (q+1)s - q.$$
    In this case,
    $$ \norm{u(t_1) - v_0(t_1)}_{L^1} \leq K \norm{u}_{L^\infty(0,T;W^{s,p}( (-R-\delta,R+\delta) ))}\delta^s L(\tau) n  \to 0$$
    as $\delta \downarrow 0$. 
    Finally, recalling that $q$ is the H\"older conjugate to $p$ we see our condition on $s$ is 
    $$s > \frac{q}{q+1} = \frac{p}{2p - 1}. $$
    
    For $p < 2$ this requirement changes to $s > q/(q + p/2)$, in which case there exists $\beta >0$ granting
    $$ \beta q(1-s)< (q+p/2)s - q.$$
    This gives
    $$ \norm{u(t_1) - v_0(t_1)}_{L^1} \leq K \norm{u}_{L^\infty(0,T;W^{s,p}( (-R-\delta,R+\delta) ))}^{p/2}\delta^{sp/2} L(\tau) n  \to 0$$
    as $\delta \downarrow 0$.
    This condition on $s$ is exactly 
    $$s > \frac{q}{q + \frac{p}{2}} = \frac{2}{p+1}.$$ 

    Finally, we note that the proof that $$ \norm{u_{GL}(t_1) - v_0(t_1)}_{L^1} \to 0,$$ follows identically to the argument in the proof of \Cref{thm:main-obs}. This establishes \Cref{prop:sob-obs} (and so also \Cref{thm:sob-ext-obs}).
\end{proof}

In order to complete the program, it simply remains to prove \Cref{thm:cor}. The rest of the paper will be devoted to doing so. Actually, we will prove \Cref{prop:main}, from which \Cref{thm:cor} will be a simple corollary.

\section{Structure of the system and the main proposition}\label{sec:localstructure}
We start by introducing the reader to the framework used in \cite{cfk}. Recall that we consider a state space $\mathcal{G}_0$ which may be unbounded with interior $\mathcal{G}$. Fix a bounded set $\mathcal{V}_0 \subset \mathcal{G}_0$ and denote by $\mathcal{V}$ its interior. Then, we have the following set of hypotheses utilized in \cite{cfk}, \cite{MR4487515}, \cite{MR4667839}.

\textbf{Hypotheses 1.} The system \eqref{cl} verifies Hypotheses 1 relative to $\mathcal{V}_0$ if the following conditions hold. 
\begin{enumerate}[(a)]
    \item For any $u \in \mathcal{V}$, the matrix $f'(u)$ is diagonalizable with eigenvalues $\lambda_1(u) < \lambda_2(u)$. We denote by $r_i(u)$ a unit eigenvector associated to the eigenvalue $\lambda_i(u)$ for $i=1,2$. \label{hyp-a}
    \item For any $u \in \mathcal{V}$ we have $\nabla \lambda_i(u) \cdot r_i(u) > 0$ for $i=1,2$.\label{hyp-b}
    \item There exists a strictly convex function $\eta \in C(\mathcal{V}_0) \cap C^3(\mathcal{V})$ and a function $q \in C(\mathcal{V}_0) \cap C^3(\mathcal{V})$ satisfying
    $$ \nabla q=\nabla\eta f' \quad\text{on}\quad \mathcal{V}.$$ \label{hyp-c}
    \item For any state $b \in \mathcal{V}$, and any left eigenvector $\ell_i$ of $f'(b)$ corresponding to $\lambda_i(b), \ i=1,2$, the function $u \mapsto \ell_i \cdot f(u)$ is either convex or concave on $\mathcal{V}$. \label{hyp-d}
    \item There exists $L > 0$ such that $|\lambda_i(u)| \leq L$ for any $u \in \mathcal{V}$ and $i=1,2$. \label{hyp-e}
     \item \label{assum:shock-curve-param}For $u_L\in \mathcal{V}$, we denote $s\mapsto  S^1_{u_L}(s)$ the $1$-shock  curve through $u_L$ defined for $s\in[0,s_{u_L})$. See \Cref{lem:dafermoslemma} for the definition of $s_{u_L}$ (remark that possibly $s_{u_L}=+\infty$). We choose the parametrization such that $s=|u_L-S^1_{u_L}(s)|$.
     Further, we define $\sigma^1_{u_L}(s)$ to be the Rankine-Hugoniot velocity associated to the pair $(u_L,u_R)$.
     Therefore, $(u_L, S^1_{u_L}(s), \sigma^1_{u_L}(s))$ is the $1$-shock with left-hand state $u_L$ and strength $s$. Similarly, we define $s\mapsto S^2_{u_R}(s)$ to be the $2$-shock curve and $\sigma_{u_R}^2$ the Rankine-Hugoniot velocity such that $(S^2_{u_R}(s), u_R, \sigma^2_{u_R}(s))$ is the $2$-shock with right hand state $u_R$ and strength $s$. We assume that these curves are defined globally in $\mathcal{V}$ for every $u_L\in \mathcal{V}$ and $u_R\in \mathcal{V}$. \label{hyp-f}
    \item If $(u_L,u_R)$ is an entropic shock with speed $\sigma$, then $\sigma>\lambda_1(u_R)$. \label{hyp-g}
    \item (for $1$-shocks) If $(u_L,u_R)$ with $u_L \in \mathcal{V}$ is an entropic shock with speed $\sigma$ verifying,
    \begin{align*}
    \sigma\leq \lambda_1(u_L),
    \end{align*}
    then $u_R$ is in the image of $S^1_{u_L}$. That is, there exists $s_{u_R}\in[0,s_{u_L})$ such that $S^1_{u_L}(s_{u_R})=u_R$ (so $\sigma=\sigma^1_{u_L}(s_{u_R})$). \label{hyp-h}
    \item If $(u_L,u_R)$ is an entropic shock with speed $\sigma$, then $\sigma<\lambda_2(u_L).$ \label{hyp-i}
    \item (for $2$-shocks) If $(u_L,u_R)$ with $u_R \in \mathcal{V}$ is an entropic shock with speed $\sigma$ verifying,
    \begin{align*}
    \sigma\geq \lambda_2(u_R),
    \end{align*}
    then $u_L$ is in the image of $S^2_{u_R}$. That is, there exists $s_{u_L}\in[0,s_{u_R})$ such that $S^2_{u_R}(s_{u_L})=u_L$ (so $\sigma=\sigma^2_{u_R}(s_{u_L})$). \label{hyp-j}
    \item For $u_L\in \mathcal{V}$, and  for all $s>0$,  $\ds{\frac{d}{ds}\eta(u_L | S^1_{u_L}(s))}>0$ (the shock ``strengthens'' with $s$).
    Similarly, for $u_R\in \mathcal{V}$, and for all $s>0$, $\ds{\frac{d}{ds}\eta(u_R | S^2_{u_R}(s))}>0$. Moreover, for each $u_L,u_R\in \mathcal{V}$ and $s > 0$, $\frac{d}{ds}\sigma^1_{u_L}(s) < 0$ and $\frac{d}{ds}\sigma^2_{u_R}(s) > 0$. \label{hyp-k}
\end{enumerate}
Note that these hypotheses are global conditions on the set $\mathcal{V}_0$. Such global conditions are needed in \cite{cfk} precisely to obtain a stability theorem with respect to large $L^\infty$ perturbations. In this work, we will restrict to a space $\Sweak$ consisting of small $L^\infty$ perturbations of the state $d \in \mathcal{G}$. Thus, we do not need to make any of these a priori assumptions. In particular, we do not need to assume Hypotheses 1 \ref{hyp-d}, which is not generally verified for systems verifying \textbf{(GL)} even locally (see \Cref{appendix-hypd}). In past works, Hypotheses 1 \ref{hyp-d} was necessary to obtain relative entropy stability for rarefaction waves. In this paper, we develop a new method to obtain stability for rarefaction waves that completely circumvents the need for Hypotheses 1 \ref{hyp-d} (see \Cref{prop:raredissipationestimate}).

Let us give some terminology before precisely defining $\Sweak$. Firstly, we assume there exists a $C^3$ strictly convex entropy $\eta$ in a neighborhood of $d$. We remark that we will be able to show the existence of such an entropy $\eta_d$ (cf. \Cref{lem:localstructurelemma}). Secondly, we will require solutions to verify the Strong Trace Property: 
\begin{definition}\label{def:strongtrace}
Let $u \in L^\infty(\R^+ \times \R)$. We say that $u$ verifies the \emph{Strong Trace Property} if for any Lipschitz curve $t \mapsto h(t)$, there exists two bounded functions $u_-, u_+ \in L^\infty(\R^+)$ such that for any $T > 0$,
\begin{align*}
&\lim_{n\to \infty}\int_0^T\sup_{y \in (0,\frac{1}{n})}|u(t,h(t)+y)-u_+(t)|\diff t \\
&=\lim_{n\to \infty}\int_0^T\sup_{y \in (-\frac{1}{n},0)}|u(t,h(t)+y)-u_-(t)|\diff t=0.
\end{align*}
\end{definition}
With this in hand, we may define the space of non-BV perturbations for a strictly convex entropy $\eta$
\begin{align*}
\Sweak:=\Big\{u \in L^\infty(([0,T) \times \R);B_\eps(d)) \,\Big|\,&\text{weak solution to \eqref{cl} and \eqref{entropic} (with respect to $\eta$)} \\
&\text{verifying \Cref{def:strongtrace}}\Big\}.
\end{align*}
Note that this space has a smallness condition in $L^\infty$, but contains solutions which may have large or even infinite total variation. Our main proposition is the following:
\begin{proposition}\label{prop:main}
Assume that the system \eqref{cl} verifies \textbf{(GL)}. There exists $\eps > 0$ such that the following is true. Fix $R,T > 0$. Let $\delta_0 > 0$ be sufficiently small such that Theorem 1 in \cite{MR1367356} applies, which shows the existence of a small-BV semigroup of limits of front-tracking approximations. Let $u \in \Sweak$ have initial data $u^0$. Consider initial data $v^0$ such that $||v^0||_{BV}+||v^0-d||_{L^\infty(\R)} \leq \delta_0$ and let $\mathcal{S}_tv^0$ be the Standard Riemann Semigroup solution. Then we have the stability estimate
\begin{align}\label{L1L2}
||u(t, \cdot)-\mathcal{S}_tv^0||_{L^1((-R,R))} \leq K\sqrt{R + T}\norm{u^0-v^0}_{L^2((-R-ct, R+ct))},
\end{align}
for some $K = K(d, f)> 0$, where the constant $c > 0$ is the speed of information (see \Cref{rem:soi}). Further, by interpolation we obtain
\begin{align}\label{L2holder}
||u(t, \cdot)-\mathcal{S}_tv^0||_{L^2((-R,R))} \leq K\sqrt[4]{R + T}\sqrt{||u^0-v^0||_{L^2((-R-ct, R+ct))}}.
\end{align}
\end{proposition}
We will show existence of an entropy $\eta_d$ for which $\Sweak$ may be defined. In particular, for this $\eta_d$, $\Sweak$ will include the Glimm--Lax solutions $GL(\eps_1)$ for $\eps_1$ sufficiently small as they obey the $L^\infty$ estimate \eqref{smalllinfty}, they verify \Cref{def:strongtrace} (as they enter BV for positive times), and they are entropic for any convex entropy. Thus, \Cref{thm:cor} follows directly from \Cref{prop:main}.

We remark that \Cref{prop:main} has a simple corollary, establishing uniqueness of BV solutions in the Glimm--Lax class $GL(\eps_1)$ among solutions in the weaker class $\mathcal{S}_{\text{weak}}^{d,C_0\sqrt{\eps_1}}$.
\begin{corollary}
    There exists $\eps_1 > 0$ such that the following is true. 
    Let $v\in GL(\eps_1)$ be the solution to~\eqref{cl} with initial data $v^0 \in \textrm{BV}_{\loc}(\R;B_{\eps_1}(d)).$ Then $v$ is the unique solution within the class $\mathcal{S}_{\textrm{weak}}^{d,C_0\sqrt{\eps_1}}$ with initial data $v^0$.
\end{corollary}
\begin{proof}
    We let $\eps_1 > 0$ be sufficiently small such that $C_0\sqrt{\eps_1} < \min(\eps,\delta_0/2)$, where $\eps$ and $\delta_0$ are the constants from \Cref{prop:main}.
    For any $R > 0$ we fix the interval $I = (-R,R)$. 
    Since $v^0$ has finite BV norm on $I$ and no discontinuities larger than $\delta_0$, there exists an $L > 0$ such that for any $(a,b) \subset (-R,R)$ with $b-a < L$ we have $\norm{v(0,\cdot)}_{BV((a,b))} < \delta_0$. 
    Using this, we can decompose $I$ into overlapping intervals of size at most $L$ and, by the same reasoning as in \Cref{lem:trapezoidconstruction}, we have for some $T_1 \sim L$ the estimate 
    $$ \norm{u(t,\cdot) - v(t,\cdot)}_{L^2( (-R + ct, R -ct))} \leq 2K \sqrt{\norm{u(0,\cdot) - v^0}_{L^2(I)}}, $$
    for all $t < T_1$ and $u$ in $\mathcal{S}^{d,\eps}_\text{weak}$.
    Next, by~\eqref{tv1/t} we have a bound on the distribution of total variation of $v(T_1,\cdot)$, giving that $v(T_1,\cdot) \in BV_\loc(\R)$ and having no discontinuities larger than $\delta_0$. From this we can repeat the argument at the initial time, obtaining a H\"older-1/4 estimate for times $t \in [T_1,T_2]$ for some $T_2 > T_1$. 
    Inducting, we have the stability estimate
    $$ \norm{u(t,\cdot) - v(t,\cdot)}_{L^2( (-R + ct, R -ct) )} \leq 2K \norm{u(0,\cdot) - v^0}_{L^2( I)}^{1/2^{m(v_0,t)}}, $$
    where $m(v_0,t)$ is the number of iterations $m$ applied to have $t \in [T_{m-1}, T_m]$ given the initial data $v_0$. 
    We remark that for some finite $m$ $T_m > R/c$ due to the estimate~\eqref{tv1/t} spreading out the total variation as time increases, covering the information cone.
    Since $R>0$ was taken to be arbitrary, this establishes the uniqueness of $v$ among strong trace solutions.
\end{proof}

\par The rest of the paper will be devoted to proving \Cref{prop:main}. The strategy we employ is to approximate $\mathcal{S}_t v^0$ by \emph{shifted front tracking approximations}, which we construct in \Cref{sec:ft}. These differ from traditional front tracking approximations in that our approximate shocks travel at velocities determined by our weak solution $u$, rather than the Rankine--Hugoniot velocity. In section \Cref{relent} we show a weighted $L^2$-contraction between weak solutions and (shifted) solutions to the Riemann problem. In \Cref{sec:mainproof} we combine this with traditional front tracking procedures to establish an approximate form of the bound~\eqref{L1L2}. This proof involves a delicate construction of a weight in \Cref{a} which can handle a front tracking solution with many discontinuities. We then recover \Cref{prop:main} by taking our front tracking parameter to zero.

Our first key lemma says that there exists a neighborhood $W$ of the state $d$ such that the system \eqref{cl} verifies Hypotheses 1 relative to $W$ (except for Hypotheses 1 \ref{hyp-d}). 
\begin{lemma}\label{lem:localstructurelemma}
Let $\eqref{cl}$ verify \textbf{(GL)}. Fix a state $d \in \mathcal{G}$. Then, there exists a neighborhood $W$ containing $d$ such that $\eqref{cl}$ verifies all Hypotheses 1 (except for \ref{hyp-d}) relative to $W$. 
\end{lemma}
\begin{remark}
We remark that \Cref{lem:localstructurelemma} is sharp in the sense that there are systems verifying \textbf{(GL)} that do not verify {Hypotheses 1} (d) relative to any neighborhood of a given state. See \Cref{appendix-hypd} for an example.
\end{remark}
A proof of \Cref{lem:localstructurelemma} will be given in Appendix \ref{appendixB}. This lemma allows us to use the $a$-contraction method developed in \cite{MR4667839}, \cite{cfk} to obtain $L^2$-stability for small perturbations of small shocks near $d$ (see \Cref{lem:shockdissipationestimate}).

\section{Method of relative entropy}\label{relent}
We use the relative entropy method introduced by Dafermos \cite{MR546634} and DiPerna \cite{MR523630}. Recall that we are assuming the existence of a $C^3$ strictly convex entropy in a neighborhood $W$ of $d$. We define the associated pseudo-distance for any $a,b \in W \times W$,
\begin{equation} \label{def:rel-ent}
\eta(a|b)=\eta(a)-\eta(b)-\nabla \eta(b)(a-b).
\end{equation}
The quantity $\eta(a|b)$ is called the \emph{relative entropy} of $a$ with respect to $b$, and is equivalent to the square of their difference. We also define the \emph{relative entropy flux} for $a,b \in W \times W$,
\begin{equation} \label{def:rel-ent-flux}
q(a;b)=q(a)-q(b)-\nabla \eta(b)(f(a)-f(b)),
\end{equation}
and the \emph{relative flux} for $a,b \in W \times W$,
\begin{equation}\label{def:rel-flux}
f(a|b)=f(a)-f(b)-f'(b)(f(a)-f(b)).
\end{equation}
If $u$ is a weak solution to \eqref{cl} and \eqref{entropic}, then $u$ verifies the following family of entropy inequalities for all $b \in W$,
\begin{align}\label{relentb}
(\eta(u|b))_t+(q(u;b))_x \leq 0.
\end{align}
In this way, we see that the single entropy inequality \eqref{entropic} actually induces an infinite family of entropy inequalities \eqref{relentb}. It is this fact which allows the $L^2$-theory to grant stability with respect to non-BV perturbations, and shows that the $L^2$-theory can be viewed as a Kru\v{z}kov-like theory for systems. 

\subsection{Quantitative dissipation estimates for small shocks}
For the study of discontinuous solutions, the relative entropy requires an intricate modification by considering a spatially weighted relative entropy along with a shift ($a$-contraction). This was introduced by Kang-Vasseur \cite{kang2016criteria}. Given a fixed entropic shock $(u_L, u_R, \sigma)$ and solution $u \in L^\infty( \R^+\times \R;W)$, consider the pseudo-distance
\begin{align}\label{pseudodistanceshock}
E_t(u;a_1, a_2, h):=a_1 \int_{-\infty}^{h(t)}\eta(u(t,x)|u_L)\diff x+a_2 \int_{h(t)}^\infty \eta(u(t,x)|u_R)\diff x,
\end{align}
where $a_1, a_2$ are constant weights to be determined and $h(t)$ is a Lipschitz shift function. By the following lemma, we see that if $E_t$ is decreasing then we have $L^2$-control on the distance between $u$ and the shifted shock $(u_L, u_R, \dot{h})$. 
\begin{lemma}\label{relentisl2}
There exists a constant $K$ such that for all $(u,v) \in W \times W$, we have
\begin{align*}
\frac{1}{K}|u-v|^2 \leq \eta(u|v) \leq K|u-v|^2.
\end{align*}
\end{lemma}
To control $E_t$ we compute its time derivative for a solution $u$ with the Strong Trace Property (i.e. satisfying \Cref{def:strongtrace}). 
Integrating \eqref{relentb} once with $b=u_L$ on $\{(s,x) \,|\, s\in [0,t]\text{ and } x < h(t))\}$ and again with $b=u_R$ on $\{(s,x) \,|\, s\in [0,t]\text{ and } x > h(t))\}$, adding the results and using the Strong Trace Property we arrive at the following bound, 
\begin{equation} \label{eq:rel-ent-diss} \begin{aligned}
    {\frac{d}{dt}}E_t \leq \hphantom{-}a_2&\left[q(u(h(t)+,t);u_R)-\dot{h}(t) \ \eta(u(h(t)+,t)|u_R)\right] \\
             -a_1&\left[q(u(h(t)-,t);u_L)-\dot{h}(t) \ \eta(u(h(t)-,t)|u_L)\right].
    \end{aligned}
\end{equation}
The following lemma provides weights $a_1,a_2$ and a shift $h$ for which the right side of~\eqref{eq:rel-ent-diss} is negative, establishing $E_t$ is decreasing.
\begin{proposition}[\protect{Existence of the shift function for small shocks with quantitative dissipation estimate}]\label{lem:shockdissipationestimate}
    Consider a strictly hyperbolic system \eqref{cl} with flux $f$ satisfying the condition \textbf{(GL)}. Fix $d\in \mathcal{G}_0$, and let $W$ be the neighborhood of $d$ constructed in \Cref{lem:localstructurelemma}. 
    Then there exists $\hat \lambda, C_1, K, \eps > 0$ such that the following holds:
        Consider any $1$-shock or $2$-shock $(u_L,u_R,\sigma_{LR})$ satisfying 
        \begin{align*}
        |u_L - d| + |u_R - d| \leq \epsilon.
        \end{align*}
        Letting $s_0$ be the arclength between $u_L$ and $u_R$ along the curve $S_{u_L}^1$, if $a_1,a_2$ satisfy
        \begin{align*}
            1 + \frac{C_1s_0}{2} &\leq \frac{a_1}{a_2} \leq 1+2C_1s_0 \quad\quad \text{if $(u_L,u_R)$ is a 1-shock} \\
            1 - 2C_1s_0 &\leq \frac{a_1}{a_2} \leq 1-\frac{C_1s_0}{2} \quad\quad\, \text{if $(u_L,u_R)$ is a 2-shock}, 
        \end{align*}
        then for any $u \in \Sweak$, $T > 0$, $T_{start} \in [0,T]$, and $x_0 \in \R$  there exists a Lipschitz map $h:[T_{start},T] \to \R$ such that $h(T_{start}) = x_0$ and we have the following bound on our dissipation functional
        \begin{equation}\begin{aligned} 
            &a_2\left[q(u(h(t)+,t);u_R)-\dot{h}(t) \ \eta(u(h(t)+,t)|u_R)\right]\\ \label{diss:shock}
            &\qquad\qquad  -a_1\left[q(u(h(t)-,t);u_L)-\dot{h}(t) \ \eta(u(h(t)-,t)|u_L)\right] \leq -a_2 Ks_0(\dot h(t) - \sigma_{LR})^2,
        \end{aligned}\end{equation}
        for almost all $t \in [T_{start},T]$. 
        Furthermore, if $(u_L,u_R)$ is a 1-shock, then for almost all $t \in [T_{start},T]$
        \begin{align*}
            -\frac{\hat\lambda}{2} \leq \dot h(t) \leq \sup_{v \in B_{2\epsilon}} \lambda_1(v).
        \end{align*}
        Similarly, if $(u_L,u_R)$ is a 2-shock, then for almost all $t \in [T_{start},T]$
        \begin{align*}
            \inf_{v \in B_{2\epsilon}} \lambda_2(v) \leq \dot h(t) \leq \frac{\hat\lambda}{2}.
        \end{align*}
\end{proposition} 
\begin{proof}
The proof follows directly from \Cref{lem:localstructurelemma} and \cite[Proposition 3.1]{cfk} (taking $\mathcal{V}_0=W$ in their Proposition 3.1), where we note that Hypotheses 1 are exactly \cite[Assumptions 1]{cfk}, and that Hypotheses 1 \ref{hyp-d} (Assumption 1 (d)) is not used in the proof of Proposition 3.1 therein.
\end{proof}

\subsection{A weighted relative entropy contraction for rarefaction waves}
In this section, we prove a novel weighted relative entropy contraction for rarefaction waves (Lemma \ref{rarecontract}). For inviscid conservation laws, this is the first result in the relative entropy method for continuous solutions that are not Lipschitz (cf. \cite{jiayun}). Before we state exactly the result, let us recall some basic facts about rarefactions (following \cite{MR1625845}). 
\par Recall that for $u \in W$ and $i=1,2$, there exists a curve $R^i_u:\R \to \mathcal{G}_0$ verifying the ODE
\begin{equation} 
\left\{\begin{aligned}
\ds{\frac{d}{ds}}R^i_u(s)&=r_i(R^i_u(s)), \quad s \in \R, \\
R^i_u(0)&=u. 
\end{aligned} \right. \end{equation}
We call $R^i_u$ the \emph{$i^{th}$ rarefaction curve at $u$}. Given two states $u_L, u_R$ such that $u_R=R^i_{u_L}(s_0)$ for some $s_0 > 0$, a rarefaction wave connecting $u_L$ and $u_R$ can be composed in the following way. Firstly, define
\begin{align*}
f_i(s):=\int_0^s\lambda_i(R^i_{u_L}(t)) \diff t, \quad \forall \ s \in \R.
\end{align*}
Then, consider the Riemann problem for the scalar conservation law
\begin{align}\label{scl}
\begin{aligned}
y_t+[f_i(y)]_x&=0, \\
y(0,x)&=\begin{cases}
0 & x < 0, \\
s_0 & x > 0,
\end{cases}
\end{aligned}
\end{align}
where we recall $R^i_{u_L}(0)=u_L$ and $R^i_{u_L}(s_0)=u_R$. The scalar problem~\eqref{scl} has the unique entropy solution
\begin{align*}
y(t,x)=\begin{cases}
0 & \frac{x}{t} <f_i'(0) ,\\
[f_i']^{-1} (\frac{x}{t}) & f_i'(0) < \frac{x}{t} < f_i'(s_0), \\
s_0 & \frac{x}{t} > f_i'(s_0).
\end{cases}
\end{align*}
It then follows that the rarefaction for $\eqref{cl}$ is 
\begin{equation} \label{eq:u-bar-def}
\overline{u}(t,x)=R^i_{u_L}(y(t,x)).
\end{equation}
Without confusion, we may refer to $\overline{u}$ and $y$ as functions of merely one argument as they are self-similar along the lines $\frac{x}{t} \equiv c$. 

The main proposition for use in the front tracking scheme later is the following:
\begin{proposition}\label{prop:raredissipationestimate}
Consider a strictly hyperbolic system \eqref{cl} verifying \textbf{(GL)}. Let $d \in \mathcal{G}$, and let $W$ be the neighborhood of $d$ constructed in \Cref{lem:localstructurelemma}. Then, there exists $K, C_2, \eps > 0$ such that the following holds. Consider any two states $(u_L, u_R)$ satisfying
\begin{align*}
    |u_L - d| + |u_R - d| \leq \eps,
\end{align*}
where $u_R=R^i_{u_L}(s_0)$ for some $s_0 > 0$ and $i=1,2$ (so that $u_L, u_R$ are connected by a $1$-or-$2$-rarefaction $\overline{u}$). Denote
    \begin{align*}
    \overline{\sigma}=|v_L-v_R|+\sup_{t,x}|u_L-\overline{u}(y(t,x))|,
    \end{align*}
    where $v_L=\lambda_i(u_L), v_R=\lambda_i(u_R)$, and $y$ solves the scalar problem~\eqref{scl}.
        Then, if $a_1, a_2$ satisfy
        \begin{align}
            1 + \frac{C_2s_0}{2} &\leq \frac{a_1}{a_2} \leq 1+2C_2s_0, \quad\quad \text{if $u_R=R^1_{u_L}(s_0)$,} \label{control_a_one_rare}\\
            1 - 2C_2s_0 &\leq \frac{a_1}{a_2} \leq 1-\frac{C_2s_0}{2}, \quad\quad \text{if $u_R=R^2_{u_L}(s_0)$}, \label{control_a_two_rare}
        \end{align}
        then for any $u \in \Sweak, v \in I[v_L,v_R]$ (where $I[v_L,v_R]:=[\min(v_L,v_R), \max(v_L,v_R)]$), and $t > 0$, we have
\begin{align*}
&\int_0^t \Bigl\{a_2 \left(q(u(\tau,\tau v+);u_R)-v\eta(u(\tau,v\tau+)|u_R)\right) \\
&-a_1 \left(q(u(\tau, \tau v-);u_L)-v\eta(u(\tau,v\tau-)|u_L)\right)\Bigr\} \diff \tau \leq a_2\left(K\overline{\sigma}^2t+K\overline{\sigma}|\mu_u|(\mathcal{C}_{v_L, v_R})\right),
\end{align*}
where $\mu_u$ is the locally finite non-positive entropy dissipation  measure $[\partial_t \eta(u)+\partial_x q(u)]$ and 
\begin{equation} \label{eq:rare-fan} \mathcal{C}_{v_L,v_R}:=\{(\tau,x)\,|\,0 \leq \tau \leq t, v_L\tau \leq x \leq v_R \tau\},\end{equation} is the rarefaction fan.
\end{proposition}
\begin{remark}\label{same}
Note that in the proofs of \Cref{prop:raredissipationestimate}  and \Cref{lem:shockdissipationestimate} (which is found in \cite[Section 8]{cfk}), we may take $C_1, C_2$ arbitrarily large. In particular, we may assume $C_1=C_2$. 
\end{remark}
\begin{remark}
    The fact that the entropy dissipation $\mu_u$ is a locally finite measure follows from the Riesz-Markov representation theorem. 
\end{remark}
We prove \Cref{prop:raredissipationestimate} in \Cref{sec:raredissipationestimatepfsec}. Before we can do so, we need an important lemma in which we obtain an exact weighted relative entropy contraction for the true rarefaction wave. The strategy is as follows. Drawing inspiration from the quantity \eqref{pseudodistanceshock}, let us define
\begin{align}\label{pseudodistancerare}
D_t:=\int_{-\infty}^\infty a(t,x)\eta(u(t,x)|\overline{u}(t,x))\diff x,
\end{align}
where here $u$ solves \eqref{cl} and \eqref{entropic}, and $\overline{u}(t,x)$ is the rarefaction wave. In this setting, we do not have any shift $h$, and the weight $a(t,x)$ will no longer be piecewise constant (cf.~\eqref{pseudodistanceshock}). It is again our goal to estimate $\ds{\frac{d}{dt}}D_t$.
\begin{lemma}\label{rarecontract}
Consider a strictly hyperbolic system \eqref{cl} verifying \textbf{(GL)}. Let $d \in \mathcal{G}$ and $W$ be the neighborhood of $d$ constructed in \Cref{lem:localstructurelemma}. Then, for any $C_2$ sufficiently large, there exists $\eps > 0$ such that the following holds. 
Let $(u_L,u_R)$ be two states with $u_R = R_{u_L}^i(s_0)$ for a $s_0 > 0$ satisfying  
\begin{align*}
    \sup_{s \in [0,s_0]} |R_{u_L}^i(s) - d| \leq \eps.
\end{align*}
For any constant $C$  satisfying
\begin{equation}\label{eq:C-range-cont} C \in [C_2/4, 4C_2],\end{equation}
we fix a weight of the form
\begin{equation} \label{eq:a-def-cont-rare}
a(t,x):=\exp\left[(-1)^{i}\ C\ y(t,x)\right],
\end{equation}
where $y$ solves~\eqref{scl}.
Then, for any $u \in \Sweak$, we have 
\begin{align}\label{rarecontractdistribution}
\partial_t\left[a(y(t,x))\eta(u(t,x)|\overline{u}(t,x))\right]+\partial_x\left[a(y(t,x))q(u(t,x);\overline{u}(t,x))\right] \leq a(y)\mu_u,
\end{align}
in the sense of distributions, i.e. in $\mathcal{D}'(\R^+ \times \R)$, where $\mu_u$ is the locally finite non-positive entropy dissipation  measure $[\partial_t \eta(u)+\partial_x q(u)]$.

\end{lemma}
\begin{proof}
We note that by the change of variables $ x \mapsto -x $ the system~\eqref{cl} transforms into another conservation law with flux $-f$, mapping $2$-rarefactions to $1$-rarefactions. This reflection also switches the weight $a(t,x)$ for a $2$-rarefaction to a corresponding weight for $1$-rarefactions. Therefore, without loss of generality, it suffices to just consider $1$-rarefactions.
\par We recall from \cite[Equation (4.7)]{MR4487515} that any weak solution $u$ verifying \eqref{entropic} satisfies the distributional inequality
\begin{equation} \label{eq:unweighted-rel-rare}
    \partial_t \eta(u|\overline u) + \partial_x q(u;\overline u) + \partial_x \{ \nabla \eta(\overline u) \} f(u|\overline u) \leq \mu_u,
\end{equation} 
where $\overline u(t,x)$ is the $1$-rarefaction wave connecting states $u_L$, $u_R$ as defined in~\eqref{eq:u-bar-def} (a non-sharp form of \eqref{eq:unweighted-rel-rare} inequality is classical, see \cite[Equation (5.2.10)]{MR3468916}).
Multiplying both sides of~\eqref{eq:unweighted-rel-rare} by the positive function $ a(t,x) $ we find 
\begin{equation} \label{eq:rel-entr-rare-weighted}
    \begin{aligned}
        \partial_t[a\eta(u|\overline u)] + \partial_x[a q(u;\overline u)] + a\partial_x \{ \nabla \eta(\overline u) \} f(u|\overline u) - \partial_x(a) q(u;\overline u) - \partial_t(a) \eta(u|\overline u) \leq a\mu_u,
    \end{aligned}
\end{equation}
which holds distributionally in $\mathcal{D}'(\R^+ \times \R)$.
We note that for any fixed state $u \in W$,
\begin{equation} \label{eq:crit-terms-exp}\begin{aligned}
    a\partial_x \{ \nabla \eta(\overline u)) \} &f(u|\overline u) - \partial_x(a) q(u;\overline u) - \partial_t(a) \eta(u|\overline u) \\
    &= \frac{a}{t} \left([f_1']^{-1}\right)'(x/t) \left[Cq(u; \overline u)-C\frac{x}{t} \eta(u| \overline u) + r_1(\overline u)^t\nabla^2 \eta(\overline u)f(u|\overline u)\right]\mathbbm{1}_{\mathcal{C}_{v_L,v_R}}(t,x),
\end{aligned} \end{equation}
where $\mathcal{C}_{v_L,v_R}$ is the fan defined in~\eqref{eq:rare-fan} and our weight $a$ is as in~\eqref{eq:a-def-cont-rare} for some $C > 0$ to be determined.
We now observe that by construction $\overline u$ satisfies $x/t = \lambda_1(\overline u(t,x))$ for all $(t,x) \in \mathcal{C}_{v_L,v_R}$.
Taylor-expanding $q(u;\overline{u}), \eta(u|\overline{u}), \& f(u|\overline{u})$ in $u$ around $u = \overline u$ using the relation~\eqref{eq:entropy-def},
\begin{align*}
    &Cq(u; \overline u)-C\lambda_1(\overline u) \eta(u| \overline u) +  r_1(\overline u)^t\nabla^2 \eta(\overline u)f(u|\overline u) \\
    &\qquad\geq \frac{C}{2} \nabla^2 \eta(\overline u)[ f'(\overline u) - \lambda_1(\overline u) I](u-\overline u) \cdot (u-\overline u) + \frac{1}{2}r_1(\overline u)^t \nabla^2 \eta(\overline u) f''(\overline u)(u-\overline u, u-\overline u)- KC|u-\overline u|^3,
\end{align*}
where the constant $K$ depends only on $f$ and the set $W$ from Lemma \ref{lem:localstructurelemma} and $f''(\overline u)(v,v)$ is the second Fr\'{e}chet derivative of $f$. 
Writing 
$$ u - \overline u = \sum_{i=1}^2 a_i r_i(\overline u),$$
by hyperbolicity and the fact that $\nabla^2 \eta f' = f'\nabla^2 \eta$, there exists a constant $K_1 > 0$ independent of $\overline u \in W$ such that 
\begin{equation} \label{eq:lower-preyoung}
\begin{aligned}
    &Cq(u; \overline u)-C\lambda_1(\overline u) \eta(u| \overline u) +  r_1(\overline u)^t\nabla^2 \eta(\overline u))f(u|\overline u)\\
    &\qquad\geq C K_1a_2^2 + \sum_{1 \leq i,j \leq 2} \frac{a_ia_j}{2}r_1(\overline u)^t\nabla^2 \eta(\overline u)f''(\overline u)(r_i(\overline u),r_j(\overline u)) - KC|u-\overline u|^3.
\end{aligned}
\end{equation}
For the second sum we note that the coefficient to the $i=j=1$ term is strictly positive. Indeed, taking the derivative of the equation $f'\ r_1 = \lambda_1 r_1$ in the direction $r_1$ and solving for $f''(r_1, r_1)$, one discovers 
$$ f''(r_1,r_1) = (\nabla \lambda_1 \cdot r_1) r_1 + (\lambda_1 I - f') (r_1'r_1). $$
Evaluating the above relationship at $\bar u$ and multiplying by the left eigenvector $r_1(\bar u)^t \nabla^2 \eta(\bar u)$ which is parallel to $\ell_1$ we uncover 
$$ r_1(\overline u)^t\nabla^2 \eta(\overline u) f''(r_1(\overline u),r_1(\overline u)) = (r_1(\overline u)^t \nabla^2 \eta(\overline u)r_1(\overline u)) (\nabla \lambda_1(\overline u)\cdot r_1(\overline u)) > 0,$$
where the sign follows from strict convexity of $\eta$ and genuine nonlinearity.
In fact, since we are restricted to $\overline u \in W$ there is a positive lower bound on this quantity, uniform over $\overline u \in W$, 
while the remaining terms are bounded below by Young's with $\epsilon$ (with $\epsilon = \alpha$) like so:
$$\frac{1}{2}\sum_{\substack{1 \leq i,j \leq 2 \\ ij > 1}} a_ia_j r_1(\overline u)^t\nabla^2 \eta(\overline u) f''(r_i(\overline u),r_j(\overline u)) \geq K_2 \left[- \alpha a_1^2 - \frac{1}{\alpha}a_2^2 - a_2^2\right],$$
where $K_2 > 0$ depends on $f$ and $W$.
Substituting this into~\eqref{eq:lower-preyoung} and using our bounds on $C$ from~\eqref{eq:C-range-cont} we find 
\begin{equation} \label{eq:lower-postyoung}
\begin{aligned}
    &Cq(u; \overline u)-C\lambda_1(\overline u) \eta(u| \overline u) +  r_1(\overline u)^t\nabla^2 \eta(\overline u))f(u|\overline u)\\
    &\qquad\geq \left(\frac{C_2}{4} K_1 -K_2\left(\frac{1}{\alpha} + 1\right)\right) a_2^2 + \left((r_1(\overline u)^t \nabla^2 \eta(\overline u) r_1(\overline u)\frac{\nabla \lambda_1(\overline u) \cdot r_1(\overline u)}{2} - K_2\alpha\right) a_1^2 - 4KC_2|u-\overline u|^3 \\
    &\qquad\geq K_3|u-\overline u|^2 - 4KC_2|u-\overline u|^3,
\end{aligned}
\end{equation}
for a constant $K_3 > 0$, where the final line follows by choosing $\alpha > 0$ sufficiently small and $C_2$ sufficiently large. 
From~\eqref{eq:lower-postyoung} it is clear that if $|u-\overline u| \leq K_3/(8KC_2)$ we find 
\begin{equation} \label{eq:positivity-of-extras}
    Cq(u; \overline u)-C\lambda_1(\overline u) \eta(u| \overline u) +  r_1(\overline u)^t\nabla^2 \eta(\overline u))f(u|\overline u)\geq \frac{1}{2}K_3 |u-\overline u|^2.
\end{equation}
At this point, we select 
\begin{align}\label{epsval}
\epsilon = K_3/(16KC_2),
\end{align}
to be the value of $\epsilon$ in the statement of Lemma \ref{rarecontract}. We note by triangle inequality $|u - \overline u| \leq |u-d| + |\overline u - d| <  2\epsilon $ for all states $u$ attained by solutions in $\Sweak$ and $\overline u$ attained by the rarefaction, hence verifying the lower bound~\eqref{eq:positivity-of-extras}.

Furthermore, the term
$$ ([f_1']^{-1} )' = \frac{1}{\nabla \lambda_1(\overline u)\cdot r_1(\overline u)},$$
is positive within $W$ due to genuine nonlinearity. Substituting this all into~\eqref{eq:crit-terms-exp} and the inequality~\eqref{eq:rel-entr-rare-weighted}, we finally deduce
\begin{equation} \label{eq:dist-weighted-rare-ineq} \partial_t[a\eta(u|\overline u)] + \partial_x[a q(u;\overline u)] \leq a \mu_u \end{equation}
distributionally in $\mathcal{D}'(\R^+ \times \R)$ when $ \sup_{s \in [0,s_0]} \left| R_{u_L}^1(s) - d \right| < \epsilon$ and $u \in \Sweak$.
\end{proof}
\begin{remark}
Note that the same computation would hold for rarefaction waves in extremal families of an $n \times n$ system with $n > 2$. Also note that taking the derivative of $D_t$ in \eqref{pseudodistancerare} and using the inequality \eqref{rarecontractdistribution} gives a weighted relative entropy contraction for the rarefaction wave among small perturbations that \emph{may not even verify the Strong Trace Property} (\Cref{def:strongtrace}) (where the smallness is quantified by \eqref{epsval}).
\end{remark}

Before continuing, we note the relative entropy and relative entropy flux are Lipschitz in their second entry. 
\begin{lemma}[{\cite[Lemma 7.2]{MR4487515}}] \label{lem:rel-quantities-properties}
    There exists a constant $K$ such that for all $(a,b_1,b_2) \in W^3$ we have the Lipschitz bounds
    \begin{align*}
        |q(a;b_1) - q(a;b_2)| &\leq K|b_1 - b_2|, \\
        |\eta(a|b_1) - \eta (a|b_2)| &\leq K|b_1 - b_2|.
    \end{align*}
\end{lemma}

\subsection{Proof of \Cref{prop:raredissipationestimate}}\label{sec:raredissipationestimatepfsec}
Now, we give control on the error due to the approximation of the rarefaction in the front-tracking method by proving \Cref{prop:raredissipationestimate}.
\begin{proof}
Firstly, take $C_2$ sufficiently large and $\eps$ sufficiently small so that Lemma \ref{rarecontract} holds. For $a,b \in \R$, define the two-parameter family of fans:
\begin{align*}
\mathcal{C}_{a,b}:=\{(x,\tau)\, |\,0 \leq \tau \leq t, a\tau \leq x \leq b\tau\}.
\end{align*}
Testing the distributional inequality \eqref{rarecontractdistribution} with the function
\begin{align}\label{test}
\psi(\tau,x)=\omega(\tau)\chi(\tau,x),
\end{align}
where
\begin{align*}
\omega(\tau)=\begin{cases}
0 & \tau \leq \frac{\delta}{2}, \\
2\delta^{-1}(\tau-\frac{\delta}{2}) & \frac{\delta}{2} \leq \tau \leq \delta, \\
1 & \delta \leq \tau <t-\delta, \\
\delta^{-1}(t-\tau) & t-\delta \leq \tau \leq t, \\
0 & \tau \geq t,
\end{cases}
\end{align*}
and
\begin{align*}
\chi(\tau,x)=\begin{cases}
1 & v_L\tau+\delta \leq x \leq v_R \tau-\delta, \\
\delta^{-1}(v_R\tau-x) & v_R\tau-\delta \leq x \leq v_R\tau, \\
\delta^{-1}(x-v_L\tau) & v_L\tau \leq x \leq v_L\tau+\delta, \\
0 & \text{otherwise},
\end{cases}    
\end{align*}
taking the limit $\delta \downarrow 0$ and using the Strong Trace Property, we obtain
\beq \label{discineq1}\begin{aligned}
0 \leq&-\int_{v_Lt}^{v_Rt}a(y(t,x))\eta(u(t,x)|\overline{u}(t,x))\diff x+\int_{\mathcal{C}_{v_L, v_R}}a(y)\diff \mu_u(t,x)  \\
&+\int_0^ta(y(v_R))\left[-q(u(\tau,v_R\tau-);u_R) + v_R\eta(u(\tau, v_R\tau-)|u_R)\right]\diff \tau  \\
&-\int_0^ta(y(v_L))\left[-q(u(\tau,v_L\tau+);u_L)+v_L\eta(u(\tau, v_L\tau+)|u_L)\right]\diff \tau,  \\
\end{aligned} \eeq
where $a$ is defined in \eqref{eq:a-def-cont-rare}. Further, by the definition of the relative entropy~\eqref{def:rel-ent}, entropy flux~\eqref{def:rel-ent-flux}, and $u$ weakly solving~\eqref{cl} we have the equalities
\begin{align}\label{entequality}
\partial_t\left[\eta(u(t,x)|u_i)-\eta(u(t,x))\right]+\partial_x\left[q(u(t,x);u_i)-q(u(t,x))\right]=0, \indent  i=L,R.
\end{align}
Testing the $i=L$ case of \eqref{entequality} with a similar test function to \eqref{test} (adapted to the fan $\mathcal{C}_{v_L,v}$) and multiplying the result by $a(y(v_{L}))$, we obtain
\beq\label{disceq2} \begin{aligned}
-\int_0^ta(y(v_L))[-q(u(\tau,v_L\tau+);u_L)+&v_L\eta(u(\tau, v_L\tau+)|u_L)]\diff \tau\\ 
=&\int_{v_Lt}^{vt}a(y(v_L))\eta(u(t,x)|u_L)\diff x-\int_{\mathcal{C}_{v_L,v}}a(y(v_L))\diff \mu_u  \\
&-\int_0^ta(y(v_L))[v\eta(u\tau,v\tau-|u_L)-q(u(\tau,v\tau-;u_L))] \diff \tau.
\end{aligned} \eeq
In the same manner, testing the $i=R$ case of \eqref{entequality} with a similar test function to \eqref{test} (adapted to the fan $\mathcal{C}_{v,v_R}$) and multiplying the result by $a(y(v_{R}))$, we obtain
\beq \label{disceq3}\begin{aligned}
\int_0^ta(y(v_R))[-q(u(\tau,v_R\tau-);u_R) +& v_R\eta(u(\tau, v_R\tau-)|u_R)]\diff \tau \\=
&\int_{vt}^{v_Rt}a(y(v_R))\eta(u(t,x)|u_R) \diff x-\int_{\mathcal{C}_{v,v_R}}a(y(v_R))\diff \mu_u \\
&-\int_0^ta(y(v_R)[ q(u(\tau, v\tau+);u_R)-v\eta(u(\tau,v\tau+)|u_R)]\diff \tau.
\end{aligned}\eeq

Substituting \eqref{disceq2} and \eqref{disceq3} into \eqref{discineq1} gives
\begin{align}\label{discineq4}
\int_0^t& a(y(v_R)) \left(q(u(\tau,\tau v+);u_R)-v\eta(u(\tau,v\tau+)|u_R)\right) -a(y(v_L)) \left(q(u(\tau, \tau v-);u_L)-v\eta(u(\tau,v\tau-)|u_L)\right) \diff \tau \notag \\
&\leq \int_{v_Lt}^{vt}a(y(v_L))\eta(u(t,x)|u_L)\diff x+\int_{vt}^{v_Rt}a(y(v_R))\eta(u(t,x)|u_R) \diff x-\int_{v_Lt}^{v_Rt}a(y(t,x))\eta(u(t,x)|\overline{u}(t,x))\diff x \notag \\
&+\int_{\mathcal{C}_{v_L, v_R}}a(y)\diff \mu_u-\int_{\mathcal{C}_{v_L,v}}a(y(v_L))\diff \mu_u-\int_{\mathcal{C}_{v,v_R}}a(y(v_R))\diff \mu_u.
\end{align}
It remains to estimate the right-hand side of \eqref{discineq4}. We break this into two parts. For the first three terms on the right-hand side, we estimate
\begin{align*}
\text{(first three terms on RHS of \eqref{discineq4})} \leq&\hphantom{+} \int_{v_Lt}^{vt}[a(y(v_L))-a(y(t,x))]\eta(u(t,x)|\overline{u}(t,x)) \diff x \\
&+\int_{vt}^{v_Rt}[a(y(v_R))-a(y(t,x))]\eta(u(t,x)|\overline{u}(t,x))\diff x \\
&+\int_{v_Lt}^{vt}a(y(v_L))[\eta(u(t,x)|u_L)-\eta(u(t,x)|\overline{u}(t,x))]\diff x \\
&+\int_{vt}^{v_Rt}a(y(v_R))[\eta(u(t,x)|u_R)-\eta(u(t,x)|\overline{u}(t,x))]\diff x.
\end{align*}
Using the boundedness of $u$ and $\overline{u}$, properties of our relative quantities \Cref{lem:rel-quantities-properties}, and the definition of $a$ from Lemma \ref{rarecontract}, we obtain
\begin{align}\label{discineq5}
\text{(first three terms on RHS of \eqref{discineq4})} \leq Kt\overline{\sigma}^2,
\end{align}
where we recall the definition of $\overline{\sigma}$ from \Cref{prop:raredissipationestimate}.
\par For the second three terms on the right-hand side of \eqref{discineq4}, again using the definition of $a$, we compute
\beq\label{discineq6}\begin{aligned}
\int_{\mathcal{C}_{v_L, v_R}}&a(y)\diff \mu_u-\int_{\mathcal{C}_{v_L,v}}a(y(v_L))\diff \mu_u-\int_{\mathcal{C}_{v,v_R}}a(y(v_R))\diff \mu_u \\
&=\int_{\mathcal{C}_{v_L,v}}[a(y(t,x))-a(y(v_L))]\diff \mu_u+\int_{\mathcal{C}_{v,v_R}}[a(y(t,x))-a(y(v_R))] \diff \mu_u \\
& \leq K\overline{\sigma}|\mu_u|(\mathcal{C}_{v_L,v_R}).
\end{aligned}\eeq
Plugging \eqref{discineq5} and \eqref{discineq6} into \eqref{discineq4} gives
\begin{multline*}
\int_0^t a(y(v_R)) \left(q(u(\tau,\tau v+);u_R)-v\eta(u(\tau,v\tau+)|u_R)\right)-a(y(v_L)) \left(q(u(\tau, \tau v-);u_L)-v\eta(u(\tau,v\tau-)|u_L)\right) \diff \tau \\
\leq Kt\overline{\sigma}^2+K\overline{\sigma}|\mu_u|(\mathcal{C}_{v_L,v_R}).
\end{multline*}
Finally, dividing both sides by $a(y(v_R))$, 
noting that $\frac{a(y(v_L))}{a(y(v_R))}=e^{Cs_0}$ (for $1$-rarefactions), and noting that Lemma \ref{rarecontract} holds for any $C \in [\frac{C_2}{4}, 4C_2]$, we obtain
\begin{multline*}
\int_0^t \Bigl\{\left(q(u(\tau,\tau v+);u_R)-v\eta(u(\tau,v\tau+)|u_R)\right) -\frac{a_1}{a_2} \left(q(u(\tau, \tau v-);u_L)-v\eta(u(\tau,v\tau-)|u_L)\right)\Bigr\} \diff \tau  \\
\leq Kt\overline{\sigma}^2+K\overline{\sigma}|\mu_u|(\mathcal{C}_{v_L,v_R}),
\end{multline*}
for any $a_1, a_2$ such that $e^{\frac{C_2}{4}s_0} \leq \frac{a_1}{a_2} \leq e^{4C_2 s_0}$. Finally, using $e^x \approx 1+x$ for $x$ sufficiently small, we obtain \Cref{prop:raredissipationestimate} by taking $\eps$ sufficiently small. 
\end{proof}

\section{The front tracking algorithm}\label{sec:ft}
In this section, we exposit the front tracking algorithm we will use for the semigroup solution. We use exactly the scheme introduced in \cite{MR1367356} and adapted to the $a$-contraction method in \cite{cfk}. In the interest of readability, we will follow the notations used in those articles for the rest of the paper. All the results of this section will be stated without proof. We reference the papers \cite{MR1367356} and \cite{cfk} for proofs of the results.

\subsection{Introduction to the Riemann problem}\label{riemannalg}
For \eqref{cl}, let $A(u)=f'(u)$ be the Jacobian matrix of $f$ at $x$. Define the averaged matrix
\begin{align*}
A(u,u')=\int_0^1f'(\theta u+(1-\theta)u')\diff \theta.
\end{align*}
Let $\lambda_i(u,u'), r_i(u,u')$, and $\ell_i(u,u')$ be the $i^{th}$ eigenvalue and the associated right and left eigenvectors of $A(u,u')$, with the normalization
\begin{align*}
    \ell_i \cdot r_j=\delta_{ij},
\end{align*}
where $\delta_{ij}$ is the Kronecker delta.
In particular, for $u=u'$, $\lambda_i(u), r_i(u)$, and $\ell_i(u)$ are the $i^{th}$ eigenvalue and the associated right/left eigenvectors of $A(u)=A(u,u)$. 
\par Let $v=(v_1, v_2)$ be Riemann invariants of the first and second family respectively. Then, the rarefaction curves take the form
\begin{align}
\phi_1^+(v,\sigma)=(v_1+\sigma, v_2), \indent \phi_2^+(v,\sigma)=(v_1, v_2+\sigma).
\end{align}
Regarding the shock curves, we have the following representation
\begin{lemma}[{\cite[p. 10]{MR1367356}}]
The $i$-shock curve through the point $v=(v_1,v_2)$ can be parameterized with respect to the Riemann coordinates as
\begin{align}
\phi_1^-(v,\sigma)=(v_1+\sigma, v_2+\hat{\phi_2}(v,\sigma)\sigma^3), \indent \phi_2^-(v,\sigma)=(v_1+\hat{\phi_1}(v,\sigma)\sigma^3, v_2+\sigma),
\end{align}
for some smooth functions $\hat{\phi_2}, \hat{\phi_1}$. 
\end{lemma}
We emphasize that we use the same front tracking algorithm as \cite{MR1367356}. This means that in this section, we call jumps with $\sigma > 0$ rarefaction waves and jumps with $\sigma < 0$ shocks (this differs with the notation used in previous sections). Let us define the $\nu$-approximate solution to the Riemann problem. Let $\phi:\R \to \R$ be $C^\infty$ such that
\begin{align}\label{phidef}
\begin{cases}
\varphi(s)=1, & s \leq -2, \\
\varphi'(s) \in [-2,0], & s \in [-2,-1], \\
\varphi(s)=0, & s \geq -1.
\end{cases}
\end{align}
Then, for $\nu > 0$ fixed, define the following interpolations between the $i$-shock and $i$-rarefaction curve:
\begin{align}\label{interpcurve}
\Phi^\nu_i(v,\sigma)=\varphi\left(\frac{\sigma}{\sqrt{\nu}}\right)\phi_i^-(v,\sigma)+\left(1-\varphi\left(\frac{\sigma}{\sqrt{\nu}}\right)\right)\phi_i^+(v,\sigma). 
\end{align}
Note that when $\sigma \geq -\sqrt{\nu}$, $\Phi_i^\nu(v,\sigma)$ matches exactly the rarefaction curve $\phi_i^+$ and when, $\sigma \leq -2\sqrt{\nu}$, it matches exactly with the shock curve.
\par Given $u^r, u^l$ with corresponding Riemann coordinates $v^r=(v_1^r, v_2^r)$ and $v^l=(v_1^l, v_2^l)$, we can find a $\nu$-approximate solution to the Riemann problem with initial data:
\begin{align*}
u(0,x)=\begin{cases}
    u_l & x < 0, \\
    u_r & x > 0,
\end{cases}
\end{align*}
by finding a unique intermediate state $v^m$ and a pair $(\sigma_1, \sigma_2)$ such that
\begin{align*}
v^m=\Phi_1^\nu(v^l, \sigma_1), \indent v^r=\Phi_2^\nu(v^m, \sigma_2).
\end{align*}
Then, we approximate the two waves in the following manner. We give only the approximation of the $1$-waves. The $2$-waves are exactly the same. 
\par Assume $\sigma_1 > 0$. Then, let $h,k \in \mathbb{N}$ such that $h\nu \leq v_1^l < (h+1)\nu$ and $k \nu \leq v_1^m < (k+1)\nu$. Next, introduce the states $\omega_1^j=(j\nu, v_2^l)$ and $\hat{\omega}_1^j=( (j+\frac{1}{2})\nu, v_2^l)$ for $j=h, ..., k$. Then, we construct the $\nu$-approximate rarefaction in the following manner:
\begin{align}\label{nuaprare}
u^\nu(t,x):=\begin{cases}
    v^l, & \frac{x}{t} < \lambda_1\hat{\omega}_1^h), \\
    \omega_1^j, & \lambda_1(\hat{\omega}_1^{j-1}) < \frac{x}{t} < \lambda_1(\hat{\omega}_1^j), \\
    v^m, & \lambda_1(\hat{\omega}_1^k) < \frac{x}{t}.
\end{cases}
\end{align}
\par If $\sigma_1 < 0$, then the states $v^l$ and $v^m$ are connected by a single shock traveling with an interpolated speed
\begin{align*}
u^\nu(t,x)=\begin{cases}
    v^l, & \frac{x}{t} < \lambda_1^\phi(v^l, \sigma_1), \\
    v^m, & \frac{x}{t} > \lambda_1^\phi(v^l, \sigma_1),
\end{cases}
\end{align*}
where the interpolated speed is
\begin{align}\label{nushockspeed}
    \lambda_1^\varphi(v^l, \sigma_1)=\varphi\left(\frac{\sigma_1}{\sqrt{\nu}}\right)\lambda_1^s(v^l, \sigma_1)+\left(1-\varphi\left(\frac{\sigma_1}{\sqrt{\nu}}\right)\right)\lambda_1^r\left(v^l, \sigma_1\right),
\end{align}
with
\begin{align*}
&\lambda_1^s(v^l, \sigma_1)=\lambda_1(v^l,\phi_1^-(v^l, \sigma_1)), \\
&\lambda_1^r(v^l, \sigma_1)=\sum_j \left(\frac{\text{meas}([j\nu, (j+1)\nu] \cap [v_1^m, v_1^l])}{|\sigma_1|}\right)\lambda_1(\hat{\omega}_1^j).
\end{align*}

\subsection{Classical $\nu$-approximate solutions}\label{classical}
Now, let $\hat{u}(0,x) \in \mathcal{S}_{BV,\eps}^0$ be piecewise constant initial data. We construct the $\nu$-approximate solution $u^\nu$ as follows. At time $t=0$, solve all the Riemann problems defined by the jumps in $\hat{u}(0,x)$ using the algorithm in Section \ref{riemannalg}. Let the new piecewise constant solution evolve until the first interaction takes place. Then, use the algorithm to solve the Riemann problem at the place of this interaction. By inductively repeating this process, for $\eps$ sufficiently small, the piecewise constant function $u^\nu$ can be defined until any finite time $T$, since there are at most finitely many interactions in $[0,T]$, as long as $\hat{u}(0,x) \in \mathcal{D}$, where $\mathcal{D}$ is defined in Section \ref{bvbd}.

\subsection{Shifted $\nu$-approximate solutions}\label{sec:shifted}
Following the ideas introduced in \cite{MR4487515}, we also consider a shifted $\nu$-approximate front tracking solution $\psi$, where we still use \eqref{interpcurve} to solve the Riemann problem. We still use the same approximated wave speeds in \eqref{nuaprare} as before for rarefaction waves. However, for shocks, we use the shifted wave speed given by \Cref{lem:shockdissipationestimate}. Let us be more precise. For example, if $(u^l, u^m)$ are connected via a $1$-shock (where $u^m$ in physical variables corresponds to $v^m$ in Riemann variables), then the states $v^l, v^m$ are connected by a single shock
\begin{align}\label{shiftedshock}
u^\nu(t,x)=\begin{cases}
    v^l, & x < h(t), \\
    v^m, & x > h(t),
\end{cases}
\end{align}
where $h$ is given by applying \Cref{lem:shockdissipationestimate} with $u^l=u_L$ and $u^m=u_R$ when $\sigma < -2 \sqrt{\nu}$. 
When $\sigma > -2\sqrt{\nu}$ the states $(v^l,v^m)$ are not necessarily a shock, so instead we select the shift from \Cref{lem:shockdissipationestimate} corresponding to the nearby physical shock $(v^l, \phi_i^-(v^l,\sigma))$. 
We note, for future use, that 
\begin{equation} \label{eq:small-shock-errs}
    \begin{aligned}
        | \Phi_i^\nu(v^l, \sigma) - \phi_i^-(v^l, \sigma) | & \leq K | \sigma|^3, \\
        | \lambda_i^\varphi(v^l,\sigma) - \lambda_i^s(v^l,\sigma)| & \leq K\nu,
    \end{aligned}
\end{equation}
where $\lambda_i^\varphi$ is the interpolated shock speed~\eqref{nushockspeed} and $\lambda_i^s$ is the Rankine-Hugoniot velocity.
The first of these estimates follows from~\eqref{interpcurve} and classical asymptotic estimates on the shock curve (see Lemma \ref{lem:dafermoslemma}). The second is known from \cite[Section 7]{MR1367356}.

As remarked in \cite[p. 16]{cfk}, using similar arguments developed in \cite{MR1367356} for the classical $\nu$-approximate solutions, we can show that the shifted $\nu$-approximate solution $\psi$ is well-defined until any finite time $T$ as there are only finitely many interactions in $[0,T]$ for any initial data in $\mathcal{D}$. 
\begin{remark}
A crucial property to note is that the construction of the shift for shocks in \Cref{lem:shockdissipationestimate} guarantees that waves generated in the solution of a Riemann problem will not meet at a future time, even with shocks traveling at the shifted speeds, as long as $\eps_1$ is sufficiently small such that all shocks satisfy \Cref{lem:shockdissipationestimate}.
\end{remark}

\subsection{Control on the BV-norm}\label{bvbd}
For both the classical and shifted $\nu$-approximate solutions, we use the following notation for total variation and Glimm potential. At any time $t$, we denote waves (and their strengths) by $\sigma_{i,\alpha}$ (strength $|\sigma_{i,\alpha}|$), where the wave corresponding to $\sigma_{i,\alpha}$ is located at $x=x_\alpha$ and is in the $i^{th}$ family. We abuse notation slightly here as at a time of interaction when a rarefaction fan is formed, there may be more than one wave corresponding to $\sigma_{i,\alpha}$, but they split immediately afterwards.
\par For any piecewise constant function $u$ generated from either the classical or shifted $\nu$-approximate algorithm, we denote:
\begin{align}\label{tvandglimm}
V(u):=\sum_\alpha \sum_i|\sigma_{i,\alpha}|, \indent Q(u):=\sum_{\sigma_{i,\alpha},\sigma_{j,\beta} \in \mathcal{A}}|\sigma_{i,\alpha}\sigma_{j,\beta}|, \indent U(u):=V(u)+\kappa Q(u),
\end{align}
where $\mathcal{A}$ is the set of pairs of approaching waves and $\kappa > 0$ is a constant to be determined. Recall that $\sigma_{i,\alpha} \text{ and } \sigma_{j,\beta}$ are approaching if either $x_\alpha < x_\beta$ and $j < i$, or if $j=i$ and $\min(\sigma_{i\alpha}, \sigma_{j,\beta}) < 0$. 
\par Both the classical and shifted $\nu$-approximate solutions will be well-defined until any finite time $T$, and there will be only finitely many interactions in $[0,T]$ if the initial data belongs to the domain $\mathcal{D}$ (see \cite[Proposition 4]{MR1367356}), which is defined as follows
\begin{align}\label{epsforft}
\mathcal{D}:=\{u-d \in L^1(\R;\R^2)\,|\,u \text{ piecewise constant, } 0 < U(u) < \eps\},
\end{align}
for sufficiently small $\eps$ and sufficiently large $\kappa$. Furthermore, $U(u^\nu(t,\cdot))$ decays in $t$ and there exists a constant $K$ such that for any $t$:
\begin{align}\label{bvcontrol}
\frac{1}{K}||u^\nu(t,\cdot)||_{BV(\R)} \leq U(u^\nu(t,\cdot)) \leq K||u^\nu||_{BV(\R)}.
\end{align}
The same holds for the shifted $\nu$-approximate solution by \cite[Equation (4.22)]{cfk}
\begin{align}
||\psi(t, \cdot)||_{BV(\R)} \leq K||\psi(0,\cdot)||_{BV(\R)}. \label{eq:shifed-bv-bound}
\end{align}
\begin{remark}\label{rem:finitepropspeed}
There exists $\delta_0$ sufficiently small so that the following holds. Let $u^0 \in L^1(\R)$ verify $||u^0-d||_{L^\infty(\R)} \leq \delta_0$ and $||u^0||_{BV(\R)} \leq \delta_0$. Let $u_\nu$ be a sequence of $\nu$-approximate solutions with initial data $u_\nu(0,\cdot) \to u^0$ in $L^2$ as $\nu \to 0$. Then, using \eqref{bvcontrol}, we may take the limit $\nu \to 0$ and obtain a weak entropy solution $u$ to \eqref{cl} with initial data $u^0$. As each $u^\nu$ has a finite speed of propagation property, this remains true in the limit as well. That is, the value of $u(t,x)$ is only determined by the initial value $u^0$ on the interval $[x-ct, x+ct]$, where the constant $c$ is the speed of information (cf. \Cref{rem:soi}).
\end{remark}

\subsection{Weighted $L^1$-estimates for classical and shifted $\nu$-approximate solutions}
In this section, we quote results from \cite{MR1367356} and \cite{cfk} regarding control over the distance between two solutions generated from various $\nu$-approximate front tracking schemes. The first controls the distance between two classical $\nu$-approximate solutions.
\begin{lemma}[{\cite[Proposition 8]{MR1367356}}]\label{bclemma}
For any $\mathcal{S}$, which is the classical $\nu$-approximate semigroup, there exists a weighted distance $d_\nu$ that is uniformly (in $\nu$) equivalent to the $L^1$-distance, i.e.
\begin{align*}
\frac{1}{K}||u-\overline{u}||_{L^1(\R)} \leq d_\nu(u,\overline{u}) \leq K||u-\overline{u}||_{L^1(\R)},
\end{align*}
for some $K > 0$. Furthermore, $d_\nu$ is contractive with respect to the semigroup $\mathcal{S}$, i.e.
\begin{align*}
d_\nu(\mathcal{S}_tu_0, \mathcal{S}_t\overline{u}_0) \leq d_\nu(u_0, \overline{u}_0),
\end{align*}
for any $t \geq 0$ and initial data $u_0, \overline{u}_0 \in \mathcal{D}$. 
\end{lemma}
The second controls the distance between a classical and shifted $\nu$-approximate solution.
\begin{lemma}[{\cite[Lemma 4.4]{cfk}}]\label{cfklemma}
Let $u$ be a classical $\nu$-approximate front tracking solution constructed as described in Section \ref{classical}, and let $\psi$ be a $\nu$-approximate shifted front tracking solution constructed as described in \Cref{sec:shifted}, with shocks moving at artificial speeds $h_\alpha(t)$. Then, in the sense of distributions
\begin{align}
\ds{\frac{d}{dt}}d_\nu(u(t,\cdot), \psi(t,\cdot)) \leq K \sum_{\alpha=1}^n|\sigma_\alpha||\dot{h}_\alpha(t)-\dot{h}_{true, \alpha}(t)|, 
\end{align}
for some constant $K>0$, where the summation in $\alpha$ runs over all shocks in $\psi$ at time $t$. The quantity $\dot{h}_\alpha$ is the artificial speed for $\psi(t,x)$ given by \Cref{lem:shockdissipationestimate}, while $\dot{h}_{true, \alpha}$ is the speed that this discontinuity in $\psi$ would have if it were a classical front tracking solution (the speed given by \eqref{nushockspeed}).
\end{lemma}

\section{The weight function $a$}\label{a}
In this section, we construct the weight $a(t,x)$ needed to obtain $L^2$-stability via \Cref{lem:shockdissipationestimate} and \Cref{prop:raredissipationestimate}. We define the weight function in a similar manner as \cite{cfk}. However, our weight has some slight changes as we need to weight both shocks and rarefactions instead of just shocks. In all previous works on the $a$-contraction, the rarefactions do not have to be weighted due to the assumption that {Hypotheses 1} \ref{hyp-d} holds. We continue to follow the notation introduced in Section \ref{sec:ft}. Finally, for any function $f(t)$, we denote
\begin{align*}
    \Delta f(t):=f(t+)-f(t-).
\end{align*}

\subsection{Wave interaction estimates}
Firstly, we recall the basic interaction estimates from \cite{MR1367356}. Henceforth, given a shifted $\nu$-approximate solution $\psi$ (defined in \Cref{sec:shifted}), we define $U(t):=U(\psi(t,\cdot))$ and $Q(t):=Q(\psi(t,\cdot))$.
\begin{lemma}[{\cite[Lemma 5]{MR1367356}}]\label{lem:waveinteractionestimates}
For any time $t$ where waves with strength $|\sigma_i|$ and $|\sigma_j|$ interact, we have
\begin{align*}
\Delta U(t) \leq -\frac{\kappa}{2}|\sigma_i||\sigma_j|.
\end{align*}
Further, when there is an interaction at $(t,x)$ between incoming waves from the second family with strengths $\sigma_i''$ and incoming waves from the first family with strengths $\sigma_i'$, we denote the strength of the outgoing wave of the first family as $\sigma_1$ and the second wave as $\sigma_2$ and the following estimate holds:
\begin{align*}
\left|\sigma_2-\sum \sigma_i''\right|+\left|\sigma_1-\sum \sigma_i'\right| \leq K_0|\Delta Q|,
\end{align*}
for some constant $K_0$ independent of $\kappa$. Finally the following properties hold:
\begin{enumerate}
    \item In any family, if there are two or more incoming rarefaction waves, they cannot be adjacent.
    \item If, in any family, all incoming waves are shocks, the outgoing wave of that family is still a shock.
\end{enumerate}
\end{lemma}
We note that \Cref{lem:waveinteractionestimates} is proved in \cite{MR1367356} only for classical $\nu$-approximate solutions, but it holds also for shifted ones, as remarked in \cite[p. 24]{cfk}.
It is standard to obtain the following from \Cref{lem:waveinteractionestimates} (see  \cite[Lemma 5]{MR1367356}.
\begin{lemma}\label{lem:shiftedfnldec}
For any initial data $u^0 \in \mathcal{D}$, for any $\kappa > 0$ sufficiently large and any $\eps > 0$ sufficiently small, the functionals $U(t)$ and $Q(t)$ are decreasing in time.
\end{lemma}

\subsection{Construction of the weight $a$}
Here, we give the definition of the weight $a(t,x)$ associated with a shifted $\nu$-approximate front tracking solution. Note that the weight needed for $1$-shocks is exactly the same as the weight needed for $1$-rarefactions, as we may take $C_1=C_2$ by Remark \ref{same} (and the same goes for the $2$-waves). With that in mind, we set
\begin{align*}
\frac{a_r}{a_l}=e^{-\frac{3C_1}{4}|\sigma|} \text{ for $1$-waves, and } \frac{a_r}{a_l}=e^{\frac{3C_1}{4}|\sigma|} \text{ for $2$-waves},
\end{align*}
where $a_r$ and $a_l$ are the values of $a$ to the right and left of the wave respectively, and $C_1=C_2$ is the constant from \Cref{lem:shockdissipationestimate} and \Cref{prop:raredissipationestimate}. Then, when the wave strengths are small enough
\begin{align*}
e^{\frac{3C_1}{4}|\sigma|} \approx 1+\frac{3C_1}{4}|\sigma|,
\end{align*}
so the conditions on $a_r, a_l$ (here playing the role of $a_1, a_2$) are satisfied in \Cref{lem:shockdissipationestimate} and \Cref{prop:raredissipationestimate}. 
\par Let us define the weight function $a$ as follows. Firstly, define a measure $\mu(t, \cdot)$ as a sum of Dirac measures in $x$:
\begin{align*}
\mu(t,x)=-\sum_{i:1-\text{wave}}|\sigma_i|\delta_{x_i(t)}+\sum_{i:2-\text{wave}}|\sigma_i|\delta_{x_i(t)}.
\end{align*}
Then, define the weight:
\begin{align}\label{adef}
a(t,x)=\text{exp}\left(\frac{3C_1}{4}\Bigl(V(t)+\frac{3\kappa}{2}Q(t)+\int_{-\infty}^x\mu(t,y)\diff y\Bigr) \right),
\end{align}
for $\kappa$ sufficiently large to be determined.
\begin{remark}
Note that our weight is almost identical to the weight $a(t,x)$ in \cite[Equation (5.13)]{cfk}, except that we weight all waves instead of just shocks. 
\end{remark}
We now state and prove some important properties of the function $a$.
\begin{proposition}\label{aprops}
For a fixed shifted $\nu$-approximate solution $\psi$, the weight $a$ defined in \eqref{adef} verifies the following properties:
\begin{enumerate}
    \item For every time without a wave interaction, and every $x$ such that a $1$-wave $\sigma_\alpha$ is located at $x=x_\alpha(t)$ in $\psi$,
    \begin{align}\label{1stability}
    1-2C_1|\sigma_\alpha| \leq \frac{a(t,x_\alpha(t)+)}{a(t,x_\alpha(t)-)} \leq 1-\frac{C_1}{2}|\sigma_\alpha|.
    \end{align}
    \item For every time without a wave interaction, and every $x$ such that a $2$-wave $\sigma_\alpha$ is located at $x=x_\alpha(t)$ in $\psi$,
    \begin{align}\label{2stability}
    1+\frac{C_2}{2}|\sigma_\alpha| \leq \frac{a(t,x_\alpha(t)+)}{a(t,x_\alpha(t)-)} \leq 1+2C_2|\sigma_\alpha|.
    \end{align}
    \item For every time $t$ with a wave interaction, and almost every $x$, the weight function decays in time,
    \begin{align}\label{atimedecay}
    a(t+,x) \leq a(t-,x),
    \end{align}
    when $\kappa$ is sufficiently large. This holds except at the finitely many points where interactions occur.
    \item Finally, we have
    \begin{align}\label{abd}
        \norm{{1}/{a}}_{L^\infty([0,T] \times \R)} + \norm{a}_{L^\infty([0,T] \times \R)} \leq K,
    \end{align}
    for a constant $K>0$.
\end{enumerate}
\end{proposition}
\begin{proof}
The properties \eqref{1stability} and \eqref{2stability} are clear from the definition \eqref{adef} for initial data $\psi(0,\cdot) \in \mathcal{D}$ for $\eps$ sufficiently small, as then the size of any shock $|\sigma|$ in $\psi(t,\cdot)$ is bounded by $\eps$ and $e^{\frac{3C_1}{4}|\sigma|} \approx 1+\frac{3C_1}{4}|\sigma|$ for $|\sigma|$ sufficiently small.
\par Now, we show \eqref{atimedecay}. Clearly, \eqref{atimedecay} holds in any region of the $(t,x)$-plane between jumps. It remains to consider the times when there is a wave interaction. We assume there is only one interaction point $(t,x_0)$ at time $t$. If there are more than one, applying the following proof to all interaction points separately suffices. 
\par Recall that we may decompose a function $f=f_+-f_-$, where $f_{\pm}=\frac{|f|\pm f}{2}$. If $1$-waves $\sum \sigma_i'$ and $2$-waves $\sum \sigma_i''$ interact at $(t,x_0)$, then
\begin{align}\label{intermed}
\mu(t+)-\mu(t-)&=\delta_{x_0}\left(|\sigma_2|-|\sigma_1|-\sum|\sigma_i''|+\sum|\sigma_i'|\right)\notag \\
            &\leq \delta_{x_0}\left(\left||\sigma_2|-\sum |\sigma_i''|\right|+\left||\sigma_1|-\sum |\sigma_i'|\right|\right).
\end{align}
We will use \Cref{lem:waveinteractionestimates} to show
\begin{align}\label{intermed2}
\eqref{intermed} \leq \left((\Delta V)_-+\frac{\kappa}{2}|\Delta Q|\right)\delta_{x_0},
\end{align}
for $\kappa$ sufficiently large. Note that when \eqref{intermed2} holds, by \eqref{adef}, \eqref{intermed}, and the fact that $V(t)+\kappa Q(t)$ and $Q(t)$ both decay in time by \Cref{lem:shiftedfnldec}, we obtain
\begin{align*}
    \frac{a(t+,x)}{a(t-,x)} \leq 1 \implies a(t+,x)-a(t-,x) \leq 0,
\end{align*}
for any $x \in \R$ besides the point of interaction $(t,x_0)$. 
\par Let us then prove \eqref{intermed2}. Firstly, we obtain the following estimate by casework:
\begin{align}\label{2waveestimate}
\left||\sigma_2|-\sum|\sigma_i''|\right| \leq \sum|\sigma_i''|-|\sigma_2|+2K_0|\Delta Q|. 
\end{align}
\par \textbf{Case 1:} If all incoming $2$-waves are the same type, then by \Cref{lem:waveinteractionestimates}, the outgoing $2$-wave is the same type as all the incoming waves. If all $2$-waves are shocks, then by \Cref{lem:waveinteractionestimates}
\begin{align*}
 \left|(\sigma_2)_+-\sum(\sigma_i'')_+\right|&=0, \\
 \left|(\sigma_2)_--\sum(\sigma_i'')_-\right| &\leq K_0|\Delta Q|.
\end{align*}
If the incoming $2$-wave is just one discretized rarefaction we obtain instead
\begin{align*}
 \left|(\sigma_2)_+-\sum(\sigma_i'')_+\right|& \leq K_0|\Delta Q|, \\
 \left|(\sigma_2)_--\sum(\sigma_i'')_-\right| &=0.
\end{align*}
So
\begin{align*}
\left||\sigma_2|-\sum|\sigma_i''|\right| \leq \left|(\sigma_2)_+-\sum(\sigma_i'')_+\right|+\left|(\sigma_2)_--\sum(\sigma_i'')_-\right| \leq K_0|\Delta Q|,
\end{align*}
from which \eqref{2waveestimate} follows.
\par \textbf{Case 2:} Otherwise, there must be both incoming $2$-rarefactions and $2$-shock waves. Firstly, assume $\sigma_2 \geq 0$. Then, by \Cref{lem:waveinteractionestimates}
\begin{align}
K_0|\Delta Q| \geq \left|\sigma_2-\sum \sigma_i''\right| &= \left|\sigma_2-\sum (\sigma_i'')_++\sum (\sigma_i'')_-\right| \notag \\ 
&\geq \sigma_2-\sum (\sigma_i'')_++\sum (\sigma_i'')_-, \notag
\intertext{so} 
\sigma_2-\sum(\sigma_i'')_+ &\leq K_0|\Delta Q|-\sum(\sigma_i'')_- \leq K_0|\Delta Q|. \label{pos}
\intertext{Now, if $\sigma_2-\sum(\sigma_i'')_+ \geq 0$, \eqref{pos} gives}
\left|(\sigma_2)_+-\sum(\sigma_i'')_+\right| & \leq K_0|\Delta Q|, \label{case1}
\intertext{while if $\sigma_2-\sum(\sigma_i'')_+ \leq 0$,}
\left|(\sigma_2)_+-\sum(\sigma_i'')_+\right|&=\sum(\sigma_i'')_+-|\sigma_2|. \label{case2}
\intertext{Together, \eqref{case1} and \eqref{case2} give}
\left|(\sigma_2)_+-\sum(\sigma_i'')_+\right| &\leq 2K_0|\Delta Q|+\sum(\sigma_i'')_+-|\sigma_2|. \notag 
\intertext{Combined with $\left|(\sigma_2)_--\sum(\sigma_i'')_-\right|=\sum(\sigma_i'')_-$, we obtain}
\left||\sigma_2|-\sum|\sigma_i''|\right| &\leq \left|(\sigma_2)_+-\sum(\sigma_i'')_+\right|+\left|(\sigma_2)_--\sum(\sigma_i'')_-\right| \notag \\
&\leq \sum|\sigma_i''|-|\sigma_2|+2K_0|\Delta Q|, \notag
\end{align}
establishing \eqref{2waveestimate}. 

Now, assume that $\sigma_2 \leq 0$. Then, \Cref{lem:waveinteractionestimates} gives
\begin{align*}
K_0|\Delta Q| \geq \left|\sigma_2-\sum \sigma_i''\right| &= \left|\sigma_2-\sum (\sigma_i'')_++\sum (\sigma_i'')_-\right| \\
&\geq -\sigma_2+\sum (\sigma_i'')_+-\sum (\sigma_i'')_-,
\intertext{so} 
-\sigma_2-\sum(\sigma_i'')_- &\leq K_0|\Delta Q|-\sum(\sigma_i'')_+ \leq K_0|\Delta Q|,
\intertext{and similarly to above:}
\left|(\sigma_2)_--\sum(\sigma_i'')_-\right| &\leq 2K_0|\Delta Q|+\sum(\sigma_i'')_-+\sigma_2 \\
&=\sum (\sigma_i'')_--|\sigma_2|+2K_0|\Delta Q|.
\intertext{Combined with $\left|(\sigma_2)_+-\sum(\sigma_i'')_+\right|=\sum(\sigma_i'')_+$, we obtain}
\left||\sigma_2|-\sum|\sigma_i''|\right| &\leq \left|(\sigma_2)_+-\sum(\sigma_i'')_+\right|+\left|(\sigma_2)_--\sum(\sigma_i'')_-\right| \\
&\leq \sum|\sigma_i''|-|\sigma_2|+2K_0|\Delta Q|.
\end{align*}
This shows \eqref{2waveestimate} holds in all cases. The estimate
\begin{align}\label{1waveestimate}
\left||\sigma_1|-\sum|\sigma_i'|\right| \leq \sum|\sigma_i'|-|\sigma_1|+2K_0|\Delta Q|,
\end{align}
follows from entirely similar considerations. Finally, choosing $\kappa > 8K_0$, \eqref{intermed2} follows from \eqref{1waveestimate}, \eqref{2waveestimate}, and the definition \eqref{tvandglimm}. 
\par Finally, uniform bounds on $V$ and $Q$ may be derived from the time decay of $U$ and $Q$ in \Cref{lem:shiftedfnldec}, immediately granting \eqref{abd}.
\end{proof}

\section{Proof of \Cref{prop:main}}\label{thmpf}
In this section, we prove \Cref{prop:main}, following \cite{cfk} with some small changes to account for the weighting of the rarefactions. Firstly, we address some important technical points that will arise in the proof.

\subsection{Stopping and restarting the clock}
When we combine the relative entropy method with the front-tracking scheme with respect to a weak solution $u \in \Sweak$, we need to ``stop and restart the clock'' every time two waves collide. As the solutions $u \in \Sweak$ live in a low regularity class, this requires the use of approximate limits (see \cite[p. 55-57]{MR3409135}). The following estimates will be used in a telescoping sum in Section \ref{dissipationcalculations}.

\begin{lemma}\label{stitchinglemma}
Let $u \in L^\infty(\R^+ \times \R)$ satisfy \eqref{cl}, \eqref{entropic} with initial data $u^0$. Further, assume that $u$ verifies the Strong Trace Property (\Cref{def:strongtrace}). Then, for all $v \in W$, and all $c,d \in \R$ with $c < d$, the approximate left- and right-hand limits:
\begin{align*}
    \aplim_{t \to t_0^\pm}\int_c^d\eta(u(t,x)|v)\diff x,
\end{align*}
exist for all $t_0 \in (0,T)$ and verify:
\begin{align*}
\aplim_{t \to t_0^-}\int_c^d\eta(u(t,x)|v)\diff x &\geq \aplim_{t \to t_0^+}\int_c^d\eta(u(t,x)|v)\diff x.
\intertext{Moreover, the approximate right-hand limit exists at $t_0=0$ and verifies:}
\int_c^d\eta(u^0(x)|v)\diff x &\geq \aplim_{t \to t_0^+}\int_c^d\eta(u(t,x)|v)\diff x.
\end{align*}
\end{lemma}
We refer to \cite[Lemma 2.5]{MR3954680} for the proof.
\par The next lemma will control the boundaries of the cone of information.
\begin{lemma}[{\cite[Equation (5.2.5), p. 123]{MR3468916}}]\label{lem:qcontrolbyeta}
There exists a constant $c > 0$ such that:
\[
|q(a;b)| \leq c\eta(a|b), \text{ for any } (a,b) \in W \times W.
\]
\end{lemma}
\begin{proof}
We have $q(b;b)=\partial_1q(b;b)=0$ for any $b \in W$, so by Lemma \ref{relentisl2} and $\nabla^2q \in C^0(W)$, there exists a constant $c$ such that,
\[
|q(a;b)| \leq K|a-b|^2 \leq c \eta(a|b).
\]
\end{proof}

\subsection{Proof of \Cref{prop:main}} Finally, we give the proof of \Cref{prop:main}. Fix $R,T > 0$, and let $\tau \in [0,T]$. 

\subsubsection{Step 1: dissipation calculations}\label{dissipationcalculations}
Consider initial data $v^0$ such that $||v^0||_{BV} \leq \delta_0, ||v^0-d||_{L^\infty(\R)} \leq \eps$ for $\delta_0, \eps > 0$ sufficiently small so that \Cref{lem:localstructurelemma}, \Cref{lem:shockdissipationestimate}, \Cref{prop:raredissipationestimate}, and \eqref{epsforft} are satisfied. Let $u \in \Sweak$ be a fixed wild solution. Consider a sequence of front tracking approximations $v_\nu$ (defined in Section \ref{classical}) with initial data verifying $v_\nu(0,\cdot) \to v^0$ in $L^2$ as $\nu \to 0$.
\begin{remark}
Note that \cite[Proposition 4]{MR1367356}---the existence of a $\nu$-approximate solution---holds for piecewise constant functions with initial data $v_\nu(0, \cdot) \in \mathcal{D}$. In particular, the results of \Cref{sec:ft} require $v_\nu(0, \cdot)-d \in L^1(\mathbb{R};\mathbb{R}^2)$. Note this is not a priori guaranteed for an $L^2$-approximation of $v^0 \in \mathcal{S}_{BV,\eps}^0$. However, as the $L^2$-norms in \Cref{prop:main} are derived over a finite interval $[-R-cT, R+cT]$, we may simply work with the cutoff initial data
\begin{align*}
\tilde{v}^0=\begin{cases}
d & x \leq -\tilde{R}, \\
v^0(x) & x \in [-\tilde{R}, \tilde{R}], \\
d & x \geq \tilde{R},
\end{cases}
\end{align*}
for some $\tilde{R}$ sufficiently large. We pick $\delta^0$ and $\eps$ sufficiently small so that $||\tilde{v}^0||_{BV(\R)}$ is small enough for \cite[Proposition 4]{MR1367356} to apply. This does not affect the estimates~\eqref{L1L2} and~\eqref{L2holder} in \Cref{prop:main} due to the finite speed of propagation.
\end{remark}
For fixed $\nu$, consider the shifted $\nu$-approximate front tracking solution $\psi$ with $\psi(0,x)=v_\nu(0,x)$ (defined in \Cref{sec:shifted}). For this $\psi$, consider the associated weight $a(t,x)$ defined in \eqref{adef}. 
Consider two successive interaction times $t_j < t_{j+1}$ in the shifted front tracking solution $\psi$. Call the curves of discontinuity in $(t_j, t_{j+1})$ $h_1, \dots, h_N$ for some $N \in \mathbb{N}$ such that
\begin{align*}
    h_1(t) < \cdots < h_N(t),
\end{align*}
for all $t \in (t_j, t_{j+1})$. Let us now define the cone of information for $t \geq 0$:
\begin{equation}\label{conebdry}\begin{aligned}
h_0(t)&:=-R+c(t-\tau), \\
h_{N+1}(t)&:=R-c(t-\tau).
\end{aligned}
\end{equation} 
\begin{remark}\label{rem:soi}
Recall that $c$ is defined by \Cref{lem:qcontrolbyeta}. We further take $c$ bigger if necessary to be larger than the magnitudes of the speeds of any of the shift functions $\dot{h}$ in \Cref{lem:shockdissipationestimate} or the classical waves speeds in \eqref{nushockspeed}, \eqref{nuaprare}. We remark that there are no interactions between waves in $\psi$ and the cone of information (which has boundaries $h_0, h_{N+1}$). 
\end{remark} 
Now, for $t \in [t_j, t_{j+1}]$, note that on the fragment of the cone
\begin{align*}
    Q^i_{t_j,t}=\{(r,x)\,|\,t_j < r < t, h_i(r) < x < h_{i+1}(r)\},
\end{align*}
both the functions $\psi(r,x)$ and $a(r,x)$ are constant. Thus, integrating \eqref{entropic} on $Q^i_{t_j,t}$ and using the Strong Trace Property (\Cref{def:strongtrace}) grants
\begin{align*}
\aplim_{s \to t-}&\int_{h_i(t)}^{h_{i+1}(t)}a(t-,x)\eta(u(s,x)|\psi(t,x))\diff x \\
&\leq \aplim_{s \to t_j^+}\int_{h_i(t_j)}^{h_{i+1}(t_j)}a(t_j+,x)\eta(u(s,x)|\psi(t_j,x))\diff x +\int_{t_j}^t(F_i^+(r)-F_{i+1}^-(r))\diff r,
\end{align*}
where
\begin{align*}
&F_i^+(r)=a(r,h_i(r)+)\left(q(u(r,h_i(r)+);\psi(r,h_i(r)+))-\dot{h}_i(r)\eta(u(r,h_i(r)+)|\psi(r,h_i(r)+))\right), \\
&F_i^-(r)=a(r,h_i(r)-)\left(q(u(r,h_i(r)-);\psi(r,h_i(r)-))-\dot{h}_i(r)\eta(u(r,h_i(r)-)|\psi(r,h_i(r)-))\right).
\end{align*}
Summing in $i$, and combining terms corresponding to $i$ in one sum and terms corresponding to $i+1$ in another, we obtain
\begin{align*}
\aplim_{s \to t-}&\int_{-R+c(t-\tau)}^{R-c(t-\tau)}a(t-,x)\eta(u(s,x)|\psi(t,x))\diff x \\
&\leq \aplim_{s \to t_j^+}\int_{-R+c(t_j-\tau)}^{R-c(t_j-\tau)}a(t_j+,x)\eta(u(s,x)|\psi(t_j,x))\diff x +\sum_{i=1}^N\int_{t_j}^t(F_i^+(r)-F_{i}^-(r))\diff r,
\end{align*}
where we have used that $F_0^+ \leq 0$ and $F_{N+1}^- \geq 0$ due to \Cref{lem:qcontrolbyeta}, the definition of $c$, and $\dot{h}_0=c=-\dot{h}_{N+1}$. We break the sum on the right-hand side into two parts: the first part collects the shocks and the other part collects the rarefactions. We know from \Cref{lem:shockdissipationestimate} and Proposition \ref{aprops} that for any $i$ denoting a shock front:
\begin{equation}\label{shockdissipation}\begin{aligned}
F_i^+(r)-F_i^-(r) \leq& -\frac{1}{K}|\psi(r,h_i(r)+)-\psi(r,h_i(r)-)| \\
&\times\left(\lambda_{\alpha_i}^\varphi(\psi(r,h_i(r)+),\psi(r,h_i(r)-))-\dot{h}_i(r)\right)^2 \\
&+ K\nu|\psi(r,h_i(r)+)-\psi(r,h_i(r)-)|,
\end{aligned}\end{equation}
where $\lambda_{\alpha_i}^\varphi$ is the $\nu$-approximate speed of the shock~\eqref{nushockspeed} and $\alpha_i = 1,2$ is the wave family of the $i^{th}$ front. 
We note that the last term of the bound~\eqref{shockdissipation} is present to account for the errors at small shocks with $\sigma > -2\sqrt{\nu}$: they can be deduced by the bounds on the $\nu$-approximate curves and velocities in equation~\eqref{eq:small-shock-errs}.

Denote $\mathcal{R}(r)$ the set of $i$ corresponding to rarefactions at time $r$. Then, from \Cref{prop:raredissipationestimate} and Proposition \ref{aprops}, we have
\begin{align}\label{raredissipation}
\sum_{i \in \mathcal{R}(r)}\int_{t_j}^t(F_i^+(r)-F_i^-(r))\diff r \leq K \nu(t-t_j)+K\nu|\mu_u|(Q_{t_j,t}),
\end{align}
where $Q_{t_j,t} := \cup_iQ_{t_j,t}^i$ for $t \in [t_j,t_{j+1}]$.
Collecting both families of waves, and denoting $\mathcal{S}(r)$ the set of $i$ corresponding to shocks:
\begin{align*}
\aplim_{s \to t-}\int_{-R+c(t-\tau)}^{R-c(t-\tau)}a(t-,x)\eta(u(s,x)|\psi(t,x))\diff x \leq&\, \aplim_{s \to t_j^+}\int_{-R+c(t_j-\tau)}^{R-c(t_j-\tau)}a(t_j+,x)\eta(u(s,x)|\psi(t_j,x))\diff x \\
&-\frac{1}{K}\Biggl(\int_{t_j}^t\sum_{i \in \mathcal{S}(r)}|\psi(r,h_i(r)+)-\psi(r,h_i(r)-)| \\
&\hphantom{-\frac{1}{K}\Biggl(}\times(\lambda_{\alpha_i}^\varphi(\psi(r,h_i(r)+),\psi(r,h_i(r)-))-\dot{h}_i(r))^2 \diff r\Biggr) \\
&+K\nu (t-t_j)+K\nu|\mu_u|(Q_{t_j,t}).
\end{align*}
Finally, let $0 < t_1 < ... < t_J$ be the times of wave interactions before $\tau$. Also denote $t_0=0$ and $t_{J+1}=\tau$. By the convexity of $\eta$, Lemma \ref{stitchinglemma}, and \eqref{atimedecay}, we derive:
\begin{align*}
\int_{-R}^R a(\tau,x)&\eta(u(\tau,x)|\psi(\tau,x))\diff x - \int_{-R-c\tau}^{R+c\tau} a(0,x)\eta(u(0,x)|\psi(0,x))\diff x \\
\leq&\, \aplim_{s \to \tau^+}\int_{-R}^R a(\tau,x)\eta(u(s,x)|\psi(\tau,x))\diff x - \int_{-R-c\tau}^{R+c\tau} a(0,x)\eta(u(0,x)|\psi(0,x))\diff x \\
\leq& \sum_{j=1}^{J+1}\Biggl(\aplim_{s \to t_j+}\int_{-R+c(t_j-\tau)}^{R-c(t_j-\tau)} a(t_j-,x)\eta(u(s,x)|\psi(t_j,x))\diff x \\
&\hphantom{\sum_{j=1}^{J+1}\Biggl(}-\aplim_{s \to t_{j-1}+}\int_{-R+c(t_{j-1}-\tau)}^{R-c(t_{j-1}-\tau)} a(t_{j-1}-,x)\eta(u(s,x)|\psi(t_{j-1},x))\diff x \Biggr) \\
\leq&\sum_{j=1}^{J+1}\Biggl(\aplim_{s \to t_j-}\int_{-R+c(t_j-\tau)}^{R-c(t_j-\tau)} a(t_j-,x)\eta(u(s,x)|\psi(t_j,x))\diff x \\
&\hphantom{\sum_{j=1}^{J+1}\Biggl(}-\aplim_{s \to t_{j-1}+}\int_{-R+c(t_{j-1}-\tau)}^{R-c(t_{j-1}-\tau)} a(t_{j-1}-,x)\eta(u(s,x)|\psi(t_{j-1},x))\diff x \Biggr) \\
\leq& -\frac{1}{K}\Biggl(\int_{0}^{\tau}\sum_{i \in \mathcal{S}(r)}|\psi(r,h_i(r)+)-\psi(r,h_i(r)-)| \\
&\hphantom{-\frac{1}{K}\Biggl(}\times(\lambda_{\alpha_i}^\varphi(\psi(r,h_i(r)+),\psi(r,h_i(r)-))-\dot{h}_i(r))^2\diff r\Biggr) \\
&+K\nu+K\nu|\mu_u|(Q),
\end{align*}
where $Q := \cup_{j=0}^{J-1} Q_{t_j,t_{j+1}}$ is the full cone of information with boundaries given by \eqref{conebdry}.
\par From this, we obtain:
\begin{align}\label{step1estimates1}
\int_{-R}^R a(\tau,x)\eta(u(\tau,x)|\psi(\tau,x))\diff x \leq \int_{-R-c\tau}^{R+c\tau} a(0,x)\eta(u(0,x)|\psi(0,x))\diff x+K\nu+K\nu|\mu_u|(Q),
\end{align}
and 
\begin{equation}\label{step1estimates2}\begin{aligned}
\frac{1}{K}\Biggl(\int_{0}^{\tau}\sum_{i \in \mathcal{S}(r)}&|\psi(r,h_i(r)+)-\psi(r,h_i(r)-)| \times(\lambda_{\alpha_i}^\varphi(\psi(r,h_i(r)+),\psi(r,h_i(r)-))-\dot{h}_i(r))^2\diff r\Biggr)  \\ 
 \leq& \int_{-R-c\tau}^{R+c\tau} a(0,x)\eta(u(0,x)|\psi(0,x))\diff x+K \nu+K\nu|\mu_u|(Q),
\end{aligned}\end{equation}
where $\lambda_{\alpha_i}^\varphi$ is the $\nu$-approximate speed of the shock~\eqref{nushockspeed} and $\alpha_i = 1,2$ is the wave family of the $i^{th}$ front. 

\subsubsection{Step 2: calculation of the amount of shifting}
In this next step, we control the $L^1$-distance between $\psi$ and $v_\nu$, the classical $\nu$-approximate front tracking solution with initial data $u_\nu(0,\cdot)=\psi(0,\cdot)$. From Lemmas \ref{bclemma} and \ref{cfklemma}, we have
\begin{align*}
||v_\nu(\tau,\cdot)-\psi(\tau, \cdot)||_{L^1((-R,R))}  \leq&\, Kd_\nu(u_\nu(\tau,\cdot),\psi(\tau,\cdot)) \\
         \leq&\, K\int_0^\tau  \hspace{.4in}\sum_{\mathclap{\substack{i:\text{shock in $\psi$ at time $r$}}}}|\psi(r,h_i(r)-)-\psi(r,h_i(r)+)|\big|\dot{h}_i-\lambda_{\alpha_i}^\varphi\big|\diff r,
\end{align*}
where $\lambda_{\alpha_i}^\varphi$ is the $\nu$-approximate shock speed~\eqref{nushockspeed}.
Applying Cauchy-Schwarz, using the uniform BV bound on $\psi(t, \cdot)$~\eqref{eq:shifed-bv-bound}, and using \eqref{step1estimates2} gives
\begin{equation}\label{step2estimate}\begin{aligned}
||v_\nu(\tau&,\cdot)-\psi(\tau, \cdot)||_{L^1((-R,R))} \\
\leq&\, K\Bigl[\int_0^\tau  \hspace{.4in}\sum_{\mathclap{\substack{i:\text{shock in $\psi$ at time $r$}}}}|\psi(r,h_i(r)-)-\psi(r,h_i(r)+)| \diff r \Bigr]^\frac{1}{2}  \\
&\,\times\Bigl[\int_0^\tau  \hspace{.4in}\sum_{\mathclap{\substack{i:\text{shock in $\psi$ at time $r$}}}}|\psi(r,h_i(r)-)-\psi(r,h_i(r)+)|\big|\dot{h}_i-\lambda_{\alpha_i}^\varphi(\psi(r,h_i(r)+),\psi(r, h_i(r)-))\big|^2 \diff r \Bigr]^\frac{1}{2} \\
\leq& \,K \tau^{1/2}\Bigl[\int_{-R-c\tau}^{R+c\tau}a(0,x)\eta(u(0,x)|\psi(0,x))\diff x +K \nu+K\nu|\mu_u|(Q)\Bigr]^{\frac{1}{2}}.
\end{aligned}\end{equation}

\subsubsection{Step 3: putting it all together}
The triangle inequality gives
\begin{align*}
||u(\tau,\cdot)-v_\nu(\tau,\cdot)||_{L^1((-R,R))} \leq ||u(\tau,\cdot)-\psi(\tau,\cdot)||_{L^1((-R,R))}+||\psi(\tau,\cdot)-v_\nu(\tau,\cdot)||_{L^1((-R,R))}.
\end{align*}
H\"{o}lder's inequality gives:
\begin{align*}
||u(\tau,\cdot)-v_\nu(\tau,\cdot)||_{L^1((-R,R))} \leq \sqrt{2R}||u(\tau,\cdot)-\psi(\tau,\cdot)||_{L^2((-R,R))}+||\psi(\tau,\cdot)-v_\nu(\tau,\cdot)||_{L^1((-R,R))}.
\end{align*}
Then, from Lemma \ref{relentisl2}, \eqref{step1estimates1}, and \eqref{step2estimate}, we obtain
\begin{align*}
||u(\tau,\cdot)-v_\nu(\tau,\cdot)||_{L^1((-R,R))} \leq K\left(\sqrt{2R}+\sqrt{\tau}\right)\Bigl[||u(0,\cdot)-u_\nu(0,\cdot)||_{L^2((-R-ct, R+ct))}^2+K\nu +K\nu|\mu_u|(Q)\Bigr]^\frac{1}{2},
\end{align*}
where we have used also $u_\nu(0,\cdot)=\psi(0,\cdot)$. 
\par We recall that $\mu_u$ is locally finite and that $Q$ is compact, giving $|\mu_u|(Q) < \infty$. Finally, taking the limit $\nu \to 0$, we obtain a solution $v$ for \eqref{cl} with $v(0,\cdot)=v^0(\cdot)$ and $v_\nu(t,\cdot) \to v(t,\cdot)$ in $L^1$ for all $t$. We additionally obtain the estimate \eqref{L1L2}. Using the interpolation inequality $||f||_{L^2} \leq \sqrt{||f||_{L^1}||f||_{L^\infty}}$ gives \eqref{L2holder}. This completes the proof of \Cref{prop:main}.

\appendix 
\section{A hyperbolic system verifying \textbf{(GL)} which does not verify {Hypotheses 1}}\label{appendix-hypd}
In this section, we present an example of a hyperbolic system verifying \textbf{(GL)} which does not verify {Hypotheses 1} (d) relative to any neighborhood of the origin. We make the standard conventions
\begin{align}
\nabla \lambda_i(u) \cdot r_i(u) &> 0,\hphantom{_{ij}} \indent u \in \mathcal{V}, \label{gnl+} \\
\ell_i(u) \cdot r_j(u)&=\delta_{ij}, \indent i,j=1,2. \label{evorient}
\end{align}
First, we prove a lemma (following \cite{MR4487515}) that provides a useful criterion for whether a system verifies {Hypotheses 1} (d) relative to a set $\mathcal{V}$.
\begin{lemma}\label{lem:obstructiontod}
Let \eqref{cl} verify \textbf{(GL)} and additionally verify {Hypotheses 1} (d) relative to $\mathcal{V} \subset \mathcal{G}$. Then, for any $b \in \mathcal{V}$,
\begin{align*}
\ell_i(b)f''(b)(r_j(b),r_j(b)) \geq 0, \indent \text{ for $i \neq j \in \{1,2\}$}.
\end{align*}
\end{lemma}
\begin{proof}
Let $b \in \mathcal{V}$. Let $\ell_i(b)$ be a left eigenvector of the $i^{th}$ family at $b$. Then, by {Hypotheses 1} (d), $u \mapsto \ell_i(b) f(u)$ is convex or concave on $\mathcal{V}$. Now, let $r_i(u)$ be a right eigenvector of the $i^{th}$ family (not necessarily at $b$). Then, we compute:
\begin{align*}
\ell_i(b)[f'(u)-f'(b)]r_i(u)=\ell_i(b)\lambda_i(u)r_i(u)-\lambda_i(b)\ell_i(b)r_i(u)=\ell_i(b)\cdot r_i(u)[\lambda_i(u)-\lambda_i(b)].
\end{align*} 
Evaluating this at $u=b+\eps r_i(b)$, dividing by $\eps$, and taking the limit $\eps \to 0$, we obtain
\begin{align} \label{eq:GNL-via-SSJC}
\ell_i(b)f''(b)(r_i(b), r_i(b))=(\nabla \lambda_i(b) \cdot r_i(b))(\ell_i(b)\cdot r_i(b)) > 0,
\end{align}
by \eqref{gnl+} and \eqref{evorient}. By {Hypotheses 1} \cref{hyp-d}, this implies that the function $u \mapsto \ell_i(b)\cdot f(u)$ is convex for all $b \in \mathcal{V}$, i.e. the matrix:
\begin{align*}
\ell_i(b)f''(b),
\end{align*}
is positive semi-definite for all $b \in \mathcal{V}$. Finally evaluating at $r_j(b)$:
\begin{align}\label{nonstrictsjc}
\ell_i(b)f''(b)(r_j(b), r_j(b)) \geq 0, \indent i \neq j, \ b \in \mathcal{V}.
\end{align}
as desired.
\end{proof}
\begin{remark} In \cite{MR236527}, Smoller--Johnson constructed global $L^\infty$ solutions for some special systems verifying the \emph{Smoller--Johnson condition} globally, alongside some other exceptional properties. We note that \eqref{nonstrictsjc} is precisely a non-strict form of the Smoller--Johnson condition relative to $\mathcal{V}$.
They additionally note that the Smoller--Johnson condition implies the classical assumption of Glimm--Lax that the interaction of two weak $i$-waves always produces a $j$-rarefaction in their construction of small $L^\infty$ solutions \cite{MR265767}. \end{remark}

\par Now, we use \Cref{lem:obstructiontod} to show that the following system verifies \textbf{(GL)} but does not verify {Hypotheses 1} (d). Consider the system~\eqref{cl} with the flux 
\begin{equation} \label{eq:flux-no-ssjc} f(u,v) = \begin{pmatrix}
    (u-1)^2 + 3uv -v^2\\
    (v+1)^2 + 3uv - u^2
\end{pmatrix}. \end{equation}
At the origin we find 
$$ f'(0,0) = \begin{pmatrix}
    -2 & 0 \\
    0 & 2
\end{pmatrix}.$$
Hence, in a neighborhood of zero the flux is strictly hyperbolic, with $\lambda_1(0) = -2,\ \lambda_2(0) = 2$, and $r_i(0) = \ell_i(0)^t = e_{i}$ for $i =1,\ 2$ where $e_i$ are the standard coordinate vectors.
Computing the second derivative, we find 
$$ \ell_1(0) f''(0) = \begin{pmatrix}
    2 & 3 \\
    3 & -2
\end{pmatrix}.$$
We note that 
$$(\ell_1(0)r_1(0)) (\nabla \lambda_1(0) r_1(0)) = \ell_1(0) f''(0)(r_1(0),r_1(0)) = 2 > 0,$$ 
which holds for all $d$ sufficiently close to the origin due to continuity of $r_1,\ \ell_1,\ f''$, and $\nabla \lambda_1$. 
This shows the system is genuinely nonlinear in the first family by~\eqref{eq:GNL-via-SSJC} (and an identical computation shows the same property in the second family) in a neighborhood of the origin.
\par Finally, we observe $\ell_1(0) f''(0)(r_2(0), r_2(0)) = -2 < 0$. This shows the flux~\eqref{eq:flux-no-ssjc} fails to verify {Hypotheses 1} \ref{hyp-d} relative to any neighborhood of the origin, as if it did this would contradict \Cref{lem:obstructiontod}.

\section{Entropy conservation}\label{appendix-entropy}

In this section we prove \Cref{prop:entropy-conservation}. While our proof is written in the $1$-d setting, the same argument holds on $\T^d$. 

To begin, we recall that $u \in B_{p,\infty}^\alpha(\T^1)$ if 
$$ \norm{u}_{B_{p,\infty}^\alpha(\T^1)} := \norm{u}_{L^p(\T^1)} + \sup_{y \in \T^1} \frac{\norm{u(\cdot) - u(\cdot - y)}_{L^p(\T^1)}}{|y|^{\alpha}}  < \infty. $$
From this definition, we arrive at the following bounds on $u$ and its mollifications $u_\delta := u \ast \gamma_\delta$: 
\begin{align}
    \norm{\partial_x u_\delta}_{L^p(\T^1)} &\leq K \norm{u}_{B^\alpha_{p,\infty}(\T^1)} \delta^{\alpha -1}, \label{eq:besov-1}\\
    \norm{u - u_\delta}_{L^p(\T^1)} & \leq K \norm{u}_{B^\alpha_{p,\infty}(\T^1)} \delta^\alpha,\label{eq:besov-2}\\
    \norm{u(\cdot) - u(\cdot - y)}_{L^p(\T^1)} &\leq K \norm{u}_{B^\alpha_{p,\infty}(\T^1)} |y|^\alpha\label{eq:besov-3}.
\end{align}

The strategy we employ is entirely similar to that of \cite{CET-Onsager}. We will mollify the system~\eqref{cl} to deduce an entropy conservation (up to an error) and use the above bounds to deduce that entropy conservation holds at the limit. 

\noindent 
\textbf{Step 1: Negative time extension.} To begin, we note that give a solution $u \in L^\infty(\R^+\times \T^1)$ we may extend it to negative times by assigning 
$$ u(t,x) := u(0,x) \quad \ \forall t < 0.$$
After this extension, $u$ distributionally solves the equation 
\begin{equation}\label{eq:cl-weak-ext} \partial_t u + \partial_x f(u) = \begin{cases}
    0 & t \geq 0, \\
    \partial_x f(u(0,\cdot)) & t < 0.
\end{cases} \end{equation}

\noindent
\textbf{Step 2: Space-time mollification.}
Next, we take $\gamma$ to be a smooth mollifier; that is, $\gamma \in C^\infty(\R)$, $\int \gamma\,dx = 1$, $\text{supp}(\gamma) \subset [-1,1]$, and $\gamma \geq 0$. 
The mollifier is then 
$$ \gamma_\epsilon(x) := \frac{1}{\epsilon} \gamma\left( \frac{x}{\epsilon} \right).$$
For any fixed $t_0, x_0\in \R$ we define the test function
$$ \Phi(t,x) = \gamma_\epsilon(t_0 - t)\gamma_\delta(x_0 - x).$$
Testing~\eqref{eq:cl-weak-ext} with $\Phi$ yields the following equation evaluated at $(t_0,x_0)$,
\begin{equation} \partial_t u_{\epsilon,\delta} + \partial_x f(u_{\epsilon,\delta}) = \partial_x[ f(u_{\epsilon,\delta}) - f(u)_{\epsilon,\delta}] + \int_{-\infty}^0 \partial_x f(u_0)_\delta \gamma_\epsilon (s-t)\,\diff s \label{eq:space-time-mol} \end{equation}
where $u_{\epsilon,\delta} :=u\ast_t \gamma_\epsilon \ast_x \gamma_\delta$. 
We remark that in the above and in future formulas, $v_{\epsilon,\delta}$ denotes a function mollified in time and space, with parameter $\epsilon$ for time and $\delta$ for space. 

Multiplying~\eqref{eq:space-time-mol} by $\nabla \eta(u_{\epsilon,\delta})$ we use that $u_{\epsilon,\delta}$ is smooth to deduce
\begin{equation}
    \partial_t \eta(u_{\epsilon,\delta}) + \partial_x q(u_{\epsilon,\delta}) = \nabla\eta(u_{\epsilon,\delta})\partial_x[ f(u_{\epsilon,\delta}) - f(u)_{\epsilon,\delta}] + \nabla\eta(u_{\epsilon,\delta})\int_{-\infty}^0 \partial_x f(u_0)_\delta \gamma_\epsilon (s-t)\,\diff s \label{eq:space-time-mol-ent}.
\end{equation}
We now multiply~\eqref{eq:space-time-mol-ent} by a test function of the form $\phi \chi_\epsilon$, where $\phi \in C^\infty_c(\R\times \T^1)$ and $$\chi_\epsilon(t) = \int_{-\infty}^t \gamma_\epsilon(s+\epsilon)\,ds. $$
After integrating in time and space, we find
\begin{equation} \label{eq:test-before-eps-limit}
\begin{aligned}
\int_\R \int_{\T^1} \chi_\epsilon \partial_t \phi \eta(u_{\epsilon,\delta}) + \chi_\epsilon \partial_x \phi q(u_{\epsilon,\delta}) \,\diff x\,\diff t=&\, \int_\R \int_{\T^1} \chi_\epsilon\phi \nabla\eta(u_{\epsilon,\delta})\partial_x[ f(u_{\epsilon,\delta}) - f(u)_{\epsilon,\delta}]\,\diff x\,\diff t \\
    &\,- \int_\R \int_{\T^1} \partial_t \chi_\epsilon \phi\eta(u_{\epsilon,\delta})\,\diff x\,\diff t \\
    &\,+ \int_\R\int_{\T^1}\chi_\epsilon\phi\nabla\eta(u_{\epsilon,\delta})\int_{-\infty}^0 \partial_x f(u_0)_\delta \gamma_\epsilon (s-t)\,\diff s\,\diff x\,\diff t.
    \end{aligned}
\end{equation}

\noindent
\textbf{Step 3: Limits of our mollification parameters.}
We first note that 
$$ - \int_\R \int_{\T^1} \partial_t \chi_\epsilon \phi\eta(u_{\epsilon,\delta})\,\diff x\,\diff t = -\int_{-2\eps}^\eps \int_{\T^1} \gamma_\eps(t) \phi \eta(u_{\eps,\delta}(0,x))\,\diff x\,\diff t \to - \int_{\T^1}  \phi(0,x) \eta(u_{\eps,\delta}(0,x))\,\diff x$$
and observe the final term of~\eqref{eq:test-before-eps-limit} converges to zero, due to the bound
\begin{multline*}
    \int_\R\int_{\T^1}\chi_\epsilon\phi \nabla\eta(u_{\epsilon,\delta})\int_{-\infty}^0 \partial_x f(u_0)_\delta \gamma_\epsilon (s-t)\,\diff s\,\diff x\,\diff t \\\leq \norm{\phi \nabla \eta(u_{\epsilon,\delta})\partial_x f(u_0)_\delta}_{L^\infty(\R\times \T^1)} \int_{-\epsilon}^T \int_{-\infty}^0 \gamma_\epsilon(s-t) \,\diff s\,\diff t \lesssim \epsilon,
\end{multline*}
where $T$ is the largest time for which $\phi$ is supported.
As a result, taking $\epsilon \to 0$ in~\eqref{eq:test-before-eps-limit} we establish
\begin{equation} \label{eq:eps-lim}
\begin{aligned}
    \int_0^\infty \int_{\T^1} \partial_t \phi \eta(u_{\delta}) + \partial_x \phi q(u_{\delta})\,\diff x\,\diff t =&\, \int_0^\infty \int_{\T^1} \phi \nabla \eta(u_{\delta})\partial_x[ f(u_{\delta}) - f(u)_{\delta}]\,\diff x\,\diff t \\
    &\,- \int_{\T^1} \phi(0,\cdot)  \eta(u_{\delta}(0,\cdot))\,\diff x.
\end{aligned}
\end{equation}

Taking the $\delta \to 0$ limit of~\eqref{eq:eps-lim} gives us
\begin{equation} \label{eq:del-lim}
\begin{aligned}
    \int_0^\infty \int_{\T^1} \partial_t \phi \eta(u) + \partial_x \phi q(u)\,\diff x\,\diff t =&\, \lim_{\delta \to 0} \int_0^\infty \int_{\T^1} \phi \nabla \eta(u_{\delta})\partial_x[ f(u_{\delta}) - f(u)_{\delta}]\,\diff x\,\diff t \\
    &\,- \int_{\T^1} \phi(0,\cdot)  \eta(u_{}(0,\cdot))\,\diff x 
\end{aligned}
\end{equation}
To realize the entropy equality \eqref{eq:ent-eq}, we must show the commutator term converges to zero. 
To do this,we recall the following lemma,  
\begin{lemma}[{\cite[Lemma 3.1]{gwiazda2018note}}] \label{lem:commutator}
    Let $u \in L^2_\loc(\T^1; B_{\eps_1}(d))$ and  $f \in C^2(B_{\eps_1}(d)).$ Then there exists a constant $K$ such that 
    $$ \norm{f(u)_\delta - f(u_\delta)}_{L^q(\T^1)} \leq K\left( \norm{u_\delta - u}_{L^{2q}(\T^1)}^2 + \sup_{y \in \mathrm{supp}(\gamma_\delta)} \norm{u(\cdot) - u(\cdot - y)}_{L^{2q}(\T^1)}^2 \right).$$
\end{lemma}
Applying \Cref{lem:commutator} with the estimates~\eqref{eq:besov-1},~\eqref{eq:besov-2}, and~\eqref{eq:besov-3} we see that if $u \in L^3(0,T; B^\alpha_{3,\infty}(\R))$ for $\alpha > 1/3$ then
$$ \int_0^T \norm{\partial_x u_\delta(t,\cdot)}_{L^3(\T^1)} \norm{ f(u_{\delta}(t,\cdot)) - f(u(t,\cdot))_{\delta}}_{L^{3/2}(\T^1)}\,\diff t \leq C\norm{u}_{L^3(0,T; B^\alpha_{3,\infty}(\T^1))}^3\delta^{3\alpha - 1} \to 0, $$
establishing that $u$ locally conserves entropy. 

\section{Approximations of Sobolev data} \label{appendix-sobolev}
In this section we construct sequences of approximations $\{u_\delta\}_{\delta > 0}$ approximating data $u \in W_\loc^{s,p}(\R)$ for $s > 0$ and $p \geq 1$. 
We state estimates on both the variation of the $u_\delta$ and their rate of convergence to $u$. 
Specifically, we fix a mollifier $\gamma \in C^\infty(\R)$ satisfying~\eqref{eq:mol-assumptions}. Defining $\gamma_\delta(x):=\frac{1}{\delta}\gamma(\frac{x}{\delta}), $
we consider the sequence $$u_\delta = u\ast \gamma_\delta,$$
for an arbitrary $u \in W^{s,p}_\loc(\R;B_{\eps_1}(d))$.
\begin{lemma} \label{lem:sob-for-thm-obs}
    Let $\eps_1 > 0$. For any $u \in L^\infty(\R;B_{\eps_1}(d))$ we have 
    \beq \label{eq:TV-bound-Linf} TV(u_\delta; L) \leq K \norm{u}_{L^\infty(\R)} \frac{L}{\delta}, \quad \forall \ L > 0.\eeq
    Further, if $u \in W^{s,p}(\R;B_{\eps_1}(d))$ for $s > 0$ and $p \in (1, \infty]$, then for any fixed interval $(-R,R)$ there exists a constant $K > 0$ such that
    \begin{align}
        \norm{u_\delta}_{BV((a,b))} &\leq K \norm{u}_{W^{s,p}((-R-\delta,R+\delta))} \frac{|b-a|^{1/q}}{\delta^{1-s}}, \quad \forall \ (a,b)\subset (-R,R), \label{eq:TV-frac-L}\\
        \norm{u - u_\delta}_{L^2(( -R, R))} &\leq K  \begin{cases}\norm{u}_{W^{s,p}((-R-\delta,R+\delta))} \delta^s & p\geq 2, \\ \norm{u}_{W^{s,p}((-R-\delta,R+\delta))}^{p/2}\delta^{s p/2} & p < 2,\end{cases}\label{eq:L2-cv-frac-L}
    \end{align}
    where $q \geq 1$ is such that $\frac{1}{p} + \frac{1}{q} = 1.$
\end{lemma}
\begin{proof}
    The bound~\eqref{eq:TV-bound-Linf} follows from  
    $$ TV(u_\epsilon;L) = \sup_{a \in \R} \int_a^{a+L} |u_\delta'(x)| \,\diff x = \sup_{a \in \R} \int_a^{a+L} \left | \int_\R u(y) \frac{1}{\delta^2} \gamma'\left(\frac{x-y}{\delta}\right)\,\diff y \right| \,\diff x \leq K\norm{u}_{L^\infty(\R)}\frac{L}{\delta}. $$

    To deduce~\eqref{eq:TV-frac-L} we compute for $1 < p < \infty$,
    \begin{align*}
        \norm{u_\delta}_{BV((a,b))} &= \int_a^{b} |u_\delta'(x)|\,\diff x \\
        &= \int_a^{b} \left | \int_{-\delta}^\delta u(x-y)\frac{1}{\delta^2} \gamma'\left(\frac{y}{\delta}\right) \,\diff y\right|\,\diff x \\
        &\leq \int_a^b \int_{0}^\delta |u(x-y) - u(x+y)|\frac{1}{\delta^2} \left|\gamma'\left(\frac{y}{\delta}\right) \right|\,\diff y\,\diff x, \\
        \intertext{as $\gamma$ is an even function. By H\"older,}
        &\leq  \frac{1}{\delta^2} \left(\int_a^b \int_{0}^\delta |u(x-y) - u(x+y)|^p \,\diff y\,\diff x\right)^{1/p} \left(\int_a^{b} \int_{0}^\delta \left|\gamma'\left(\frac{y}{\delta}\right)\right|^q \,\diff y\,\diff x\right)^{1/q} \\
        &\leq K  \frac{\delta^{1/p +s }}{\delta^2} \left(\int_a^{b} \int_{0}^\delta \frac{|u(x-y) - u(x+y)|^p}{(2y)^{1+sp}} \,\diff y\,\diff x\right)^{1/p} \left( |b-a| \delta \right)^{1/q} \\
        &\leq K\norm{u}_{W^{s,p}((-R-\delta,R+\delta)) } \frac{|b-a|^{1/q}}{\delta^{1 - s}}.
    \end{align*}
    When $p = \infty$, we note $W^{s,\infty}((-R-\delta,R+\delta)) = C^s((-R-\delta,R+\delta))$ and we similarly have
    \begin{align*}
        \norm{u_\delta}_{BV((a,b))} &= \int_a^b |u_\delta'(x)|\,\diff x \\
        &=  \int_a^{b} \left | \int_{-\delta}^\delta u(x-y)\frac{1}{\delta^2} \gamma'\left(\frac{y}{\delta}\right) \,\diff y\right|\,\diff x \\
        &\leq  \int_a^b \int_{0}^\delta |u(x-y) - u(x+y)|\frac{1}{\delta^2} \left|\gamma'\left(\frac{y}{\delta}\right)\right| \,\diff y\,\diff x \\ 
        &\leq \norm{u}_{C^{s}((-R-\delta,R+\delta))} \delta^s\int_a^{b}\int_0^\delta\frac{1}{\delta^2} \left|\gamma'\left(\frac{y}{\delta}\right)\right| \,\diff y\,\diff x  \\
        &\leq K\norm{u}_{C^{s}((-R-\delta,R+\delta))} \frac{|b-a|}{\delta^{1-s}}.
    \end{align*}

    Next, for $p < \infty $ we compute 
    \beq \label{eq:sob-L2-conv} \begin{aligned}
        \int_{-R}^R |u - u_\delta|^p\,\diff x &= \int_{-R}^R \left| \int_{-\delta}^\delta  (u(x) - u(x-y)\gamma_\delta(y)) \,\diff y\right|^p\,\diff x \\
        &\leq \int_{-R}^{R} \left(\int_{-\delta}^\delta |u(x) - u(x-y)|^p\,\diff y\right)\left(\int_{-\delta}^\delta \gamma_\delta(y)^q\,\diff y\right)^{p/q}\,\diff x \\
        &\leq K \frac{1}{\delta} \int_{-R}^{R} \int_{-\delta}^\delta \delta^{1+ps}\frac{|u(x) - u(x-y)|^p}{|y|^{1+ps}}\,\diff y\,\diff x\\
        &\leq K \delta^{ps} \norm{u}_{W^{s,p}((-R-\delta,R+\delta))}^p,
    \end{aligned}\eeq
    where $q$ is the H\"older conjugate to $p$. 
    For $p = \infty$, the proof follows similarly,
    \beq \label{eq:sob-L2-conv-infty} \begin{aligned}
        \norm{u - u_\delta}_{L^\infty((-R,R)} &= \sup_{x \in (-R,R)} \left| \int_{-\delta}^\delta  (u(x) - u(x-y)\gamma_\delta(y)) \,\diff y\right| \\
        &\leq \sup_{x \in (-R,R)} \norm{u(x) - u(x-\cdot )}_{L^\infty(x-\delta,x+\delta)} \int_{-\delta}^\delta \gamma_\delta(y)\,\diff y \\
        &\leq \norm{u}_{W^{s,\infty}((-R-\delta,R+\delta))} \delta^s.
    \end{aligned}\eeq
    By $L^p$ nesting,~\eqref{eq:sob-L2-conv} and~\eqref{eq:sob-L2-conv-infty} give us
    \beq \norm{u - u_\delta}_{L^2((-R,R))} \leq K \norm{u}_{W^{s,p}((-R-\delta,R+\delta))}\delta^s, \label{eq:prop-L2-p-big}\eeq when $p \geq 2$. 
    For $p < 2$ we interpolate the bound~\eqref{eq:sob-L2-conv} with the $L^\infty$ norm of $u$,
    \beq \label{eq:prop-L2-p-small}
    \begin{aligned}
        \norm{u-u_\delta}_{L^2( (-R,R))} &\leq \left(\norm{u-u_\delta}_{L^p( (-R,R))}^{\frac{p}{2}} \norm{u-u_\delta}_{L^\infty((-R,R))}^{1-\frac{p}{2}}\right) \\
        &\leq K \norm{u}_{W^{s,p}((-R-\delta,R+\delta))}^{\frac{p}{2}}\norm{u}_{L^\infty((-R,R))}^{1-\frac{p}{2}}\delta^{sp/2}.
    \end{aligned}\eeq
\end{proof}

\section{Proof of \Cref{lem:localstructurelemma}}\label{appendixB}
In this section, we prove \Cref{lem:localstructurelemma}. Much of this section follows \cite[Section 6.3]{MR2807139}. We prove only the properties for $1$-shocks; the hypotheses for $2$-shocks follow similarly. The fact that the minimal assumption of genuine nonlinearity is enough to verify most of Hypotheses 1 locally is due to the following lemma on the shock and rarefaction curves.

\begin{lemma}[From \protect{\cite[p.~263-265, Theorem 8.2.1, Theorem 8.3.1, Theorem 8.4.2]{MR3468916}}]\label{lem:dafermoslemma}
    For any fixed state $v\in \mathcal{G}$, there is an open neighborhood $U$ of $v$
    such that for each $i = 1, 2$, there exist functions $s_u\colon U \rightarrow \R$, $\sigma_u^i(s)\colon U\times [0,s_u) \rightarrow \R$ and $S^i_u(s)\colon U\times [0,s_u)\rightarrow \mathcal{G}_0$ satisfying the Rankine-Hugoniot condition, i.e. for any $u\in U$ and any $0 \leq s < s_u$,
    \begin{equation}
    f(S^i_u(s)) - f(u) = \sigma_u^i(s)\left(S^i_u(s) - u\right).
    \end{equation}
    Further, the Lax admissibility criterion is verified, i.e. for $u \in U$ and $0 < s < s_u$,
    \begin{equation}
    \lambda^i(S^i_u(s)) < \sigma_u^i(s) < \lambda^i(u).
    \end{equation}
    Furthermore, $u \mapsto s_u$ is Lipschitz, $(u,s) \mapsto \sigma^i_u(s)$ is $C^3$, and $(u,s) \mapsto S^i_u(s)$ is $C^3$. Finally, $\sigma^i_u(s)$ satisfies the asymptotic expansion,
    \begin{equation}
    \sigma_u^i(s) = \frac{1}{2}\left(\lambda^i(u) + \lambda^i(S^i_u(s))\right) + \mathcal{O}(s^2),
    \end{equation}
    and similarly $S^i_u(s)$ satisfies the asymptotic expansion
    \begin{equation}
    S^i_u(s) = u - r_i(u)s + \frac{s^2}{2} r_i'(u) r_i(u) + \mathcal{O}(s^3).
    \end{equation}
\end{lemma}
Now, we begin the proof of \Cref{lem:localstructurelemma}. Firstly, Hypotheses 1 \ref{hyp-a}, \ref{hyp-b} are satisfied by assumption \textbf{(GL)}. 
\par Fixing a state $d \in \mathcal{G}$ we can construct an entropy with arbitrary diagonal Hessian by following the procedure outlined in \cite[Section 9.3]{serre1999systems}. This amounts to solving Goursat's problem, which produces an entropy agreeing with any specified data along the Riemann invariants emanating from $d$. 
Selecting our specified data to be $C^4$ leads to an entropy $\eta_d$ solving Goursat's problem which is $C^3$. As a result, since $\nabla^2\eta_d(d)$ is positive definite, we can restrict $W$ to be a neighborhood of $d$ in which $\eta_d$ is strictly convex. This shows Hypotheses 1 \ref{hyp-c} is satisfied with the entropy $\eta_d$.   
\par Hypotheses 1 \ref{hyp-e} follows directly from continuity and compactness, taking $W \subset B_1(d)$. Hypotheses 1 \ref{hyp-f} follows directly from \Cref{lem:dafermoslemma}, and an elementary continuity argument. 
\par Regarding Hypothesis 1 \ref{hyp-k}, we compute
\begin{align*}
    \ds{\frac{d}{ds}\eta_d(u_L | S^1_{u_L}(s))}= ((S^1_{u_L})'(s))^t \nabla^2 \eta_d(S^1_{u_L}(s)) (S^1_{u_L}(s)-u_L).
\end{align*}
For $s$ sufficiently small, we have the expansion
\begin{align*}
((S^1_{u_L})'(s))^t \nabla^2 \eta_d(S^1_{u_L}(s))(S^1_{u_L}(s)-u_L)=r_1(u_L)^t \nabla^2 \eta_d(S^1_{u_L}(s))r_1(u_L)s+\mathcal{O}(s^2),
\end{align*}
 by \Cref{lem:dafermoslemma}. This is strictly positive for $s$ sufficiently small as $\nabla^2 \eta_d(S^1_{u_L}(s))$ is positive-definite. So taking $W \subset B_{\eps}(d)$ for $\eps > 0$ sufficiently small, we obtain Hypotheses 1 \ref{hyp-k}. The fact that $\ds{\frac{d}{ds}\sigma^1_{u_L}(s)} < 0$ follows immediately from \Cref{lem:dafermoslemma} and Hypothesis \ref{hyp-b}. 
\par Regarding Hypothesis 1 \ref{hyp-g}, it is well-known that weak shocks valued in $W$ are entropic with respect to the entropy $\eta_d$ if and only if they verify the Lax or Liu entropy condition \cite{MR350216}. Thus, if $(u_L,u_R)$ is a $1$-shock, the Lax entropy condition immediately gives \ref{hyp-g}. So, it remains only to check Hypothesis 1 \ref{hyp-g} for $2$-shocks. Let $(u_L, u_R)$ be a $2$-shock with both states in $W$. Then, \Cref{lem:dafermoslemma} gives
\begin{align*}
\sigma(u_L, u_R)=\lambda_2(u_{i})+\mathcal{O}(|u_L-u_R|), \text{ for $i=L,R$.}
\end{align*}
By strict hyperbolicity and taking $W$ smaller if needed, we see that $\sigma(u_L, u_R) > \max(\lambda_1(u_L), \lambda_1(u_R))$, granting Hypothesis 1 \ref{hyp-g}.
Hypothesis 1 \ref{hyp-i} follows identically to Hypothesis 1 \ref{hyp-g} by symmetry. 
Finally, regarding Hypothesis 1 \ref{hyp-h}, if $(u_L, u_R)$ is an entropic shock with speed $\sigma$ verifying 
\begin{equation*}
    \sigma \leq \lambda_1(u_L),
\end{equation*}
then it must be a $1-$shock by the same argument as above. Thus, $u_R=S^1_{u_L}(s_R)$ for some $s_R > 0$ by the Implicit Function Theorem. 

\bibliographystyle{plain}
\bibliography{refs}

\begin{center} 
\includegraphics[width=.4\linewidth]{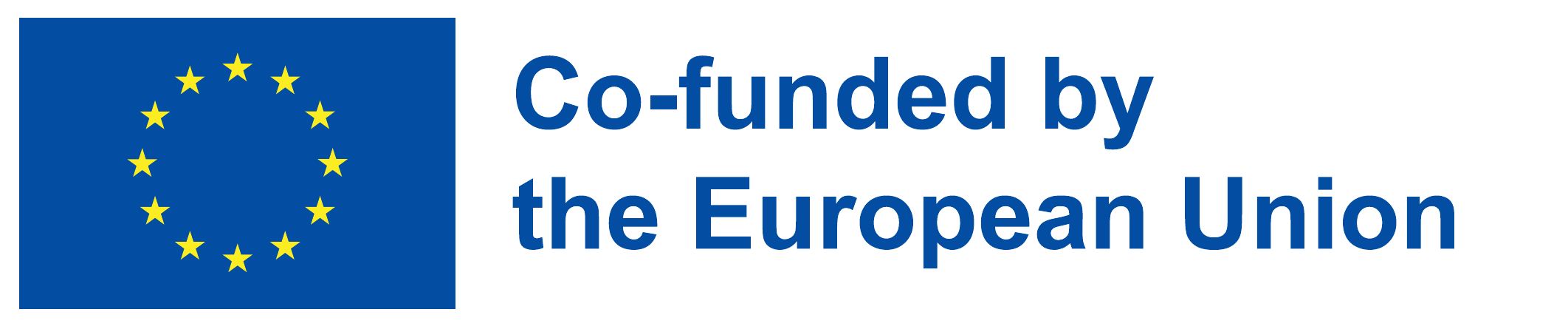}
\end{center}

\end{document}